\documentclass[twoside,11pt]{article}
\usepackage{amssymb,amsmath,amsthm,amsfonts, epsfig}           
\usepackage{mathtools} 
\usepackage{hyperref}
\usepackage{eucal}
\usepackage{a4wide}
\usepackage[svgnames]{xcolor}
\usepackage{mathptmx} 
\usepackage{graphicx}
\usepackage{tikz-cd}
 \usepackage{fancyhdr}
 \usepackage{MnSymbol}
\usepackage{subcaption}
\usepackage[normalem]{ulem}

\title{Platonic solids and symmetric solutions of the $N$-vortex problem on the sphere}

\author{ Carlos~Garc\'ia-Azpeitia and Luis C.~Garc\'ia-Naranjo }

\date{\today}

\numberwithin{equation}{section}
\numberwithin{table}{section}
\numberwithin{figure}{section}

\hypersetup{colorlinks=true,linkcolor=violet,citecolor=purple}

\newtheorem{theorem}{Theorem}[section]
\newtheorem{lemma}[theorem]{Lemma}
\newtheorem{proposition}[theorem]{Proposition}
\newtheorem{corollary}[theorem]{Corollary}

{\theoremstyle{definition}
\newtheorem{definition}[theorem]{Definition}
\newtheorem{remark}[theorem]{Remark}
\newtheorem*{remarks*}{Remarks}

}

\newtheorem{innercustomgeneric}{\customgenericname}
\providecommand{\customgenericname}{}
\newcommand{\newcustomtheorem}[2]{%
  \newenvironment{#1}[1]
  {%
   \renewcommand\customgenericname{#2}%
   \renewcommand\theinnercustomgeneric{##1}%
   \innercustomgeneric
  }
  {\endinnercustomgeneric}
}

\newcustomtheorem{customhyp}{}

\newcommand{\defn}[1]{{\bfseries\itshape{#1}}}

\def\headcolor{\color{Grey}}
\def\restr#1{\,\vrule height1.2ex width.4pt
  depth0.8ex\lower0.4ex\hbox{\scriptsize $\,#1$}}

\newcommand{\R}{\mathbb{R}}

\newcommand{\so}{\mathfrak{so}}

\newcommand{\SO}{\mathrm{SO}}

\begin{document}

\maketitle

\begin{abstract}
We consider the $N$-vortex problem on the sphere assuming that all vortices
have equal strength. We develop a theoretical framework to analyse solutions
of the equations of motion with prescribed symmetries. Our construction relies
on the discrete reduction of the system by twisted subgroups of the full
symmetry group that rotates and permutes the vortices. Our approach formalises and extends
ideas outlined previously by Tokieda (\emph{C. R. Acad. Sci., Paris I} 333
(2001)) and Souli\`{e}re and Tokieda (\emph{J. Fluid Mech.} 460 (2002)) and
allows us to prove the existence of several 1-parameter families of periodic
orbits. These families either emanate from equilibria or converge to
collisions possessing a specific symmetry. Our results are applied to show
existence of families of small nonlinear oscillations emanating from the
platonic solid equilibria.

\end{abstract}

\rightline{\em Dedicated to J. Ize}
\vspace{0.5cm}

\section{Introduction}

In the last three decades, a growing number of publications have considered
the $N$-vortex problem on the sphere. A (necessarily incomplete) list of
references is \cite{Polveni, Kidambi-Newton-integrability,
Borisov-Lebedev, Borisov-Pavlov, Marsden-Pekarsky, LimMontaldi,Tokieda, Laurent-Polz,
SouliereTokieda,Bo03,Bo05,Mo13,Va13,CGA,Wa20}. The first reference to the
problem seems to go back to Zermelo \cite{Zermelo} while the equations of
motion were presented by Gromeka \cite{Gromeka} and Bogomolov
\cite{Bogolmonov}. The importance of the model is usually associated with
geophysical fluid dynamics since it describes the interaction of hurricanes on
the Earth. Interestingly, similar models of vortices are also relevant in the
study of Bose-Einstein condensates \cite{Pe}, while the steady solutions of the
problem have applications in semiconductors \cite{Br} and
reaction-diffusion models \cite{Ward}. We refer the reader to the papers by
Aref et al \cite{ArNe02}, Aref \cite{Aref2007} and the book by
Newton~\cite{Ne01} for a general overview of the $N$-vortex problem and for an
extensive bibliography on the subject.

The investigation of the $N$-vortex problem on the sphere is interesting from
a mathematical point of view since the system is Hamiltonian and invariant
under rotations by the group $\SO(3)$. Symplectic reduction leads to the conclusion that the problem is integrable if $N\leq3$, and also for
$N=4$ if the \emph{centre of vorticity} (momentum map) vanishes
\cite{Borisov-Lebedev, Kidambi-Newton-integrability, Modin-Viviani}. A strong
indication that the problem is non-integrable for $N=4$ for a general centre
of vorticity is given by Bagrets \& Bagrets \cite{Bagrets}, and hence it is
natural to expect that the system is fully chaotic for $N>4$. The dynamics of
the problem in the integrable case $N=3$ was considered in
\cite{Borisov-Lebedev,
Kidambi-Newton-integrability,Kidambi-Newton-collapse,Borisov-Mamaev-Killin-2017}%
.

Much effort has been devoted to the investigation of particular solutions of
the problem for $N\geq4$. Following the pioneering work of Lim, Montaldi \&
Roberts \cite{LimMontaldi}, several publications have considered the
existence, stability and bifurcations of fixed equilibria and relative
equilibria e.g. \cite{Laurent-Polz, Bo03, Mo11}. On the other hand, relative
periodic solutions have been found by Laurent-Polz \cite{Laurent-PolzREPO} and
Garc\'ia-Azpeitia \cite{CGA}. Periodic solutions with prescribed symmetry are
determined in Tokieda \cite{Tokieda} and Souli\`ere \& Tokieda
\cite{SouliereTokieda}. Choreographies were found by Borisov, Mamaev \& Kilin
\cite{Bo05} for $N=4$ and by Garc\'ia-Azpeitia \cite{CGA} for general $N$.
Some of these choreographies were computed numerically by Calleja, Doedel \&
Garc\'ia-Azpeitia \cite{CaDoGa2}.

In this paper we consider the case  in which all vortices have equal strength for
$N\geq4$. With the appropriate normalisations, the governing equations for the
motion of the vortices become
\begin{equation}
\dot{v}_{j}=\sum_{i=1(i\neq j)}^{N}\frac{v_{i}\times v_{j}}{\left\vert
v_{j}-v_{i}\right\vert ^{2}},\qquad j=1,\dots,N,\label{E}%
\end{equation}
where $v_{j}(t)$ belongs to the unit sphere $S^{2}$ in $\R^{3}$ and denotes
the position of the $j^{th}$ vortex and $\times$ denotes the vector product.
The derivation of the equations may be found in Newton's book~\cite{Ne01}. The
assumption that the vortices have equal strength results in the invariance of
the system under the action of the permutation group $S_{N}$ on the vortices  and
 this extra symmetry is essential in our analysis. The Hamiltonian
structure of Eqs. \eqref{E} is described in Section \ref{sec:Prelim} below. An
interesting Lagrangian interpretation of the system is given in Vankerschaver
\& Leok \cite{Va13}.

The most fundamental solutions of Eqs. \eqref{E} are the equilibria, and in
particular, the ground states of the Hamiltonian, but their determination is
difficult for large $N$. Actually, determining the ground state of the
Hamiltonian is a special case of one of Smale's open problems, generalising
the Thomson problem from the Coulomb potential to more general ones. The
ground states for different number of vortices exhibit many symmetries that have been
established rigorously only for $N\leq5$. On the other hand, the five platonic
solids are natural equilibrium solutions for $N=4,6,8,12,20$. It was proved by
Kurakin \cite{K} that the tetrahedron, octahedron and icosahedron are
nonlinearly stable, while the cube and the dodecahedron are unstable.

\subsection*{Goals of the paper}

The original motivation of this paper was to prove the existence of small
nonlinear oscillations near the platonic solid equilibria. This is a
non-trivial task since the Liapunov Centre Theorem \cite{Liapunov} and its
extensions obtained by Weinstein \cite{Weinstein} and Moser \cite{Moser76} do
not apply. The reason is that the equilibria in question are not isolated, but
rather form $\SO(3)$ orbits on the phase space. The symmetric extension 
of these theorems by Montaldi, Roberts \& Stewart \cite{Mo88} does not apply for
the exact same reason. Other extensions relying on topological methods developed by Ize \& Vignoli \cite{Ize} and
Strzelecki  \cite{Strzelecki} also do not apply because they require the phase space to be a euclidean space and the action to be linear.


Our results are summarised below. We succeeded in finding several (but not all) families of
periodic solutions emanating from the platonic solids by developing a general
framework used to analyse solutions of Eqs. \eqref{E} with prescribed
symmetries. Apart from the nonlinear oscillations around the platonic solids,
our construction proves the existence of many other families of periodic
solutions of Eqs. \eqref{E}   which either emanate from an equilibrium or converge to
collisions possessing a specific symmetry.

\subsection*{Summary of results}

Let $\mathcal{P}$ be one of the platonic solids and $N$ be the number of its
vertices. Our approach to prove the existence of periodic solutions of
\eqref{E} in which the vortices oscillate around the vertices of $\mathcal{P}$
is to restrict our attention to the family of solutions that are
$K$-symmetric, where $K<\SO(3)$ is a discrete subgroup that leaves
$\mathcal{P}$ invariant. The crucial point is to choose $K$ in such a way that
the aforementioned family of solutions forms a 1-degree of freedom integrable
Hamiltonian subsystem of \eqref{E}, and hence its generic orbits are periodic.
The idea of the method goes back to Tokieda \cite{Tokieda} and Souli\`ere \&
Tokieda \cite{SouliereTokieda}. A similar   approach is applied by Fusco,
Gronchi \& Negrini \cite{Fu11} to the $N$-body problem.

In this paper we proceed with a degree of generality beyond the case of the
platonic solids described above. Our main contribution is to develop a general
framework, valid for arbitrary $N\geq4$, for the analysis of the $K$-symmetric
solutions of \eqref{E} that form a 1-degree of freedom integrable Hamiltonian
subsystem of \eqref{E}, where $K<\SO(3)$ is one of the following groups: the
dihedral group $\mathbb{D}_{n}$, tetrahedral group $\mathbb{T}$, octahedral
group $\mathbb{O}$ or icosahedral group $\mathbb{I}$.  Our
construction relies on the concept of \defn{$(K,F)$-symmetric solutions}, that
are solutions of Eqs. \eqref{E} of the form
\begin{equation}
\label{eq:KF-symm-sol-intro}(v_{1}(t),\dots, v_{N}(t))= (u(t),g_{2}%
u(t),\dots,g_{m}u(t),f_{m+1},\dots,f_{N}),
\end{equation}
where $K_{o}=(g_{1}=e,g_{2},\dots,g_{m})$ is an ordering of $K$ ($e$ is the
identity element in $\SO(3)$) and $F_{o}=\left(  f_{m+1},...,f_{N}\right)  $
is an ordering of a certain set $F\subset\mathcal{F}[K]$, which is assumed to
be $K$-invariant. Here $\mathcal{F}[K]$ denotes the set of points in $S^{2}$
having non-trivial $K$-isotropy. 

In Theorem \ref{th-main-symmetry} we prove that \eqref{eq:KF-symm-sol-intro}
is a solution of Eqs. \eqref{E} if and only if $u(t)$ is a solution of the
\defn{reduced system}
\begin{equation}
\label{eq:reduced-syst-intro}%
\begin{split}
&  \dot{u} =-\frac{1}{m}u\times\nabla_{u}h_{(K,F)}(u),\\
&  h_{(K,F)}(u) =-\frac{m}{4}\sum_{j=2}^{m}\ln\left\vert u-g_{j}u\right\vert
^{2}-\frac{m}{2}\sum_{j=m+1}^{N}\ln\left\vert u-f_{j}\right\vert ^{2}.
\end{split}
\end{equation}
We call \eqref{eq:reduced-syst-intro} the reduced system since it is obtained
by the discrete reduction (see e.g. Marsden \cite{Marsden-LOM}) of Eqs. \eqref{E} by an appropriate
 twisted subgroup $\hat K< S_{N}\times\SO(3)$ which
is isomorphic to $K$. The smooth function $h_{(K,F)}:S^{2}\setminus
\mathcal{F}[K]\to\R$ is the \defn{reduced Hamiltonian}. In Theorem
\ref{th-main-symmetry} we also specify the symmetries of the reduced system in
terms of the normaliser group $N(K)$ of $K$ in $\SO(3)$ and the invariance
properties of $F$. Furthermore, we show that the centre of vorticity of every
$(K,F)$-symmetric solution vanishes. 

After proving Theorem \ref{th-main-symmetry}, we systematically analyse the
properties of the reduced system \eqref{eq:reduced-syst-intro} and determine
the implications about the corresponding $(K,F)$-symmetric solutions of
\eqref{E}. We first work with general $K$ and $F$. The reduced system
\eqref{eq:reduced-syst-intro} is a smooth, 1-degree of freedom, integrable
Hamiltonian system on $S^{2}\setminus\mathcal{F}[K]$. We show that points in
$\mathcal{F}[K]$ are in one-to-one correspondence with $(K,F)$-symmetric
collisions of Eqs. \eqref{E} and we propose a smooth regularisation of the
reduced system to all of $S^{2}$. This is done in terms of the
\defn{regularised Hamiltonian} $\tilde h_{(K,F)}$ that is the smooth function
on $S^{2}$ given by
\[
\tilde{h}_{(K,F)}(u)=\exp(-2h_{(K,F)}(u)).
\]
The resulting regularised system is a 1-degree of freedom, integrable
Hamiltonian system on the compact manifold $S^{2}$ whose dynamics consists of
equilibrium points, periodic solutions and heteroclinic/homoclinic orbits. In
particular, we conclude that all regular level sets of $\tilde{h}_{(K,F)}$
(and hence also of ${h}_{(K,F)}$) are periodic solutions. This allows us to
prove the existence of 1-parameter families of periodic solutions near the
extrema of ${h}_{(K,F)}$ and the $(K,F)$-symmetric collisions of Eqs.
\eqref{E} (Corollary \ref{cor:periodic-orbits-general}).

We then proceed to analyse the reduced system \eqref{eq:reduced-syst-intro} in
detail for specific choices of $K$ and $F$. Our choices include all
possibilities for which the corresponding $(K,F)$-symmetric solutions contain
the platonic solids as equilibria. In Theorems \ref{th:dihedral} and
\ref{th:dihedral-poles} we classify all the equilibria and collisions for
$K=\mathbb{D}_{n}$ for the cases in which the set $F$ is, respectively, empty
and consists of the north and south poles. The corresponding $(K,F)$-symmetric
equilibria and collisions of Eqs. \eqref{E} are respectively illustrated in
Figures \ref{fig:dihedral-equilibria} and \ref{fig:dihedral-equilibria-poles}.
The equilibria are equatorial polygons, prisms and anti-prisms (with and
without a pair of perpendicular antipodal vortices) for which we give the
explicit dimensions for arbitrary even $N$. The collisions are polygonal (a
binary collision occurring at each vertex) and a polar collision in which half
of the vortices occupy the north and south poles. Using this classification,
and applying our theoretical framework, we establish the existence of the
1-parameter families of periodic orbits emanating from the stable equilibria
of the reduced system and the collisions (Corollaries \ref{cor:dihedral-osc}
and \ref{cor:dihedral-osc-poles}). These periodic solutions are respectively
illustrated in Figures \ref{fig:dihedral-osc} and \ref{fig:dihedral-cube-osc}.
For particular values of $n$, the periodic solutions emanating from the stable
equilibria of the reduced system prove the existence of the following
1-parameter families of periodic solutions of \eqref{E} near the platonic
solid equilibria:

\begin{enumerate}
\item a $\mathbb{D}_{2}$-symmetric family of 4 vortices emanating from the tetrahedron;

\item a $\mathbb{D}_{2}$-symmetric family of 6 vortices emanating from the
octahedron in which two antipodal vortices remain fixed;

\item a $\mathbb{D}_{3}$-symmetric family of 6 vortices emanating from the octahedron;

\item a $\mathbb{D}_{3}$-symmetric family of 8 vortices emanating from the
cube in which two antipodal vortices remain fixed;

\item a $\mathbb{D}_{5}$-symmetric family of 12 vortices emanating from the
icosahedron in which two antipodal vortices remain fixed.
\end{enumerate}

Next we consider the case $K=\mathbb{T}$. Theorems \ref{th:tetrahedral} and
\ref{th:tetrahedral-cube} respectively classify the equilibria and collisions
of the reduced system for $F$ empty and $F$ consisting of two antipodal
tetrahedra that make up a cube. These equilibria and collisions are
respectively illustrated in Figures \ref{fig:n12-eq} and \ref{fig:n20-eq}. Our
theoretical framework applied to this classification proves the existence of
the 1-parameter families of periodic solutions described in Corollaries
\ref{cor:tet-ico-osc} and \ref{cor:tet-ico-osc-cube} and respectively
illustrated in Figures \ref{fig:ico-osc} and \ref{fig:ico-osc-cube}. In
particular we determine the existence of:

\begin{enumerate}
\item[(vi)] a $\mathbb{T}$-symmetric family of periodic solutions of 12
vortices emanating from the icosahedron;

\item[(vii)] a $\mathbb{T}$-symmetric family of periodic solutions of 20
vortices emanating from the dodecahedron in which eight vortices remain fixed
at the vertices of a cube.
\end{enumerate}
We also present the phase portrait of the reduced system
\eqref{eq:reduced-syst-intro} obtained numerically for all choices of $K$ and
$F$ described above that have a platonic solid as a $(K,F)$-symmetric
equilibria. These are given in Figures \ref{fig:Dn-F-empty},
\ref{fig:Dn-F-poles}, \ref{fig:Spheren12T} and \ref{fig:Sphere-n20-T}.

\subsection*{Future work}
\label{sec:future-work}

A natural continuation of this work is to apply the theoretical framework of
Section \ref{sec:symmetric-framework} to different choices of the group $K$
and the $K$-invariant set $F\subset\mathcal{F}[K]$. The cases treated in Sections \ref{sec:Dn}
through \ref{sec:Tet-cube} are only a few possibilities that we chose to work
with because they allowed us to prove the existence of nonlinear small
oscillations around the platonic solids.  It turns out that  for the
subgroups $K=\mathbb{D}_{n},\mathbb{T},\mathbb{O},\mathbb{I}$, the set
$\mathcal{F}[K]$ is finite and  has been classified in \cite[Table
1]{LimMontaldi}. Based upon this classification one concludes that for each
dihedral group $\mathbb{D}_{n}$ and for the tetrahedral group $\mathbb{T}$
there are $6$ distinct choices of $F$, whereas for $\mathbb{O}$ and
icosahedral group $\mathbb{I}$ there are $8$ such possibilities. Some
interesting cases are:

\begin{enumerate}
\item $24$ vortices with octahedral symmetry $\mathbb{O}$ (this case contains
a truncated octahedron as equilibrium).

\item $60$ vortices with icosahedral symmetry $\mathbb{I}$ (this case contains
a truncated icosahedron or Fullerene as equilibrium).
\end{enumerate}

It is also of interest to investigate the persistence of
the periodic, equilibrium and heteroclinic/homoclinic solutions that we found,
and of the invariant sets $M_{(K_{o},F_{o})}$, under
perturbations. Such perturbations will in general destroy the $S_{N}%
\times\SO(3)$ equivariance of the system and the fate of these objects is
unclear. Possible sources of this perturbation may be:

\begin{enumerate}
\item a variation of the strength of some of the vortices. The
\textquotedblleft twisters" of Souli\`{e}re \& Tokieda \cite{SouliereTokieda}
are an indication that persistence may indeed be expected in some cases.

\item a variation of the underlying Riemannian metric on $S^{2}$ as considered
by Boatto \& Koiller \cite{BoKo15}. It was recently found by Wang \cite{Wa20}
that the system has infinitely many periodic orbits.
\end{enumerate}

Another interesting extension of this work is to generalise Theorem
\ref{th-main-symmetry} considering   larger values of $N$ such that
the reduced system is no longer integrable for a subgroup $K<\SO(3)$. For instance, one could look for
$(K,F)$-symmetric solutions generated by two vortices replacing the ansatz
\eqref{eq:KF-symm-sol-intro} with
\[
(v_{1}(t),\dots, v_{N}(t))= (u(t), g_{2}u(t), \dots, g_{m} u(t), w(t),
g_{2}w(t), \dots, g_{m} w(t),f_{2m+1}, \dots, f_{N}),
\]
where, as usual, $K_{o}=\{g_{1}=e,g_{2},\dots g_{m}\}$ and $F_{o}=\{f_{2m+1},
\dots, f_{N}\}$ are orderings of $K$ and $F$. The corresponding reduced system
for $(u(t), w(t))$ is a 2-degree of freedom Hamiltonian system on (an open
dense set of) $S^{2}\times S^{2}$. Although we expect the reduced dynamics to
be non-integrable, one may apply the Lyapunov Centre Theorem or KAM techniques
to prove the existence of periodic and quasi-periodic solutions of the system.
We plan to pursue this research direction in a future publication. Some
interesting cases of the above setup which contain platonic solids as
equilibria are:

\begin{enumerate}
\item $12$ vortices with symmetry $K=\mathbb{D}_{3}$ (the icosahedron is an
equilibrium).

\item $20$ vortices with symmetry $K=\mathbb{D}_{5}$ (the dodecahedron is an
equilibrium).

\end{enumerate}

One could also  apply the techniques  followed by Garc\'ia-Azpeitia \cite{CGA}  to prove the existence of
 relative periodic solutions near the $\SO(3)$-orbit of a platonic solid. In such approach one looks for periodic solutions
 in a rotating frame of reference and performs a stereographic projection. The solutions in question then  correspond to 
 critical points of an $\SO(2)\times S^1$ equivariant gradient map on $\mathbb{R}^{2N}$, where the $\SO(2)$ action is linear.
 Given that $\SO(2)$ is abelian,  one may apply the equivariant degree theory of Ize \& Vignoli  \cite{Ize} to prove the existence
 of a global family of such periodic solutions in the rotating frame  which are the sought relative periodic solutions of the system.  
Alternatively, the local existence of the family of relative periodic solutions may be established using equivariant Conley index as in \cite{Strzelecki} or Poincar\'e maps as in \cite{Doedel}.
 It is important to notice that these solutions have a non-vanishing centre of vorticity and, therefore, in contrast to the solutions found in this 
 paper, do not remain close to a platonic solid configuration but only to its  $\SO(3)$-orbit. Finally, we mention that 
 the  approach of introducing a rotating
 frame of reference is of interest because the resulting equations coincide with those describing the motion of  $N$-vortices on a rotating sphere
 which is a problem with a natural physical relevance. Existence of relative equilibria and quasi-periodic solutions for this system has been respectively
 considered by Laurent-Polz \cite{Laurent-PolzRotSphere} and Newton \& Shokraneh \cite{NeSh}.

\subsection*{Structure of the paper}

We begin by introducing some preliminary material in Section \ref{sec:Prelim}.
All of this material is known except perhaps for Proposition
\ref{prop:complete-flow} that states that, under our hypothesis that all
vortices have equal strength, the system cannot evolve into collision. Our
theoretical framework for the analysis of $(K,F)$-symmetric solutions is
developed in Section \ref{sec:symmetric-framework}. We begin by giving some
basic definitions in Subsection \ref{sec:defns} and then formulate and prove
our main Theorem \ref{th-main-symmetry} on the reduction of the dynamics in
Subsection \ref{sec:reduction}. The regularisation of the collisions is
treated in Subsection \ref{sec:collisions} and the qualitative properties of
$(K,F)$-symmetric solutions is described in Subsection \ref{sec:qualitative}.
In Sections \ref{sec:Dn} through \ref{sec:Tet-cube} we apply the results of
Section \ref{sec:symmetric-framework} to analyse $(K,F)$-symmetric solutions
for specific choices of $K$ and $F$ as described above. Section \ref{sec:Dn}
deals with $K=\mathbb{D}_{n}$ and $F=\emptyset$. Section \ref{sec:Dn-poles}
with $K=\mathbb{D}_{n}$ and $F$ consisting of the north and south poles. In
Section \ref{sec:Tet} we take $K=\mathbb{T}$ and $F=\emptyset$ and in Section
\ref{sec:Tet-cube} we consider $K=\mathbb{T}$ and $F$ consisting of two
antipodal tetrahedra that make up a cube. 

\section{Preliminaries: the equations of motion and their symmetries}
\label{sec:Prelim}

Let $M=S^2 \times \dots \times S^2$, the product of $N$ copies of the unit sphere $S^2$ on $\R^3$. The motion of $N$ vortices on the sphere is described by the Hamiltonian system
on $M$ where the Hamilton function $H$ and symplectic form $\Omega$ are given by
\begin{equation*}
H(v)=-\frac{\Gamma_i \Gamma_j}{4\pi}\sum_{i<j}\ln\left(  \left\vert v_{j}-v_{i}\right\vert^{2}\right), \qquad \Omega= \sum_{i=1}^N \Gamma_j\pi_i^*\omega_{S^2}.
\end{equation*}
Here $v=(v_1,\dots, v_N)\in M$ and $v_j$ is the position of the $j$th vortex whose vorticity is assumed to be $\Gamma_j$, and 
 $\pi_i$ is the Cartesian projection on to the $i$th factor with $\omega_{S^2}$ denoting the usual area form on $S^2$. 

In this work we assume that all vortices have the same  vorticity.   After suitable re-scalings,  the system is  described by the Hamiltonian system on $M$ with 
\begin{equation}
\label{eq:Ham}
H(v)=-\frac{1}{2}\sum_{i<j} \ln\left(  \left\vert v_{j}-v_{i}\right\vert ^{2}\right), \qquad  \Omega= \sum_{i=1}^N \pi_i^*\omega_{S^2}.
\end{equation}
 The corresponding equations of motion take the form
\begin{equation}
 \label{eq:motion}
\dot{v}_{j}=-v_{j}\times\nabla_{v_{j}}H(v)=\sum_{i=1(i\neq j)}^{N}\frac
{v_{i}\times v_{j}}{\left |v_{j}-v_{i} \right | ^{2}}, \qquad j=1,\dots, N,
\end{equation}
where here, and throughout `$\times$' denotes the vector product in $\R^3$.
One may check that the above equations  indeed define a vector field $X$ on $M$ as follows: consider them as a system on $(\R^3)^N$ and notice that $|v_j|^2$ is a first  integral for $j=1,\dots , N$. Then the
equations may be restricted to the level set $M$ where all these integrals take the value 1. The vector field $X$ satisfies $\Omega(X,\cdot ) =dH$.

\paragraph{Collisions}  Both $H$ and the equations of motion are undefined at the {\em collision set } $\Delta \subset M$ where at least two vortices occupy the same position, i.e.
\begin{equation*}
\Delta=\{ (v_1, \dots, v_n)\in M \, : \, v_i=v_j \quad \mbox{for some} \quad i\neq j \}.
\end{equation*}
It is usual to remove these points from the phase space to work with smooth objects. In our approach we will often find it convenient not to do this (in fact we work
with a regularisation of the equations of motion ahead). In any case, it is convenient to have in mind
 that  any collision-free configuration $v\in M\setminus \Delta$ {\em cannot} evolve into a collision  as we show in the following proposition.

\begin{proposition} 
\label{prop:complete-flow}
The flow of \eqref{eq:motion} on $M\setminus \Delta$ is complete. Namely, 
if $v_0\in M\setminus \Delta$ and  $t\mapsto v(t)$ denotes the solution of \eqref{eq:motion} with initial condition $v_0$, then $v(t)$ is defined for all time $t$. (In particular  $v(t)\nin \Delta$ for all
$t\in \R$.)
\end{proposition}
\begin{proof}
Since the sphere $S^2$ is a bounded set, for any $i,j$ we have $|v_i-v_j|\leq 2$ , so $-\ln |v_i-v_j|^2$ is    bounded from below.  As a consequence, 
if $\{v^{(k)}\}_{k\in \mathbb{N}}$ is a sequence in $M\setminus \Delta$ with $v^{(k)}\to \Delta$ as $k\to \infty$, then necessarily $H(v^{(k)})\to \infty$. Considering that $H(v(t))= H(v_0)<\infty$  we conclude that $v(t)$ stays away from $\Delta$ at all time at which it is defined.  However, since $M\setminus \Delta$ is a bounded set, standard theorems on extensibility of
 of  solutions of differential equations imply that $v(t)$ can only cease to exist if it approaches the boundary of $M\setminus \Delta$. But this boundary is precisely $\Delta$. Therefore $v(t)$ is defined for all time
 $t\in \R$.
\end{proof}

\begin{remark} Note that if the vortex strengths  are not identical and  have different signs, collisions may indeed occur in finite time \cite{Kidambi-Newton-integrability,Kidambi-Newton-collapse}. 
On the other hand, the above property of the $N$ vortex problem on the sphere
 is a fundamental difference with the $N$-body problem on the sphere \cite{Borisov-2bodySphere} where collisions may indeed take place. 
 This is due to dependence of the Hamiltonian on the velocities in the latter problem.
When going to collision,  the  kinetic energy approaches infinity and the potential energy approaches minus infinity while their sum remains constant.
\end{remark}

\subsection{Symmetries}

\paragraph{Rotational symmetries.} The group $\SO(3)$ acts diagonally on $M=S^2\times \dots \times S^2$ and  is easy to check that the action is symplectic and the Hamiltonian $H$ is invariant.  As a consequence, the 
equations of motion \eqref{eq:Ham} are  $\SO(3)$-equivariant. Moreover, Noether's theorem applies and  we have a conservation law: the quantity
\begin{equation*}
J:M\to \R^3, \qquad J(v_1, \dots , v_N)= \sum_{i=1}^Nv_i,
\end{equation*}
is constant along the motion. This statement may be verified directly from the equations of motion \eqref{eq:motion}. In geometric terms, $J$ is the momentum map of the  $\SO(3)$ action on $M$, 
with the usual identification of $\so(3)^*$ with $\R^3$. We will refer to $J(v)$  as
the {\em centre of vorticity} of the configuration $v=(v_1, \dots, v_N)$.

\paragraph{Vortex relabelling symmetry.} Since all the vortices  have the same strength, the system is also invariant under relabelling of the vortices. This may be represented by the action of the permutation group $S_N$ on 
$M$,
\begin{equation*}
 \sigma : (v_1, \dots, v_N) \mapsto (v_{\sigma^{-1}(1)},\dots,v_{\sigma^{-1}(N)}), \qquad \mbox{for} \quad \sigma \in S_N.
\end{equation*}
One may  check that this action is symplectic, that the  Hamiltonian $H$ is invariant and the equations
of motion  \eqref{eq:Ham} are $S_N$-equivariant.
\begin{remark}
\label{rmk:product-in-SN}
In our convention, the product of two permutations $\sigma_1, \sigma_2\in S_N$ is $\sigma_1\sigma_2:=\sigma_1\circ \sigma_2$.   The action above is defined with  $\sigma^{-1}$ in order to 
have a {\em left} action with respect to this product.  Other  papers on the subject
 (e.g. \cite{CGA} and references therein) consider instead the  action 
$\sigma:(v_{1},\dots,v_{N})\mapsto(v_{\sigma(1)},\dots,v_{\sigma(N)})$ which is a  left action only if the product on  $S_{N}$ is  defined according to the opposite convention,
 $\sigma_{1}\sigma_{2}:=\sigma_{2}\circ\sigma_{1}$.
\end{remark}

\paragraph{Full symmetries.} The two symmetries described above may be combined into a symplectic action of the direct product group $\hat G:=S_N\times \SO(3)$ on $M$. Throughout the paper this action will be denoted by
a {\em centre dot}  `$\cdot$' as follows:
\begin{equation*}
( \sigma,g)\cdot  (v_1, \dots, v_N)= (gv_{\sigma^{-1}(1)},\dots,gv_{\sigma^{-1}(N)}), \qquad (\sigma,g) \in \hat G, \quad (v_1,\dots, v_N)\in M.
\end{equation*}
Apart from the simplecticity of this action,  the  Hamiltonian $H$ is invariant and the equations
of motion  \eqref{eq:motion} are $\hat G$-equivariant. A key feature of this action is that it is not free and this will allow us to extract valuable information about the dynamics of   \eqref{eq:motion}.

\paragraph{Twisted subgroups. }
Suppose that $K<\SO(3)$ is a discrete subgroup and 
$\tau:K\to S_N$ is a group morphism. Then 
\begin{equation}
\label{eq:twisted}
\hat K_\tau:=\{(\tau(g),g)\, : \, g\in K\}
\end{equation}
 is a discrete subgroup of  $\hat G=S_N\times \SO(3)$ which is often called a \defn{twisted subgroup}. 
  In this work a special role is played by the twisted subgroups of $\hat G$ corresponding to  one-to-one group morphisms $\tau$.

\section{Symmetric solutions of the $N$-vortex problem on the sphere}
\label{sec:symmetric-framework}

\subsection{Symmetric configurations: definitions}
\label{sec:defns}

Let $K< \SO(3)$ be any subgroup. Then $K$ acts on $S^{2}$ as usual. The set of points in $S^2$ having non-trivial $K$-isotropy will be 
 denoted as
\[
\mathcal{F}[K]=\left\{  u\in S^{2}:K_{u}\neq\{e\}\right\}  .
\]
Here, and in what follows, $e$ denotes the identity element in $K$ and $K_u$ is the isotropy group of $u\in S^2$, namely, $K_u=\{g\in K \, : \, gu=u\}$.

\begin{remark}
\label{prop:finiteF}
For the (finite) subgroups $K=\mathbb{Z}_{n},\mathbb{D}_{n},\mathbb{T},\mathbb{O},\mathbb{I}$ the set
$\mathcal{F}[K]$ is finite. Actually, the sets $\mathcal{F}[K]$ are completely classified in Table 1 and the appendix of \cite{LimMontaldi} (that also gives a brief description of these groups).
\end{remark}

The following definitions are essential in our work.

\begin{definition}
\label{def:sphere-symmetric}
Let $K<\SO(3)$ be a discrete subgroup of order $m\leq N$ and $F\subset
\mathcal{F}[K]$ be a $K$-invariant subset of order $N-m$. Let 
$K_o=(g_{1}=e,g_{2},\dots,g_{m})$ be an ordering of  $K$ and $F_o=\left(
f_{m+1},...,f_{N}\right)$  an ordering of  $F$. We define $M_{(K_o,F_o)}\subset M$ as  the set of configurations $(v_{1},\dots,v_{N})\in M$ that satisfy 
\begin{align*}
v_{j}  &  =g_{j}v_{1},\qquad j=1,\dots,m,\\
v_{j}  &  =f_{j},\qquad j=m+1,\dots,N.
\end{align*}
\end{definition}

\begin{definition} 
\label{def:symmetric}
Let $K$ and $F$ be as  in Definition \ref{def:sphere-symmetric}, we will say that a configuration $v\in M$ is
\defn{$(K,F)$-symmetric}  if    $v\in M_{(K_o,F_o)}$ for certain orderings $K_o$ and $F_o$. If $F$ is empty and $m=N$ we will simply
say that the corresponding configuration is \defn{$K$-symmetric}.
\end{definition}

\begin{remark}
The above definitions require that the first $m$ entries of a $(K,F)$-symmetric configuration to be described  in terms of the   elements of $K$ and the remaining $N-m$ in terms
of the elements of $F$. This constraint in the ordering is artificial and could be removed  in view of the relabelling symmetry, but we keep it for clarity of the presentation. The same observation 
holds for our requirement that $g_1=e$.
\end{remark}

For the rest of the section, the symbol   
$K$ will always denote one of the groups $\mathbb{D}_{n},\mathbb{T},\mathbb{O},\mathbb{I}<\SO(3)$, and the symbol $F$ will always 
denote a $K$-invariant subset of $\mathcal{F}[K]$. Moreover, we will continue to denote  $m=|K|>0$ and  $N-m=|F|\geq 0$.

\subsection{Reduction  of  the dynamics of symmetric configurations}
\label{sec:reduction}

Suppose that  $v=(v_1, \dots, v_N)$ is a  $(K,F)$-symmetric configuration so that  $v\in M_{(K_o,F_o)}$ for certain orderings $K_o=(g_1=e,\dots, g_m)$ of $K$  and $F_o=(f_{m+1},\dots, f_N)$ of $F$. Let 
\begin{equation}
\label{eq:embedding}
\rho_{(K_o,F_o)}:S^2\to  M, \qquad \rho_{(K_o,F_o)}(u)=(u,g_{2}u,\dots,g_{m}u, f_{m+1},\dots, f_N).
\end{equation}
 In this section we will prove that the solution of equations \eqref{eq:motion} with initial condition $v$ is given by
$
t\mapsto \rho_{(K_o,F_o)}(u(t)),
$
where $u(t)$ is the solution to the \defn{reduced system} on $S^2$, 
\begin{equation}
\label{eq:reduced-system}
\dot u = -\frac{1}{m} u \times \nabla_u h_{(K,F)}(u),
\end{equation}
with initial condition $u(0)=v_1$. Here,  $h_{(K,F)}:S^2\to \R$ is  the \defn{ reduced Hamiltonian}, that  is defined in terms of the Hamiltonian \eqref{eq:Ham} by
\begin{equation}
\label{eq:red-Ham}
h_{(K,F)}(u):= H(\rho_{(K_o,F_o)}(u)).
\end{equation}
In particular, this shows that the evolution of a  $(K,F)$-symmetric configuration remains a $(K,F)$-symmetric configuration at all time, and, therefore, we may speak of \defn{ $(K,F)$-symmetric solutions}.
Note that along these solutions, the vortices located at $f_{m+1}, \dots, f_N$ remain fixed.

We will also show that the reduced system~\eqref{eq:reduced-system} possesses  a symmetry, and we will describe it in detail. Note that, in virtue of the invariance of $H$ under the relabelling of the vortices,
 the reduced Hamiltonian  $h_{(K,F)}$ is well defined    independently of the
specific orderings $K_o$ and $F_o$ of $K$ and $F$, and this  is reflected in our notation.

The properties described above follow from the first three items of the  following theorem  whose 
 proof relies on the concept of discrete reduction (see e.g. \cite{Marsden-LOM}). Applications of 
discrete reduction  to the study of the $N$-vortex problem on the sphere  already appear in \cite{LimMontaldi}.

\begin{theorem}
\label{th-main-symmetry}
Let $K$ be any of the groups  $\mathbb{D}_{n},\mathbb{T},\mathbb{O},\mathbb{I}$ and let  $F\subset \mathcal{F}[K]$ be a $K$-invariant set. Suppose that $|K|=m\leq N$ and $|F|=N-m\geq 0$. 
Let $K_o=(g_1=e,g_2,\dots, g_m)$ and $F_o=(f_{m+1},\dots, f_N)$ be orderings of $K$ and $F$. The following statements hold. 
\begin{enumerate}
\item The set $M_{(K_o,F_o)}$ is an embedded submanifold of $M$, diffeomorphic to $S^{2}$,  and
invariant under the flow of the equations of motion \eqref{eq:motion}.

\item The restriction of the flow of  \eqref{eq:motion} to $M_{(K_o,F_o)}$  is conjugated by $ \rho_{(K_o,F_o)}$ to the flow of
the integrable Hamiltonian system on $(S^2 , m \omega_{S^2})$, with (reduced) Hamiltonian $h_{(K,F)}:S^2\to \R$ defined by \eqref{eq:red-Ham}. That is, $t\to u(t)$ is a solution
of  \eqref{eq:reduced-system} if and only if $t\mapsto \rho_{(K_o,F_o)}(u(t))$ is a solution of \eqref{eq:motion}.

\item Let $N(K)$ be the  normaliser of $K$ in $\SO(3)$ and suppose that the subgroup   $ K_1<\SO(3)$ satisfies $K\leq  K_1 \leq N(K)$. If 
  $F$ is invariant with respect to the $K_1$-action on $S^2$, then the reduced Hamiltonian $h_{(K,F)}$ is $K_1$-invariant and the reduced system
    \eqref{eq:reduced-system} is $K_1$-equivariant. In particular, these conclusions always hold for $K_1=N(K)$ if $F=\emptyset$ and  for $ K_1=K$ for general $F$.

\item Up to the addition of a constant term, the reduced Hamiltonian $h_{(K,F)}:S^2\to \R$ satisfies
\begin{equation}
\label{eq:redHam-mainthm}
h_{(K,F)}(u)=-\frac{m}{4} \sum_{j=2}^m \ln \left \vert u-g_j u \right \vert^2 -\frac{m}{2} \sum_{j=m+1}^N\ln \left  \vert u-f_j \right \vert ^2 .
\end{equation}

\item The centre of vorticity of elements in $M_{(K_o,F_o)}$ is $0\in \R^3$, i.e.  $M_{(K_o,F_o)}\subset J^{-1}(0)$.

\end{enumerate}
\end{theorem}

\begin{remark}
\label{rmks:theorem-symmetry}
In trying to understand which are the symmetries of the reduced system \eqref{eq:reduced-system} it will be useful to keep in mind the following relations between the 
groups $\mathbb{D}_{n},\mathbb{T},\mathbb{O}$ and $\mathbb{I}$, and their normalisers in $\SO(3)$:
\begin{equation}
\label{eq:table-normalizers}
\begin{array}
[c]{|c|c|c|c|c|c|}%
\hline
K  & \mathbb{D}_{2} & \mathbb{D}_{n}, \, n\geq 3 & \mathbb{T} & \mathbb{O} &  \mathbb{I} \\ \hline
 N(K) & \mathbb{O} &  \mathbb{D}_{2n} &  \mathbb{O }& \mathbb{O} &  \mathbb{I} \\ \hline
\end{array}
\end{equation}
\end{remark}

\begin{remark} 
\label{rmks:theorem-symmetry2}
 To be precise, at this point of the paper, all statements about the flow of  \eqref{eq:motion} in items (i)-(v) of the theorem  only make sense away from collisions.  In fact,
the reduced system \eqref{eq:reduced-system} is only defined at those points of $S^2$ at which $h_{(K,F)}$ is smooth.
In Section~\ref{sec:collisions} ahead will show that the reduced system is well defined away from (finitely many) points in $\mathcal{F}[K]$ which are in one-to-one correspondence with the  collision
configurations within $M_{(K_o,F_o)}$. Moreover, we will 
 introduce a  regularisation  that extends the reduced system  \eqref{eq:reduced-system} to all of  $S^2$,
  the flow of    \eqref{eq:motion}  to all of  $M_{(K_o,F_o)}$, and  the   conclusions of the theorem are valid for this regularisation  without any restriction. We have decided to 
  oversee this detail in the statement of the theorem and in its proof to simplify the presentation.
\end{remark}
The proof of the theorem that we present relies on the following three lemmas whose proof is postponed until the end of the section.
\begin{lemma}
\label{lemma:exist-twisted-subgroup}
Let $K$, $F$, $K_o$ and $F_o$ be as  in the statement of Theorem~\ref{th-main-symmetry}. There exists a one-to-one group morphism $\tau:K\rightarrow S_{N}$ such that
$M_{(K_o,F_o)}$ is a connected component of  $\mbox{\em Fix}(\hat{K_\tau})\subset M$ where $\hat{K_\tau}<\hat G$ is the twisted subgroup  \eqref{eq:twisted}, and where
\begin{equation*}
\mbox{\em Fix}(\hat{K_\tau}) :=\{ v\in M \, : \, \hat g\cdot v =v \quad \mbox{for all} \quad \hat g\in \hat{K_\tau} \}.
\end{equation*}
\end{lemma}

\begin{lemma}
\label{lemma:pull-back}
Let $K$, $F$, $K_o$ and $F_o$ be as  in the statement of Theorem~\ref{th-main-symmetry}, then $\rho_{(K_o,F_o)}^*\Omega =m\omega_{S^2}$.
\end{lemma}

\begin{lemma}
\label{lemma:sumK}
 For the groups $K=\mathbb{D}_{n},\mathbb{T},\mathbb{O}, \mathbb{I}<\SO(3)$, we have $\sum_{g\in K}g=0$.
\end{lemma}

\begin{proof}[Proof of Theorem~\ref{th-main-symmetry}]

(i) The set $M_{(K_o,F_o)}$ is clearly an embedded submanifold of $M$ isomorphic to $S^2$ with the embedding given by \eqref{eq:embedding}. Indeed, we have $M_{(K_o,F_o)}=\rho_{(K_o,F_o)}(S^2)$. The invariance of $M_{(K_o,F_o)}$  under the flow of \eqref{eq:motion} is immediate in virtue of  Lemma~\ref{lemma:exist-twisted-subgroup}: since the system  \eqref{eq:motion} is $\hat G$-equivariant then $\mbox{ Fix}(\hat{K_\tau})$ is invariant by its flow and so are each of its connected components.

(ii) First note that for $\varphi: S^2\to \R$, the associated Hamiltonian vector field  $X_\varphi$  on $S^2$, determined by $\omega_{S^2}(X_\varphi,\cdot )=d\varphi$, defines the equations of motion 
$\dot u = -u\times \nabla_u\varphi(u)$. If the symplectic form $\omega_{S^2}$ is scaled by a factor of $m$, then the corresponding Hamiltonian vector field $X_\varphi$ inherits a rescaling by $1/m$ which
leads to an appearance of this factor on the right hand side of the equations of motion. This shows that
the Hamiltonian system on $(S^2,m\omega_{S^2})$ with Hamiltonian $h_{(K,F)}$ defines the equations \eqref{eq:reduced-system}, as required. Moreover, 
this system is trivially integrable in the Arnold-Liouville sense since  $S^2$ has dimension 2 and $h_{(K,F)}$  is a first integral.

Next,  since $\hat G$ acts symplectically on $(M,\Omega)$ and $H$ is $\hat G$-invariant,  it is known (e.g. \cite{Marsden-LOM})
  that $\mbox{Fix}(\hat{K_\tau})$ is a symplectic submanifold of $M$ and that the restriction of the  flow of $X$  to $\mbox{ Fix}(\hat{K_\tau})$ is Hamiltonian with respect to the restricted Hamiltonian and symplectic form. The
same is  true about each of its connected components. In particular, in view of Lemma \ref{lemma:exist-twisted-subgroup},
 this implies that $ M_{(K_o,F_o)}$ is a symplectic manifold equipped with the restriction $\Omega_0:=\left . \Omega \right   |_{M_{(K_o,F_o)}}$ 
 of the symplectic form $\Omega$, and that the restriction of the flow of \eqref{eq:motion} 
 to $ M_{(K_o,F_o)}$ is Hamiltonian with respect to $\Omega_0$ and the Hamilton function $H_0:= \left . H \right   |_{M_{(K_o,F_o)}}$.

The key point of the proof is to observe that  $\rho_{(K_o,F_o)}$ defined by \eqref{eq:embedding} is in fact a symplectomorphism between $(S^2, m\omega_{S^2})$ and 
$(M_{(K_o,F_o)},\Omega_0)$. This is an immediate consequence of Lemma~\ref{lemma:pull-back} together with the observation that $M_{(K_o,F_o)}=\rho_{(K_o,F_o)}(S^2)$.
As any symplectomorphism,
$\rho_{(K_o,F_o)}$ takes Hamiltonian vector fields into Hamiltonian vector fields (see e.g.~\cite{MaRa}). Considering that the reduced Hamiltonian \eqref{eq:red-Ham} and the restricted Hamiltonian $H_0$  
 are related by $h_{(K,F)} = H_0  \circ \rho_{(K_o,F_o)}=  \rho_{(K_o,F_o)}^*H_0$, it follows that
$\rho_{(K_o,F_o)}$ pulls back  the vector field  on  $ M_{(K_o,F_o)}$ defined by the  restriction of  \eqref{eq:motion} onto the vector field on $S^2$ defined by  \eqref{eq:reduced-system}. In particular, $\rho_{(K_o,F_o)}$
maps solutions of  \eqref{eq:reduced-system} into solutions of  \eqref{eq:motion} that are contained in $ M_{(K_o,F_o)}$. This correspondence between solutions is one-to-one since
 $\rho_{(K_o,F_o)}:S^2\to M_{(K_o,F_o)}$ is invertible.

(iv)  By definition of $h_{(K,F)}$ we have
\begin{equation}
\label{eq:aux1-proposition-simp-red-Ham}
h_{(K,F)}(u)=-\frac{1}{2}\sum_{1\leq i<j\leq m}\ln \left   \vert g_iu-g_j u \right \vert^2  - \frac{1}{2}\sum_{j=m+1}^N\sum_{i=1}^m \ln \left \vert g_iu-f_j \right \vert^2 
-\frac{1}{2}\sum_{m+1\leq i<j\leq N}\ln \left   \vert f_i - f_j \right \vert^2.
\end{equation}
Now, on  the one hand we have
\begin{equation}
\label{eq:aux2-proposition-simp-red-Ham}
\sum_{1\leq i<j\leq m} \ln \left \vert g_i u-g_ju\right \vert^2 = \frac{1}{2} \sum^m_{\substack{i, j=1 \\ i\neq j}} \ln \left \vert g_i u-g_ju\right \vert ^2
=  \frac{1}{2} \sum^m_{\substack{i, j=1 \\ i\neq j}} \ln \left \vert  u-g_i^{-1}g_ju\right \vert^2  = \frac{m}{2} \sum_{j=2}^m \ln \left \vert  u-g_ju\right \vert^2 .
\end{equation}
On the other hand,  fix $f \in F$ and let $I \subset \{ m+1,\dots, N\}$ be such that the $K$-orbit of $f$ satisfies 
$Kf=\left\{  f_{i}\,:\,i\in I\right\}  $.  For each  $i\in I$ let $h_{i}\in K$  such that  $h_if=f_i$.
Since the orbit $Kf$ is isomorphic
to $K/K_{f}$ we have $Kf=\left\{  h_{i}f \, :\, i\in I\right\}  $
 and $K=\cup_{i\in I}h_{i}K_{f}$. Thus 
\begin{equation*}
\sum_{i=1}^m \ln \left \vert g_iu-f \right \vert^2 = \sum_{g\in K} \ln \left \vert u-gf \right \vert^2= \sum_{i\in I}\sum_{h\in K_{f}} \ln \left \vert u-h_ihf \right \vert^2=
\vert K_{f} \vert \sum_{i\in I} \ln \left \vert u-f_i \right \vert^2.
\end{equation*}
Considering that the above formula holds when setting $f=f_j$ for all $j\in I$, and that  $|K_{f_j}|=m /\vert I\vert $ is constant for  $j\in I$, we have
\begin{equation*}
\sum_{j\in I}\sum_{i=1}^m \ln \left \vert g_iu-f_j \right \vert^2 = \vert I \vert 
\frac{m}{\vert I \vert} \sum_{i\in I} \ln \left \vert u-f_i \right \vert^2 =m  \sum_{j \in I} \ln \left \vert u-f_j \right \vert^2.
\end{equation*}
Therefore, breaking up the set of indices $\{m+1, \dots, N\}$ into the disjoint subsets $I_k$, each containing the indices of a $K$ orbit of $F$, we have
\begin{equation}
\label{eq:aux3-proposition-simp-red-Ham}
\sum_{j=m+1}^N\sum_{i=1}^m \ln \left \vert g_iu-f_j \right \vert^2 =\sum_k \sum_{j\in I_k} \sum_{i=1}^m \ln \left \vert g_iu-f_j \right \vert^2 =m \sum_k \sum_{j\in I_k}\ln \left \vert
 u-f_j \right \vert^2 = m\sum_{j=m+1}^N\ln \left \vert
 u-f_j \right \vert^2.
\end{equation}
Substituting \eqref{eq:aux2-proposition-simp-red-Ham} and \eqref{eq:aux3-proposition-simp-red-Ham} into \eqref{eq:aux1-proposition-simp-red-Ham} yields 
\eqref{eq:redHam-mainthm} since the third sum in  \eqref{eq:aux1-proposition-simp-red-Ham}
 is a constant independent of $u$.

(iii)  Let $g\in K_{1}$. Starting from  \eqref{eq:redHam-mainthm} and using the
$\SO(3)$-invariance of the euclidean norm we have%
\[
h_{(K,F)}(gu)=-\frac{m}{4}\sum_{j=2}^{m}\ln\left\vert u-g^{-1}g_{j}%
gu\right\vert ^{2}-\frac{m}{2}\sum_{j=m+1}^{N}\ln\left\vert u-g^{-1}%
f_{j}\right\vert ^{2}.
\]
Now, for any $g\in K_{1}\leq N(K)$ the map $k\in K\mapsto g^{-1}kg\in K$ is
bijective, and hence $(g_{1}^{\prime}=e,g_{2}^{\prime},\dots,g_{m}^{\prime})$
with $g_{j}^{\prime}=g^{-1}g_{j}g$ is a new ordering of $K$. Moreover, since,
by hypothesis, $F$ is $K_{1}$-invariant, then $(f_{m+1}^{\prime},\dots
,f_{N}^{\prime})$ with $f_{j}^{\prime}=g^{-1}f_{j}$ is a new ordering of $F$.
Therefore $h_{(K,F)}(gu)=h_{(K,F)}(u)$ showing that  $h_{(K,F)}$ is indeed $K_1$-invariant. Since the $K_1$ action on $(S^2, m\omega_{S^2})$ is symplectic,
it follows that the flow of the reduced system \eqref{eq:reduced-system} is $K_1$-equivariant.

 (v)  In view of Lemma~\ref{lemma:sumK}, if $v=(v_1,\dots v_N)\in M_{(K_o,F_o)}$, we have
\[
J(v)=\sum_{j=1}^{N}v_{j}=\sum_{j=1}^{m}g_{j}v_{1}+\sum_{j=m+1}^{N}f_{j} =\left ( \sum_{g\in K}g \right ) v_1 +\sum_{j=m+1}^{N}f_{j} =\sum_{j=m+1}^{N}f_{j} .
\]
To show that the remaining sum on the right also vanishes, we rely on the $K$-invariance of $F$.
 Proceeding as in the proof of item (iv) above, fix $f_{k}\in F$ and let $I_k\subset \{ m+1,\dots, N\}$ be such that the $K$-orbit of $f_{k}$ satisfies 
$Kf_k=\left\{  f_{j}:j\in I_{k}\right\}  $.  For each  $j\in I_{k}$ let $h_{j}\in K$  such that  $h_jf_k=f_j$.
Since the orbit $Kf_{k}$ is isomorphic
to $K/K_{f_{k}}$ we have $Kf_{k}=\left\{  h_{j}f_{k}:j\in I_{k}\right\}  $
 and $K=\cup_{j\in I_{k}%
}h_{j}K_{f_{k}}$. Thus 
\[
0=\left(  \sum_{g\in K}g\right)  f_{k}=\sum_{j\in I_{k}}\sum_{h\in K_{f_{k}}%
}h_{j}hf_{k}=\left\vert K_{f_{k}}\right\vert \sum_{j\in I_{k}}h_{j}%
f_{k}=\left\vert K_{f_{k}}\right\vert \sum_{f\in Kf_{k}}f\text{.}%
\]
This shows that the sum of the elements of the orbit of $f_{k}$ is zero. Since   $F$ is
$K$-invariant, then $\sum_{j=m+1}^{N}f_{j}$ is the sum of elements in disjoint orbits, each of which vanishes. Therefore,   $\sum_{j=m+1}^{N}f_{j}=0$ and $v\in J^{-1}(0)$.

\end{proof}

We finish this section with the proofs of Lemmas~\ref{lemma:exist-twisted-subgroup},  \ref{lemma:pull-back} and  \ref{lemma:sumK}.

\begin{proof}[Proof of Lemma~\ref{lemma:exist-twisted-subgroup}]
  Let $\Psi$ and $\Lambda$ denote the \emph{index} mappings associated to the given orderings $K_o=(g_1=e,g_2,\dots, g_m)$ and 
$F_o=(f_{m+1},\dots, f_N)$  of $K$ and $F$:
\[%
\begin{split}
&  \Psi:K\rightarrow\{1,\dots,m\},\qquad g_{i}\mapsto i,\\
&  \Lambda:F\rightarrow\{m+1,\dots,N\},\qquad f_{j}\mapsto j.
\end{split}
\]
Then $\Psi$ and $\Lambda$ are well defined bijections that satisfy
$g_{\Psi(\tilde{g})}=\tilde{g}$, $f_{\Lambda(\tilde{f})}=\tilde{f}$ for any
$\tilde{g}\in K$ and $\tilde{f}\in F$.

We define $\tau:K\rightarrow S_{N}$ by
\[
\tau(g_{i})(j)=%
\begin{cases}
\Psi(g_{i}g_{j})\quad\mbox{if}\quad j\in\{1,\dots,m\},\\
\Lambda(g_{i}f_{j})\quad\mbox{if}\quad j\in\{m+1,\dots,N\}.
\end{cases}
\]
It is a simple exercise to show that $\tau$ as defined above is indeed a
one-to-one group morphism with respect to our product convention in $S_N$ (see Remark~\ref{rmk:product-in-SN}).

We now show that $ M_{(K_o,F_o)}\subset\mbox{Fix}(\hat{K}_\tau)$. Let $(g_{1}%
v_{1},\dots,g_{m}v_{1},f_{m+1},\dots,f_{N})\in  M_{(K_o,F_o)}$ and for
$i\in\{1,\dots,m\}$ denote $w_{i}=g_{i}v_{1}$. Using the definition of the
action of $\hat{K}_\tau$ on $M$, we have, for any $j\in\{1,\dots,m\}$, that
\[%
\begin{split}
(\tau(g_{j}),g_{j})\cdot &  (g_{1}v_{1},\dots,g_{m}v_{1},f_{m+1},\dots
,f_{N})=\\
&  \qquad(g_{j}w_{\tau(g_{j}^{-1})(1)},\dots,g_{j}w_{\tau(g_{j}^{-1}%
)(m)},g_{j}f_{\tau(g_{j}^{-1})(m+1)},\dots,g_{j}f_{\tau(g_{j}^{-1})(N)}).
\end{split}
\]
However, using the definition of $\tau$, we find
\[
g_{j}w_{\tau(g_{j}^{-1})(i)}=g_{j}w_{\Psi(g_{j}^{-1}g_{i})}=g_{j}%
g_{\Psi(g_{j}^{-1}g_{i})}v_{1}=g_{j}g_{j}^{-1}g_{i}v_{1}=g_{i}v_{1},\quad
i=1,\dots,m,
\]
and
\[
g_{j}f_{\tau(g_{j}^{-1})(k)}=g_{j}f_{\Lambda(g_{j}^{-1}f_{k})}=g_{j}%
g_{j}^{-1}f_{k}=f_{k},\quad k=m+1,\dots,N,
\]
which shows that
\[
(\tau(g_{j}),g_{j})\cdot(g_{1}v_{1},\dots,g_{m}v_{1},f_{m+1},\dots
,f_{N})=(g_{1}v_{1},\dots,g_{m}v_{1},f_{m+1},\dots,f_{N}),
\]
and indeed $ M_{(K_o,F_o)}\subset\mbox{Fix}(\hat{K_\tau})$.

Now let $v=(v_{1},\dots,v_{N})\in\mbox{Fix}(\hat{K_\tau})$. The condition that
$(\tau(g_{j}),g_{j})\cdot v=v$ in particular implies that
\[
v_{j}=g_{j}v_{\tau(g_{j}^{-1})(j)}=g_{j}v_{\Psi(g_{j}^{-1}g_{j})}%
=g_{j}v_{\Psi(e)}=g_{j}v_{1},\qquad j=1,\dots,m,
\]
where the last identity uses that $g_{1}=e$ in the  ordering  $K_o$. Thus
$v_{i}=g_{i}v_{1}$ for all $i\in\{1,\dots,m\}$. Below we show that for any
$g\in K$ and $i,j\in\{m+1,\dots N\}$ we have $f_{j}=gf_{i}$ if and only if
$v_{j}=gv_{i}$. This implies that $\{v_{m+1},\dots,v_{N}\}$ is a $K$-invariant
subset of $S^2$ and, moreover, that  the $K$-isotropy of $f_{j}$ coincides with the $K$-isotropy of $v_{j}$. Thus $\{v_{m+1}%
,\dots,v_{N}\}$ is a $K$-invariant subset of $\mathcal{F}[K]$. In particular, considering that 
$\mathcal{F}[K]$ is finite (Remark~\ref{prop:finiteF}), we conclude that  there are finitely many possibilities for the last $N-m$ entries of of $v\in \mbox{Fix}(\hat{K_\tau})$.
It is not hard to see that each of these possibilities for  $\{v_{m+1},\dots,v_{N}\}$ defines a connected component of $\mbox{Fix}(\hat{K_\tau})$.
In particular, we conclude that $M_{(K_o,F_o)}$ is indeed a connected component of $\mbox{Fix}(\hat{K_\tau})$ as required.

Let $i,j\in\{m+1,\dots N\}$. We now show that indeed, for any $g\in K$ we have
$f_{j}=gf_{i}$ if and only if $v_{j}=gv_{i}$. Suppose first that $f_{j}%
=gf_{i}$. Using that $(\tau(g),g)\cdot v=v$ and the definition of $\tau$,
we find
\[
v_{j}=gv_{\tau(g^{-1})(j)}=gv_{\Lambda(g^{-1}f_{j})}=gv_{\Lambda(f_{i}%
)}=gv_{i}.
\]
Conversely, suppose that $v_{j}=gv_{i}$. Using again that $(\tau
(g),g)\cdot v=v$ we get
\[
v_{j}=gv_{\tau(g^{-1})(j)},
\]
and we conclude that $v_{i}=v_{\tau(g^{-1})(j)}$. Hence $\tau
(g^{-1})(j)=i$ which, in view of the definition of $\tau$, implies that
$f_{j}=gf_{i}$.
\end{proof}

\begin{remark}
The above proof, together with the observation that  $M_{(K_o,F_o)}$ is diffeomorphic to $S^2$,  shows that, in fact, each of the finitely many  connected components of  $\mbox{Fix}(\hat{K_\tau})$ is diffeomorphic to $S^2$.
\end{remark}

\begin{proof}[Proof of  Lemma \ref{lemma:pull-back}]
Let  $v=(v_1, \dots, v_N)\in M$. The tangent space $T_vM$ is given by
$T_{v}M=T_{v_{1}}S^{2}\times....\times T_{v_{N}}S^{2}.$
If $\alpha= (a_1, \dots, a_N)$ and $\beta= (b_1, \dots, b_N)\in T_vM$, then, by the definition of $\Omega$ in \eqref{eq:Ham}, we have
$
\Omega(v)(\alpha ,\beta)= \sum_{j=1}^N\omega_{S^2}(v_j)(a_j,b_j).
$
Now, let $u\in S^2$ and $a, b\in T_uS^2$. It is not difficult to compute
\begin{equation*}
T_u \rho_{(K_o,F_o)} (a)=  (a, g_2 a, \dots, g_ma, 0,\dots 0), \qquad T_u \rho_{(K_o,F_o)} (b)=  (b, g_2 b, \dots, g_mb , 0,\dots 0).
\end{equation*}
Therefore,
\begin{equation*}
\begin{split}
\Omega \left ( \rho_{(K_o,F_o)} (u) \right ) \left ( T_u \rho_{(K_o,F_o)} (a),T_u \rho_{(K_o,F_o)} (b) \right ) & = \sum_{j=1}^m\omega_{S^2}(g_ju)(g_ja,g_jb)+\sum_{j=m+1}^N\omega_{S^2}(f_j)(0,0) \\ &  = m\omega_{S^2}(u)(a,b),
\end{split}
\end{equation*}
where the last identity uses $\omega_{S^2}(g_ju)(g_ja,g_jb)=\omega_{S^2}(u)(a,b)$, which  follows from the fact that the $\SO(3)$ action on $S^2$ preserves the area form $\omega_{S^2}$. The above calculation
shows that $ \rho_{(K_o,F_o)}^*\Omega=m\omega_{S^2}$ as required.
\end{proof}

\begin{proof}[Proof of Lemma \ref{lemma:sumK}] 
For the subgroup $K=\mathbb{D}_{n}<\SO(3)$, $n\geq2$, we consider the
generator matrices $A=e^{J\zeta}\oplus1$ and $B=1\oplus-1\oplus-1$, where  $J$ is the symplectic $2\times 2$ matrix and 
 $\zeta=2\pi/n$. Then we have
\begin{align*}
\sum_{j=1}^{n}A^{j}  & =\sum_{j=1}^{n}\left(  e^{jJ\zeta}\oplus1\right)
=0\oplus0\oplus n,\\
\sum_{j=1}^{n}BA^{j}  & =\sum_{j=1}^{n}\left(  e^{-jJ\zeta}\oplus-1\right)
=0\oplus0\oplus-n.
\end{align*}
Thus $\sum_{g\in\mathbb{D}_{n}}g=0$. The groups $K=\mathbb{T}%
,\mathbb{O}$,$\mathbb{I}$ contain $\mathbb{D}_{2}$ as a subgroup, and since
$K=h_{1}\mathbb{D}_{2}\cup .....\cup h_{L}\mathbb{D}_{2}$, then
\[
\sum_{g\in K}g=\sum_{l=1}^{L}\sum_{g\in \mathbb{D}_{2}}h_{l}g=%
\sum_{l=1}^{L}h_{l}\left( \sum_{g\in \mathbb{D}_{2}}g\right) =0.
\]

\end{proof}

\subsection{Regularisation of collisions  of symmetric configurations}
\label{sec:collisions}

We now consider in more detail the collisions of  $(K,F)$-symmetric configurations. We begin with the following propositions that clarify the role of $\mathcal{F}[K]$.

\begin{proposition}
\label{prop:collision-description}
Let $K$, $F$, $K_o$ and $F_o$ be as  in the statement of Theorem~\ref{th-main-symmetry}. The following statements hold:
\begin{enumerate}
\item There is a one-to-one correspondence between $\mathcal{F}[K]$ and the collision configurations within $M_{(K_o,F_o)}$. In particular, $M_{(K_o,F_o)}$ contains  finitely many collision points.
\item If $u\in \mathcal{F}[K]$ then the point $\rho_{(K_o,F_o)}(u)$ is a $(K,F)$-symmetric collision configuration whose only collisions occur 
at the points of the orbit $K u$. Moreover, these are  all $k$-tuple  collisions  where $k=|K_{u}|$ if $u\nin F$ and  $k=|K_u|+1$ if $u\in F$.
\end{enumerate}
\end{proposition}
\begin{proof}
(i) We will prove that 
\begin{equation}
\label{eq:reduced-collisions}
M_{(K_o,F_o)} \cap \Delta = \rho_{(K_o,F_o)}(\mathcal{F}[K]),
\end{equation}
where  $\rho_{(K_o,F_o)}:S^2\to M_{(K_o,F_o)}$ is defined by \eqref{eq:embedding}. This completes the proof since, with this specified range,  $\rho_{(K_o,F_o)}$ is a bijection. Let $u\in \mathcal{F}[K]$. Then, by definition of $ \mathcal{F}[K]$, there exists $g_j\neq e$ such that $g_ju=u$. This implies that the first and  $j$th entries of $ \rho_{(K_o,F_o)}(u)$ coincide and hence
  $\rho_{(K_o,F_o)}(u)\in \Delta$.  Considering that $\rho_{(K_o,F_o)}(S^2)=M_{(K_o,F_o)}$, it follows that $ \rho_{(K_o,F_o)}(\mathcal{F}[K])\subset M_{(K_o,F_o)} \cap \Delta$. Now let $v=(v_1, g_2v_1,\dots, g_mv_1, f_{m+1},\dots, f_N)\in M_{(K_o,F_o)} \cap \Delta$. Then one of the two following possibilities necessarily holds:
  \begin{enumerate}
\item[(a)] $g_iv_1=g_jv_1$ for some $i\neq j \in \{1,\dots, m\}$. In this case we have $v_1=g_i^{-1}g_jv_1$ implying that $v_1\in  \mathcal{F}[K]$.
 \item[(b)] $g_iv_1=f_k$ for some $i\in \{1,\dots, m\}$, $k\in \{m+1, \dots, N\}$. Then we may write $v_1=g_i^{-1}f_k$. Since $F\subset \mathcal{F}[K]$ is $K$-invariant, this implies that $v_1\in  \mathcal{F}[K]$.
\end{enumerate}
Thus, in any case, if  $v\in M_{(K_o,F_o)} \cap \Delta$ we conclude that  $v_1\in  \mathcal{F}[K]$. Considering that  we may write $v= \rho_{(K_o,F_o)}(v_1)$ we conclude that  
  $v\in  \rho_{(K_o,F_o)}(\mathcal{F}[K])$ and hence  $M_{(K_o,F_o)} \subset \Delta = \rho_{(K_o,F_o)}(\mathcal{F}[K])$.
  
(ii)   The first $m$ entries of $v=(u,g_2u,\dots, g_mu,f_{m+1}, \dots, f_N)$ belong to the orbit $K u$, so it is  clear that  collisions can only occur at  points in this orbit. Since $K u$ is isomorphic to $K/K_u$, it follows
that there are only $m/|K_u|$ distinct points among the first $m$ entries of $v$, and that each of them is repeated exactly $|K_u|$ times. Now, if $u\nin F$ then,
 since $K$ is $F$-invariant,
 the last $m+1$ entries of $v$ are distinct from the first $m$ entries of $v$ and we indeed have $|K_u|$-tuple collisions. On the other hand, 
if $u\in F$, then, again by $K$-invariance of $F$, each point in the orbit $K u$ appears exactly once within the list $(f_{m+1}, \dots, f_N)$ and we have ($|K_u|+1$)-tuple collisions. 
\end{proof}

\begin{proposition}
\label{prop:red-syst-collisions}
The reduced Hamiltonian $h_{(K,F)}$ given by \eqref{eq:red-Ham} and the reduced system \eqref{eq:reduced-system} are well-defined and smooth away
from the finite set $\mathcal{F}[K]$. Moreover, the reduced system \eqref{eq:reduced-system} is complete on $S^2\setminus \mathcal{F}[K]$.
\end{proposition}
\begin{proof}
Equation \eqref{eq:reduced-collisions} implies that  away from  $\mathcal{F}[K]$ we may write $h_{(K,F)}$ as a composition of smooth maps: $h_{(K,F)}= H\circ  \rho_{(K_o,F_o)}$. So 
 $h_{(K,F)}$ is smooth on $S^2\setminus \mathcal{F}[K]$, and, therefore, so is the reduced system \eqref{eq:reduced-system}. The completeness of the reduced flow on $S^2\setminus \mathcal{F}[K]$ follows
 from Proposition \ref{prop:complete-flow} and item (ii) of Theorem \ref{th-main-symmetry}.
\end{proof}

In view of Proposition~\ref{prop:red-syst-collisions}, the reduced system  \eqref{eq:reduced-system} is smooth away from the finite set $\mathcal{F}[K]$.  We wish to define a regularisation that extends the reduced  system  \eqref{eq:reduced-system} to the   points in $\mathcal{F}[K]$ and yields a complete flow on $S^2$.   Since,
again by Proposition~\ref{prop:red-syst-collisions},  the flow of   \eqref{eq:reduced-system}  is 
 complete on $S^2\setminus \mathcal{F}[K]$, then the  points in $\mathcal{F}[K]$  have to be added as equilibrium points.

The regularisation that we propose is built with the \defn{regularised reduced Hamiltonian} that is the smooth function 
$\tilde{h}_{(K,F)}:S^{2}\rightarrow\R$ given by
\begin{equation}
\label{eq:regul-red-Ham}
\tilde{h}_{(K,F)}(u)=\exp(-2h_{(K,F)}(u)) =\prod_{j=2}^m\left  \vert u-g_j u \right \vert^m \prod_{j=m+1}^N \left \vert u-f_j \right \vert^{2m}. 
\end{equation}
Finally, the \defn{regularised reduced system} is the (smooth) Hamiltonian
vector field on $(S^{2},m\omega_{S^{2}})$ with Hamilton function $\tilde
{h}_{(K,F)}$, i.e.,
\begin{equation}
\dot{u}=-\frac{1}{m}u\times\nabla_{u}\tilde{h}_{(K,F)}%
(u).\label{eq:reg-reduced-system}%
\end{equation}
The relationship between the the reduced system \eqref{eq:reduced-system} and
its regularisation \eqref{eq:reg-reduced-system} is given next.

\begin{proposition} The following statements hold.
\label{prop:regul}
\begin{enumerate}
\item The curve $t\mapsto u(t)$ is a solution of the reduced system  \eqref{eq:reduced-system} if and only if  $t\mapsto u(a t)$ is a solution of the regularised  reduced system  \eqref{eq:reg-reduced-system} not
contained in $\mathcal{F}[K]$, where $a= -2\exp(-2  h_{(K,F)} (u(0)))$.
\item The points in $\mathcal{F}[K]$ are stable equilibria of the  regularised  reduced system  \eqref{eq:reg-reduced-system}.
\end{enumerate}
\end{proposition}
\begin{proof}
(i) For  $u\in S^2\setminus \mathcal{F}[K]$ one computes
\begin{equation*}
\nabla_u\tilde h_{(K,F)}(u) =-2 \exp(-2 h_{(K,F)}(u))\, \nabla _u h_{(K,F)}(u).
\end{equation*}
A simple calculation that uses conservation of energy verifies the result. (ii) This follows from the fact that $0$ is the minimum value of   $\tilde h_{(K,F)}$ and $\mathcal{F}[K]$ is the corresponding level set. 

\end{proof}

Based on the above proposition, the points in $\mathcal{F}[K]$  will be called \defn{collision equilibria} of the reduced system  \eqref{eq:reduced-system}  and its regularisation   \eqref{eq:reg-reduced-system}. It is important to remember that these are always stable.
Other equilibrium points of these systems will be called  \defn{non-collision equilibria}.

\begin{remark} To finish this section, we note that one may also define a regularisation of the unreduced system \eqref{eq:motion} by considering the Hamiltonian system 
on $(M,\Omega)$ with \defn{regularised Hamiltonian} $\tilde{H}:M\rightarrow\R$ defined by
\begin{equation*}
\tilde{H}(v):=\exp(-2H(v))=\prod_{i<j}\left\vert v_{i}-v_{j}\right\vert,
\end{equation*}
with $v=(v_{1},\dots,v_{N})\in M$
 (recall that $\Omega$ is defined by \eqref{eq:Ham}). This leads to the regularised equations of motion on $M$
\begin{equation*}
\dot{v}_{j}=-v_{j}\times\nabla_{v_{j}}\tilde H(v)=2\sum_{i=1(i\neq j)}^{N}
v_{i}\times v_{j}, \qquad j=1,\dots, N.
\end{equation*}
Since the regularised Hamiltonian  $\tilde H$ is also  $\hat G$-invariant, a version of Theorem \ref{th-main-symmetry} about the (discrete) reduction of the above system to the regularised reduced system \eqref{eq:reg-reduced-system} holds, and such result is valid also at the collision configurations 
(compare with  Remark \ref{rmks:theorem-symmetry2}).
\end{remark}

\subsection{Qualitative properties of  $(K,F)$-symmetric solutions}
\label{sec:qualitative}

We are now ready to state several facts about the qualitative properties of the reduced system  \eqref{eq:reduced-system}.

\begin{proposition}
\label{prop:reduced-periodic}
The following statements hold about the dynamics of the reduced system  \eqref{eq:reduced-system}.
\begin{enumerate}
\item  The non-collision equilibrium points are in one-to-one correspondence with 
the critical points of $ h_{(K,F)}:S^2\setminus \mathcal{F}[K]\to \R$. Moreover, local maxima and minima are (Lyapunov) stable equilibrium points surrounded by a 1-parameter
family
of  periodic orbits that may be parametrised by their energy, and
saddle points are unstable equilibrium points. 
\item  All regular  level sets of the reduced Hamiltonian $ h_{(K,F)}$ are periodic orbits.
\item  There exists  a 1-parameter family of periodic orbits, parametrised by their energy,  around each collision equilibrium point $u_0\in \mathcal{F}[K]$. The energy of these
  periodic orbits approaches $\infty$ and the period approaches $0$ as the orbits approach $u_0$.
\end{enumerate}
\end{proposition}
The above proposition exhausts the possibilities of motion of the reduced system except for the possible existence of heteroclinic/homoclinic solutions emanating from the
unstable non-collision equilibrium points. 

\begin{proof}
(i) Since the reduced Hamiltonian $ h_{(K,F)}$ is a first integral, the result is standard for Hamiltonian systems on a symplectic manifold of dimension 2.

(ii) The regularised reduced system \eqref{eq:reg-reduced-system}  is an integrable, 1-degree of freedom, Hamiltonian system on the compact symplectic manifold $S^2$. By the 
Arnold-Liouville Theorem, all regular level sets of the regularised reduced Hamiltonian $\tilde h_{(K,F)}$ are periodic orbits. However, it is a simple exercise to show that the regular level sets of  $\tilde h_{(K,F)}$ are in one-to-one correspondence
with the regular level sets of  $ h_{(K,F)}$. 

(iii) This follows from Proposition~\ref{prop:regul} and the fact that $ h_{(K,F)} (u_k)\to\infty$ for any sequence $\{u_k\}_{k\in \mathbb{N}}$  
of points in $S^2\setminus  \mathcal{F}[K]$ that approaches $\mathcal{F}[K]$ as 
as $k\to \infty$.
\end{proof}

Proposition \ref{prop:reduced-periodic} may be combined   with Theorem~\ref{th-main-symmetry}  to prove the existence of several periodic 
solutions of the system \eqref{eq:motion} describing the full dynamics of the $N$-vortex problem on the sphere. The following corollary gives two particular instances.
The first of these will be used in the sections ahead to prove the existence of nonlinear oscillations in the vicinity of the platonic solid equilibrium configurations.

\begin{corollary}
\label{cor:periodic-orbits-general}
Let $K_o$ and $F_o$ be orderings of $K$ and $F$.
\begin{enumerate}
\item If $u_0\in S^2\setminus \mathcal{F}[K]$ is a local maximum or minimum of the reduced Hamiltonian  $ h_{(K,F)}$ given by \eqref{eq:red-Ham}, then 
$v_0=\rho_{(K_o,F_o)}(u_0)$ is an equilibrium  of \eqref{eq:motion}, and there exists a 1-parameter family of periodic solutions $v_h(t)$ of  \eqref{eq:motion},
emanating from $v_0$, and parametrised by their energy $h$. Moreover, these solutions  are  of the form
$v_h(t)=\rho_{(K_o,F_o)}(u_h(t))$,
where $u_h(t)$ is  the  family of periodic solutions of the reduced system \eqref{eq:reduced-system}  emanating from $u_0$ described in item (i) of Proposition \ref{prop:reduced-periodic}.

\item If $u_0\in\mathcal{F}[K]$, then  $v_0=\rho_{(K_o,F_o)}(u_0)$ is a collision configuration (described in detail in Proposition~\ref{prop:collision-description}) and there exists a 1-parameter
family of periodic solutions of \eqref{eq:motion}, which may be parametrised by their energy $h$, which approaches $v_0$ as $h\to \infty$, and whose period tends to zero
in this limit. These solutions  have the form $v_h(t)=\rho_{(K_o,F_o)}(u_h(t))$, where $u_h(t)$ is the 1-parameter family of periodic solutions of  \eqref{eq:reduced-system}
described in item (iii) of Proposition \ref{prop:reduced-periodic}.
\end{enumerate}

\end{corollary}

\section{$\mathbb{D}_n$-symmetric solutions of $N=2n$ vortices (with no fixed vortices)}
\label{sec:Dn}

We consider the dihedral subgroup $K=\mathbb{D}_n<\SO(3)$, $n\geq 2$,  generated by the matrices 
\begin{equation}
\label{eq:AB-dihedral}
A=\begin{pmatrix} \cos \zeta &  -\sin \zeta& 0 \\ \sin \zeta & \cos  \zeta & 0 \\ 0 & 0 & 1  \end{pmatrix} \qquad \mbox{and} \qquad 
B=\begin{pmatrix}  1 & 0 & 0 \\ 0 & -1  & 0 \\ 0 & 0 & -1  \end{pmatrix},
\end{equation}
where here, and throughout,  we denote  $\zeta := 2\pi /n$. The  
 set $\mathcal{F}[\mathbb{D}_n]$ is given by
\begin{equation}
\label{eq:colset-Dn}
\mathcal{F}[\mathbb{D}_n]= \left \{ \left (\cos \left  ( (j-1)\frac{\zeta}{2} \right ) ,\sin \left ( (j-1)\frac{\zeta}{2}   \right ), 0  \right ) \, : \, j=1,\dots, 2n \right  \} \cup \left \{ \left (0,0,\pm 1 \right) \right\}.
\end{equation}

\subsection{ Classification of $\mathbb{D}_n$-symmetric equilibrium configurations of $N=2n$ vortices}

We consider $K=\mathbb{D}_n$ and $F=\emptyset$ so $N=2n$ and analyse the  reduced system  \eqref{eq:reduced-system}  in detail.
We start by noting that, in view of item (iii) of  Theorem \ref{th-main-symmetry} and Table \eqref{eq:table-normalizers},
 the system is $\mathbb{O}$-equivariant if $n=2$ and $\mathbb{D}_{2n}$-equivariant for $n\geq 3$.
The following theorem gives the full classification of the collision and non-collision equilibria of the reduced system \eqref{eq:reduced-system} and describes
their stability. It also indicates the correspondence of these equilibria with the equilibrium configurations  of the
equations of motion \eqref{eq:motion}. 

In the statement of the theorem, and for the rest of the paper, $T_k$ and $U_k$ respectively 
denote the Chebyshev polynomials of the first and second kind of degree $k$. To simplify the presentation, the proof is postponed to Section 
\ref{sec:dihedral-proofs} that is devoted to it.

\begin{theorem}
\label{th:dihedral}
Let  $K=\mathbb{D}_n$,  $F=\emptyset$, $n\geq 2$ and $N=2n$. 
The classification and stability of the  equilibrium points of the reduced system  \eqref{eq:reduced-system} is 
as follows.
\begin{enumerate}
\item The only non-collision equilibria of  \eqref{eq:reduced-system} are: 
\begin{enumerate}

\item The   \defn{anti-prism  equilibrium configurations}  at the $4n$ points given by:
 \begin{equation*}
A_j^\pm:=\left (  \sqrt{1-z_a^2} \, \cos \left ( (2j-1)\zeta/4  \right ) \, , \,  \sqrt{1-z_a^2} \, \sin \left ( (2j-1)\zeta/4  \right )\, ,  \, \pm z_a \right ), \qquad j=1, \dots, 2n,
\end{equation*}
where $z_a=z_a(n)\in (0,1)$ is uniquely determined by  $z_a^2=1-1/\lambda^2_a$ where $\lambda_a=\lambda_a(n)$ 
is the unique root greater than $1$ of the polynomial 
\begin{equation*}
\mathcal{P}_a(\lambda):=(3n-1)T_{2n}(\lambda)-nU_{2n}(\lambda)+2n-1.
\end{equation*}

These are {\em stable} equilibria of  \eqref{eq:reduced-system}
which  correspond to equilibrium configurations of \eqref{eq:motion} where the $N=2n$ vortices occupy  the 
vertices of the $S^2$-inscribed  $n$-gon anti-prism of height $2z_a$    (see Fig.\ref{F:anti-prism}). 

%

\item The   \defn{prism  equilibrium configurations}  at  the $4n$ points given by:
 \begin{equation*}
P^\pm_j:=\left (  \sqrt{1-z_p^2} \, \cos \left ( (j-1) \zeta/2  \right ) \, , \,  \sqrt{1-z_p^2} \, \sin \left ( (j-1)  \zeta/2  \right )\, ,  \, \pm z_p \right ), \qquad j=1, \dots, 2n,
\end{equation*}
where $z_p=z_p(n)\in (0,1)$ is uniquely determined by  $z_p^2=1-1/\lambda_p^2$ where $\lambda_p=\lambda_p(n)$ 
is the unique root greater than $1$ of the polynomial
\begin{equation*}
\mathcal{P}_p(\lambda):= (3n-1)T_{2n}(\lambda)-nU_{2n}(\lambda)-2n+1.
\end{equation*}
These are {\em unstable} equilibria  (saddle points)  of   \eqref{eq:reduced-system} 
which  correspond to equilibrium configurations of \eqref{eq:motion} where the $N=2n$ vortices occupy  the 
vertices of the $S^2$-inscribed  $n$-gon prism of height $2z_p$    (see Fig.\ref{F:prism}). 

%
%
\item The  \defn{polygon equilibrium configurations} at  the $2n$ points given by:
\begin{equation*}
 Q_j:=\left (  \cos \left ( (2j-1) \zeta/4  \right ), \sin \left ( (2j-1) \zeta/4  \right ), 0 \right ), \qquad j=1, \dots, 2n.
\end{equation*}
These are {\em unstable} equilibria  (saddle points)  of   \eqref{eq:reduced-system} 
which  correspond to equilibrium configurations of \eqref{eq:motion} where the $N=2n$ vortices occupy  the vertices of a regular $2n$-gon at the equator   (see
Fig.\ref{F:polygon}).
\end{enumerate}

\item The only  collision equilibria of (the regularisation of)  \eqref{eq:reduced-system} are:
\begin{enumerate}
\item The \defn{polar collisions}  at the north and south poles $(0,0,\pm 1)$. These correspond to collision configurations of \eqref{eq:motion} having
two simultaneous $n$-tuple collisions at antipodal points (see
Fig.\ref{F:polar-coll}).
\item The \defn{polygonal collisions} at  the $2n$ points given by:
\begin{equation*}
 C_j:=\left (  \cos \left ( (j-1) \zeta/2  \right ), \sin \left ( (j-1)  \zeta/2  \right ), 0 \right ), \qquad j=1, \dots, 2n.
\end{equation*}
 These correspond to collision configurations of \eqref{eq:motion} having
$n$  simultaneous binary collisions at a regular $n$-gon at the equator   (see
Fig.\ref{F:polygon-coll}).  
\end{enumerate}
All collision configurations are stable equilibria of  (the regularisation of)  \eqref{eq:reduced-system}.
\end{enumerate}
\end{theorem}

\begin{figure}[ht]
\begin{subfigure}{.3\textwidth}
  \centering
  \includegraphics[width=.8\linewidth]{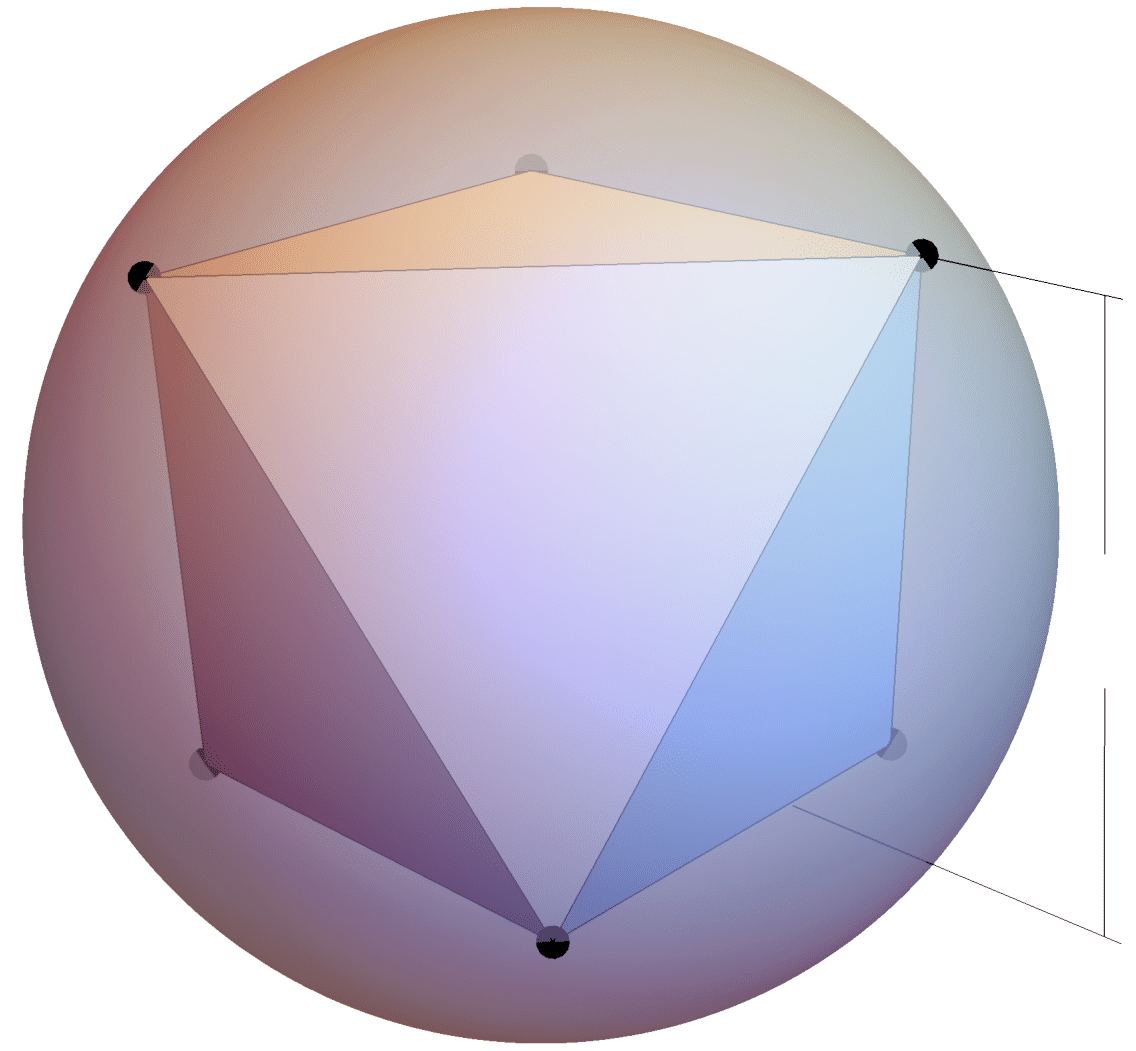}  
   \put (-6,38) {\small{$2z_a$}}
  \caption{Anti-prism equilibrium.}
  \label{F:anti-prism}
\end{subfigure}
\quad
\begin{subfigure}{.3\textwidth}
  \centering
  \includegraphics[width=.8\linewidth]{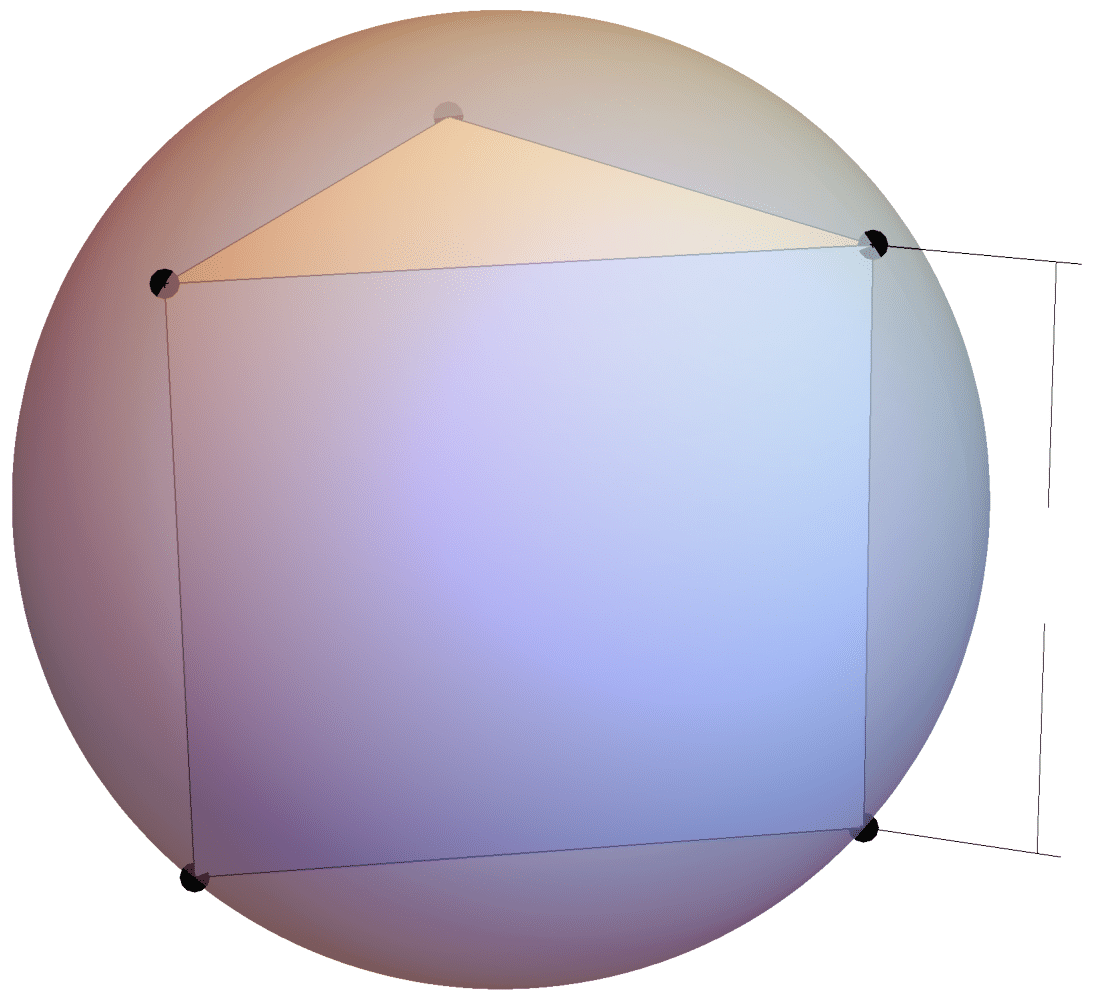}  
   \put (-8,39) {\small{$2z_p$}}
  \caption{Prism equilibrium.}
  \label{F:prism}
\end{subfigure}
\quad
\begin{subfigure}{.3\textwidth}
  \centering
  \includegraphics[width=.7\linewidth]{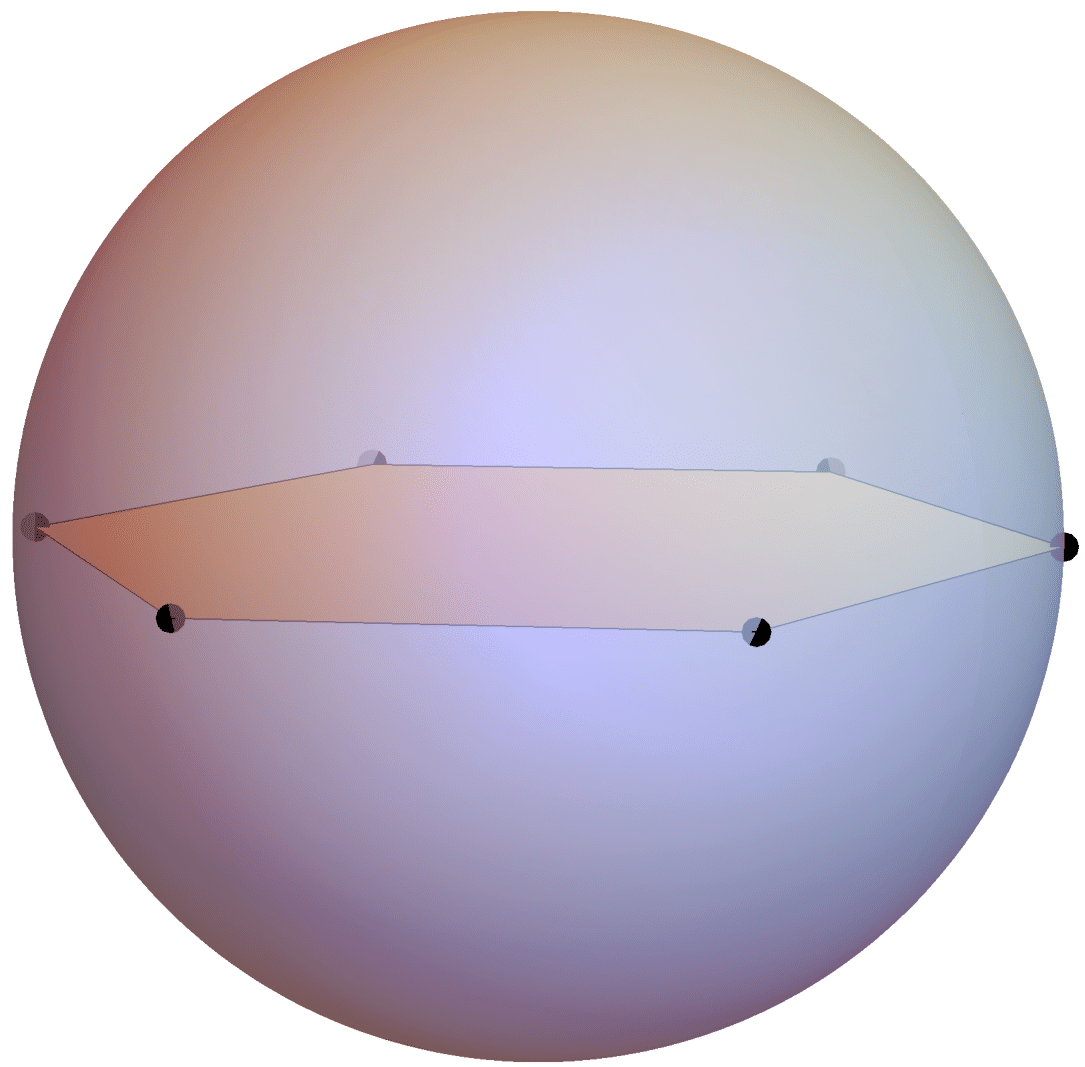}  
  \caption{Polygon equilibrium.}
  \label{F:polygon}
\end{subfigure} \\
\begin{subfigure}{.5\textwidth}
  \centering
  \includegraphics[width=.4\linewidth]{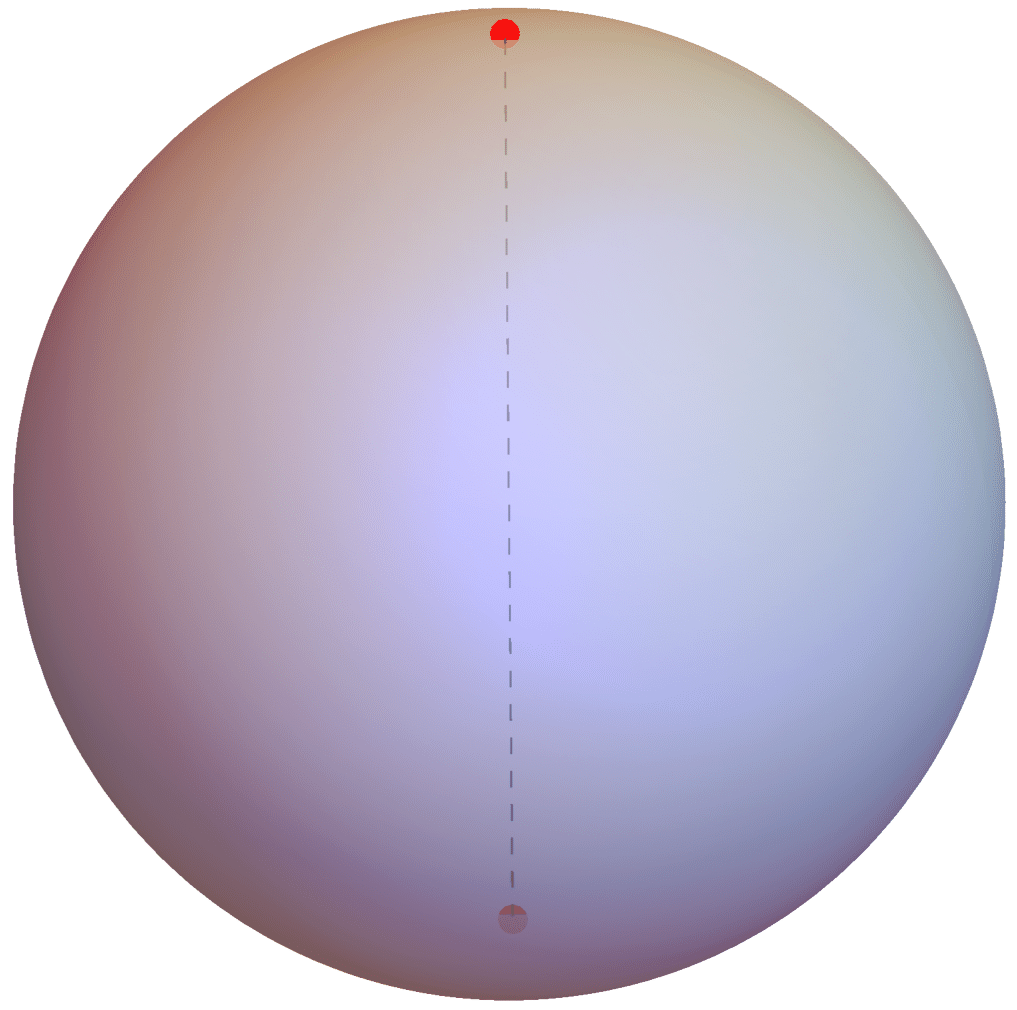}  
  \caption{Polar collision ($n$-tuple collision  at \\ antipodal points).}
  \label{F:polar-coll}
\end{subfigure} 
\begin{subfigure}{.5\textwidth}
  \centering
  \includegraphics[width=.4\linewidth]{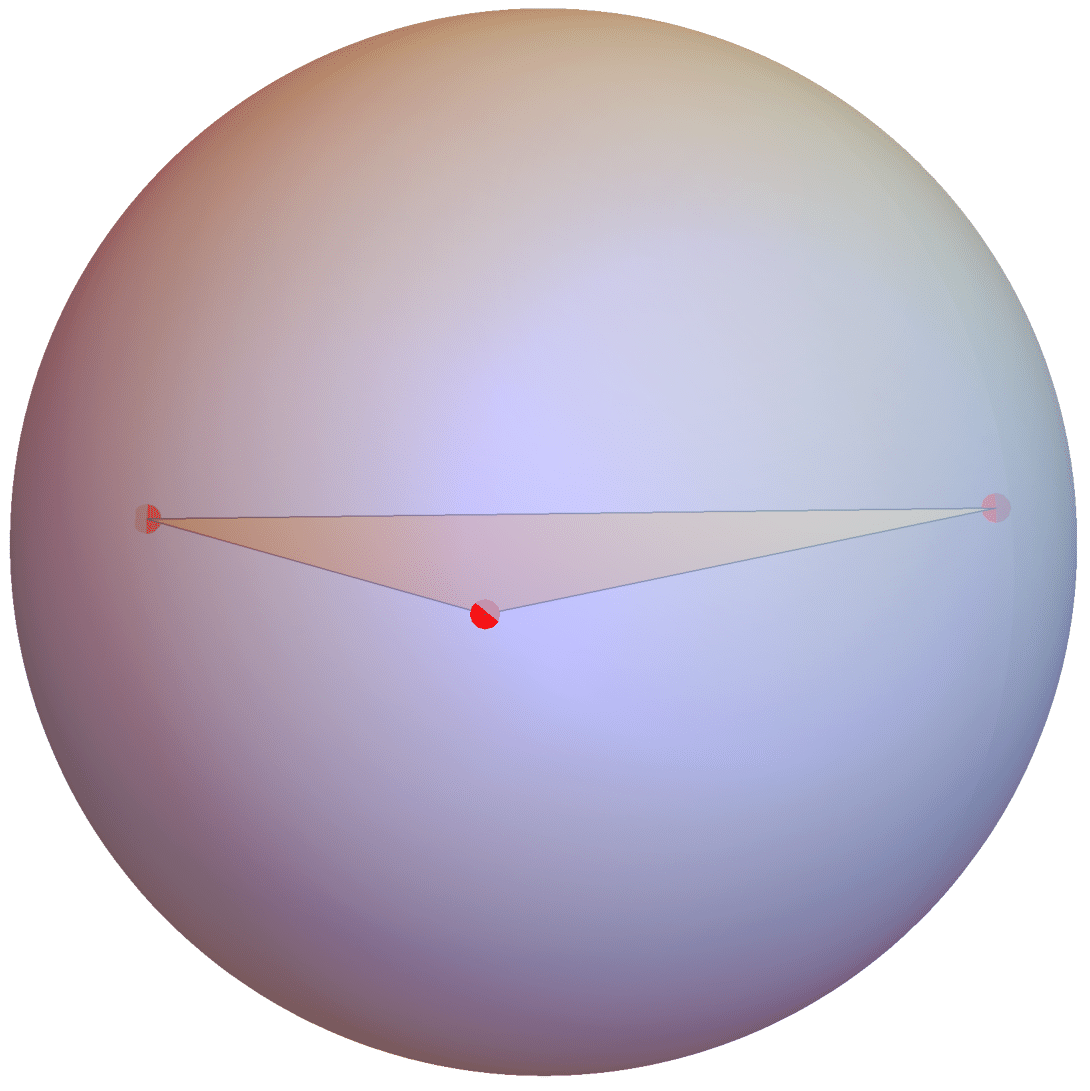}
  \caption{Polygonal collision (binary collisions at \\ vertices of a regular $n$-gon).}
  \label{F:polygon-coll}
\end{subfigure} 
\caption{Non-collision and collision equilibrium configurations described in Theorem \ref{th:dihedral} for $n=3$ and $N=2n=6$.}
\label{fig:dihedral-equilibria}
\end{figure}

In Tables \eqref{eq:table-D_n-za-roots} and \eqref{eq:table-D_n-zp-roots} below we give explicit expressions for the polynomials $\mathcal{P}_a(\lambda)$,  $\mathcal{P}_p(\lambda)$ 
and the numbers $\lambda_a$,  $z_a$, $\lambda_p$ and  $z_p$, appearing in the statement of the theorem for $n=2, \dots, 5$.

\begin{equation}
\label{eq:table-D_n-za-roots}
\small
\begin{array}
[c]{|c|c|c|c|}%
\hline
n  & \mathcal{P}_a(\lambda)  &\lambda_a &z_a \\ \hline
2  &8 \lambda ^4-16 \lambda ^2+6 &{\small \sqrt{\frac{3}{2}} }& {\small \frac{1}{\sqrt{3}} } \\ \hline
3  & 64 \lambda ^6-144 \lambda ^4+72 \lambda ^2 &\sqrt{\frac{3}{2}} &\frac{1}{\sqrt{3}} \\ \hline
4 & 384 \lambda ^8-1024 \lambda ^6+800 \lambda ^4-192
   \lambda ^2+14 &\frac{1}{2} \sqrt{\frac{1}{3}
   \left(10+\sqrt{58}\right)} &\sqrt{\frac{1}{7} \left(2 \sqrt{58}-13\right)}  \\ \hline
5&   2048 \lambda ^{10}-6400 \lambda ^8+6720 \lambda
   ^6-2800 \lambda ^4+400 \lambda ^2 &\frac{1}{4} \sqrt{15+\sqrt{65}} &\sqrt{\frac{1}{10} \left(\sqrt{65}-5\right)}  \\ \hline
\end{array}
\end{equation}

\begin{equation}
\label{eq:table-D_n-zp-roots}
\small
\begin{array}
[c]{|c|c|c|c|}%
\hline
n  & \mathcal{P}_p(\lambda)  &\lambda_p &z_p \\ \hline
2  &8 \lambda ^4-16 \lambda ^2 &\sqrt{2} &\frac{1}{\sqrt{2}}  \\ \hline
3  &64 \lambda ^6-144 \lambda ^4+72 \lambda ^2-10 &\frac{\sqrt{4+\sqrt{6}}}{2} &\sqrt{\frac{1}{5} \left(2 \sqrt{6}-3\right)} \\ \hline
4 & 384 \lambda ^8-1024 \lambda ^6+800 \lambda ^4-192
   \lambda ^2 &\sqrt{\frac{3}{2}} &\frac{1}{\sqrt{3}}  \\ \hline
5&  2048 \lambda ^{10}-6400 \lambda ^8+6720 \lambda
   ^6-2800 \lambda ^4+400 \lambda ^2-18 &\approx 1.20467...  &\approx 0.557613...  \\ \hline
\end{array}
\end{equation}

\subsection{ Dynamics of $\mathbb{D}_n$-symmetric configurations of $N=2n$ vortices}

Combining Theorem \ref{th:dihedral}   with  Corollary \ref{cor:periodic-orbits-general} we may establish the existence of three families 
of periodic orbits of the equations of motion \eqref{eq:motion} for $N$ even, $N\geq 4$.

\begin{corollary} 
\label{cor:dihedral-osc}
Let $n\geq 2$ and $N=2n$.
\begin{enumerate}
\item There exists a 1-parameter family of periodic solutions $v_h(t)$ of  the equations of motion \eqref{eq:motion}
emanating from the anti-prism equilibrium configurations   described in  Theorem \ref{th:dihedral}.
Along these solutions, each vortex travels around a small closed loop around a vertex of the $n$-gon anti-prism of height $2z_a(n)$ (see Fig. \ref{F:antiprism-osc}).

  \item There exists a 1-parameter family of periodic solutions $v_h(t)$ of  the equations of motion \eqref{eq:motion}
converging to the polar collision  described in  Theorem \ref{th:dihedral}.  
Along these solutions, $n$ vortices travel along a closed
loop around the north pole and the remaining  $n$ vortices   travel along a closed
loop around the south pole in the opposite direction (see Fig. \ref{F:polar-oscillations}). 

\item There exists a 1-parameter family of periodic solutions $v_h(t)$ of  the equations of motion \eqref{eq:motion}
converging to the polygonal collisions  described in  Theorem \ref{th:dihedral}.  Along these solutions, there is a pair of vortices that travels along a small closed
loop  around each of the vertices of the regular $n$-gon at the equator  (see Fig. \ref{F:polygonal-osc}).
\end{enumerate}
Each of these families may be  parametrised by the energy $h$. In cases (ii) and (iii) we have 
 $h\to \infty$ as the solutions  approach collision, and  the period approaches zero in this limit.

For each solution described above, the distinct  closed loops traversed by the vortices, and the position the vortices within the loop at each instant, may be obtained from a single one by the action of $\mathbb{D}_n$.
\end{corollary}

\begin{figure}[ht]
\begin{subfigure}{.3\textwidth}
  \centering
  \includegraphics[width=.8\linewidth]{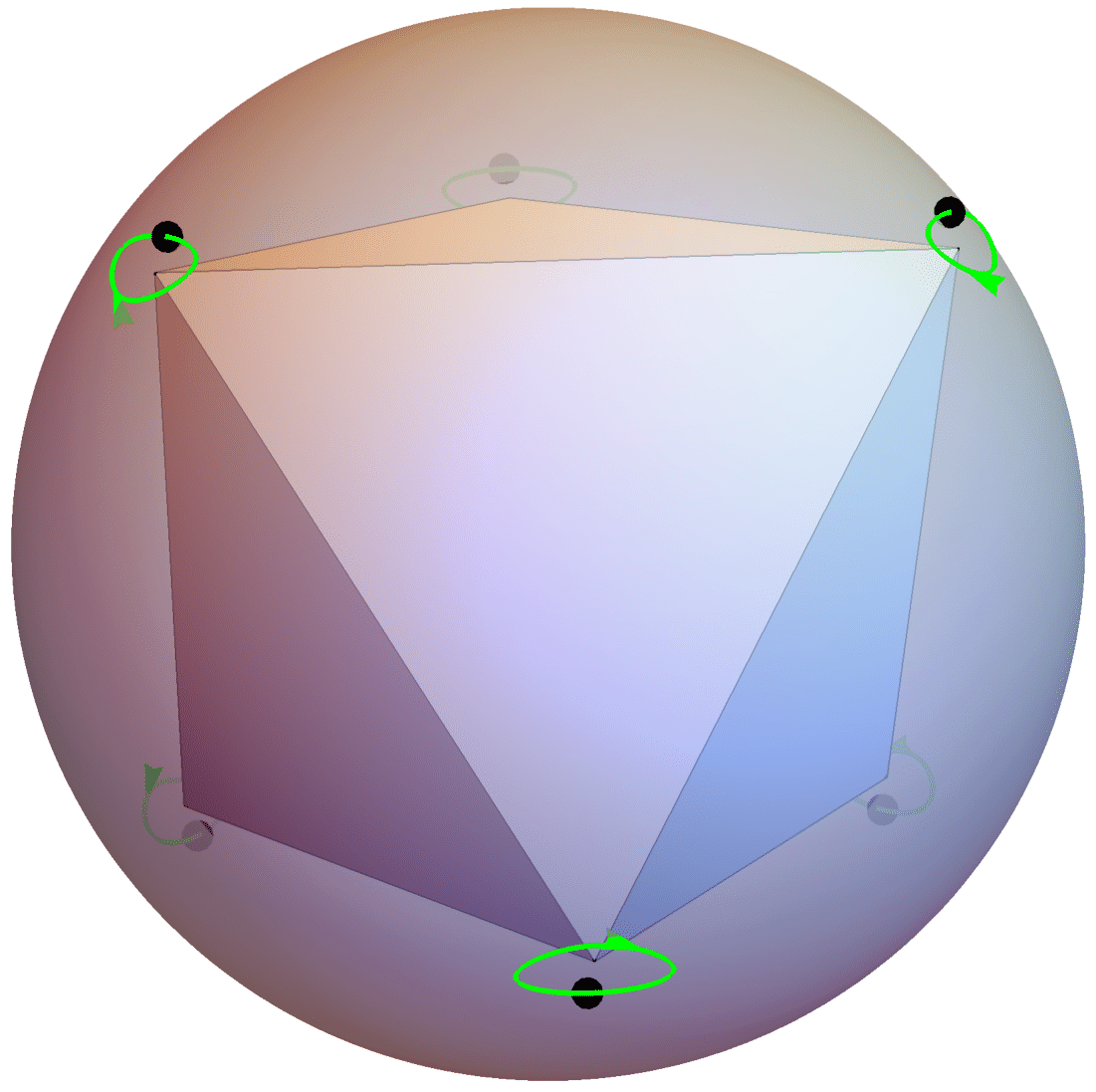}  
  \caption{Periodic solution near the anti-prism equilibrium.}
  \label{F:antiprism-osc}
\end{subfigure}
\quad
\begin{subfigure}{.3\textwidth}
 \centering
  \includegraphics[width=.7\linewidth]{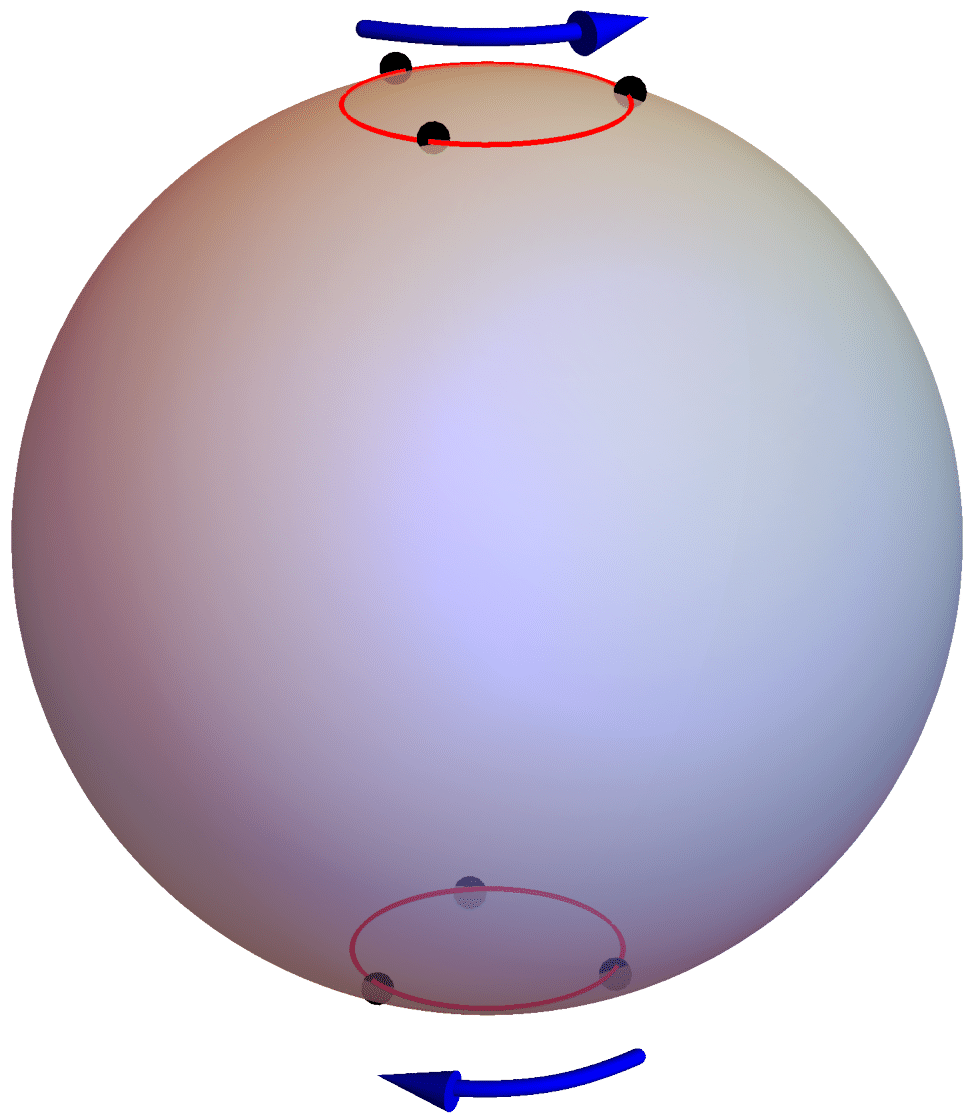}  
  \caption{Periodic solution near  the
  polar-collision.}
  \label{F:polar-oscillations}
\end{subfigure}
\quad
\begin{subfigure}{.3\textwidth}
  \centering
  \includegraphics[width=.8\linewidth]{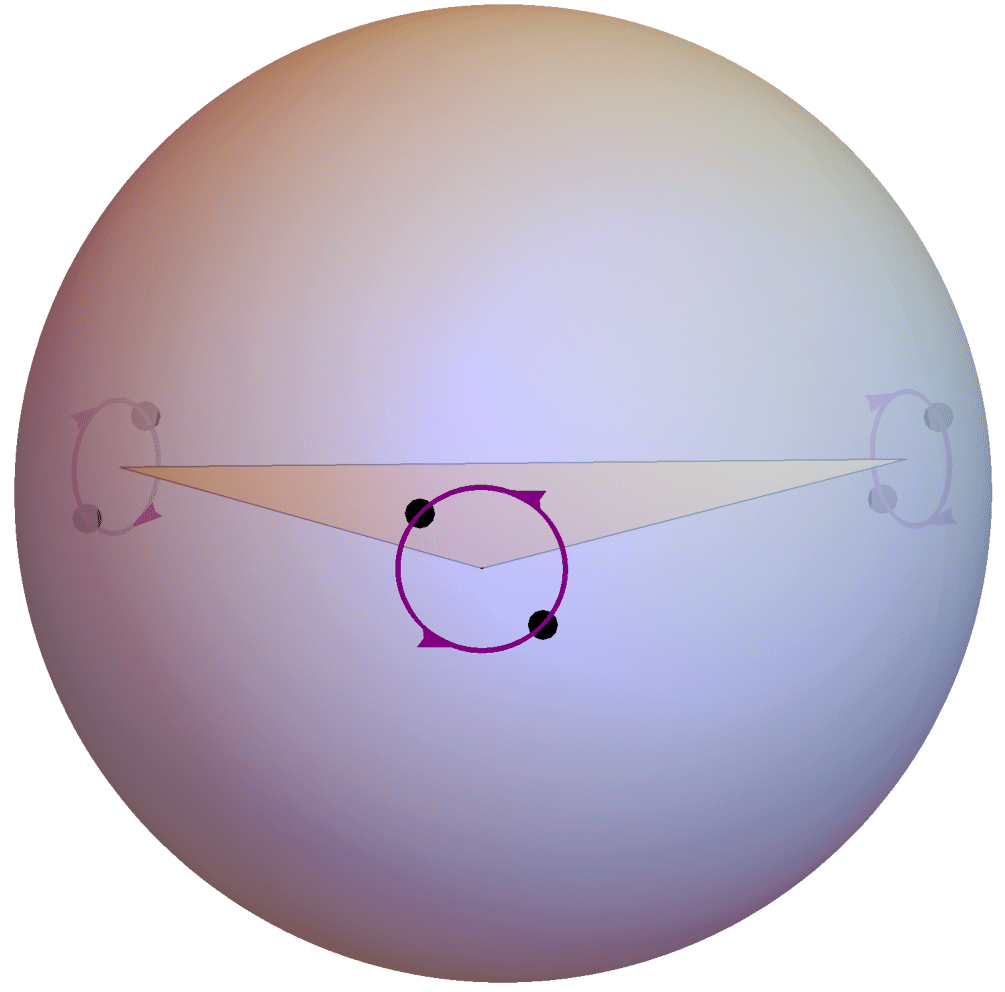}  
  \caption{Periodic solution near the 
  polygonal-collision.}
  \label{F:polygonal-osc}
\end{subfigure}
\caption{Periodic solutions  described in Corollary \ref{cor:dihedral-osc} for $n=3$, $N=6$.}
\label{fig:dihedral-osc}
\end{figure}

We now specialise our discussion to the cases $n=2, 3, 4$ which lead to appearance of  platonic solids as either prism or anti-prism equilibria.

\subsubsection*{Case $n=2$, $N=4$. Nonlinear small oscillations around the tetrahedron.} 

As we may read from Table \eqref{eq:table-D_n-za-roots}, the height of the anti-prism for $n=2$ is $2/\sqrt{3}$ and it is elementary to verify 
that the anti-prism is in fact a tetrahedron whose edges have length $2\sqrt{2/3}$. These configurations are known \cite{K}
to be stable equilibria of the unreduced dynamics \eqref{eq:motion} and in fact global minimisers of 
the Hamiltonian $H$. 
Item (i) of Corollary \ref{cor:dihedral-osc} shows the  existence
of small nonlinear oscillations of  \eqref{eq:motion} around these equilibria.

On the other hand, Table \eqref{eq:table-D_n-zp-roots} indicates that the prism configurations have height 
 $\sqrt{2}$. These (degenerate) prisms are in fact  squares of length $\sqrt{2}$. So, for $n=2$, the distinction between the prism and the polygonal equilibria  is artificial.
Similarly, since a $2$-gon on the equator degenerates to a diameter of the sphere, the distinction between the polar  and the polygonal collisions is artificial.

The phase space of the (regularised) reduced dynamics obtained numerically is illustrated in Figure~\ref{fig:n4-D2} below. 
The anti-prism equilibrium points $A_j^{\pm}$ are indicated in green,
the prism and polygonal equilibrium points, $P_j^\pm$ and $Q_j$, in black, and the polar and collision configurations $C_j$ in red. We note that the different families of
periodic orbits are separated by heteroclinic orbits connecting the unstable equilibria. 
Also, as  predicted by item (i) of Theorem~\ref{th:dihedral}, 
 we observe octahedral  symmetry in the reduced dynamics.

\subsubsection*{Case $n=3$, $N=6$. Nonlinear small oscillations around the octahedron.} 

For $n=3$, Table \eqref{eq:table-D_n-za-roots} indicates that  the height of the anti-prism  is again $2/\sqrt{3}$ and it is easy to verify 
that the anti-prism is in fact an octahedron whose edges have length $\sqrt{2}$ (see Fig.\ref{F:anti-prism}). Again, these configurations are known \cite{K}
to be stable equilibria of the unreduced dynamics \eqref{eq:motion} and  global minimisers of 
the Hamiltonian $H$.
Item (i) of Corollary \ref{cor:dihedral-osc} shows the the existence
of small nonlinear oscillations of  \eqref{eq:motion} around these equilibria. Also, as  predicted by item (i) of Theorem~\ref{th:dihedral}, 
 we observe $\mathbb{D}_6$  symmetry in the reduced dynamics.

The phase space of the (regularised) reduced dynamics obtained numerically is illustrated in Figure~\ref{fig:n6-D3} below. 
The anti-prism equilibrium points $A_j^{\pm}$ are indicated in green,
the prism equilibrium points $P_j^\pm$ in blue, polygonal  equilibrium points  $Q_j$ in black, polar  collisions in red and polygonal collisions  $C_j$ in purple. We have used
the same colour code to indicate either periodic orbits near the stable equilibria or heteroclinic orbits emanating from the unstable equilibria.
There is also a family  of periodic orbits that do not approach an equilibria or a collision that we have indicated in orange.

\subsubsection*{Case $n=4$, $N=8$. Instability of the cube.} 

For $n=4$, we read from Table \eqref{eq:table-D_n-zp-roots}  that  the height of the prism configuration is  $2/\sqrt{3}$ which corresponds to an
inscribed cube whose edges have this length. In contrast with the cases $n=2,3,$ treated above, our analysis does not lead to the existence of oscillations around a platonic solid, but rather to 
the conclusion that the cube is an unstable configuration of \eqref{eq:motion}. The instability of the cube had been reported before  \cite{K}.

On the other hand, we conclude from item (i) of Corollary \ref{cor:dihedral-osc} that there exists small nonlinear oscillations around the 
square anti-prism configuration of height  $\sqrt{(8\sqrt{58} -52)/7}$. These configurations are known \cite{K}
to be stable equilibria of the unreduced dynamics \eqref{eq:motion} and in fact global 
  minimisers of 
the Hamiltonian $H$.

The phase space of the (regularised) reduced dynamics obtained numerically is illustrated in Figure~\ref{fig:n8-D4} below. The colour 
code is identical to the one followed in the case $n=3$. This time  we observe $\mathbb{D}_8$  symmetry in the reduced dynamics.

\begin{figure}[ht]
\begin{subfigure}{.32\textwidth}
  \centering
  \includegraphics[width=.81\linewidth]{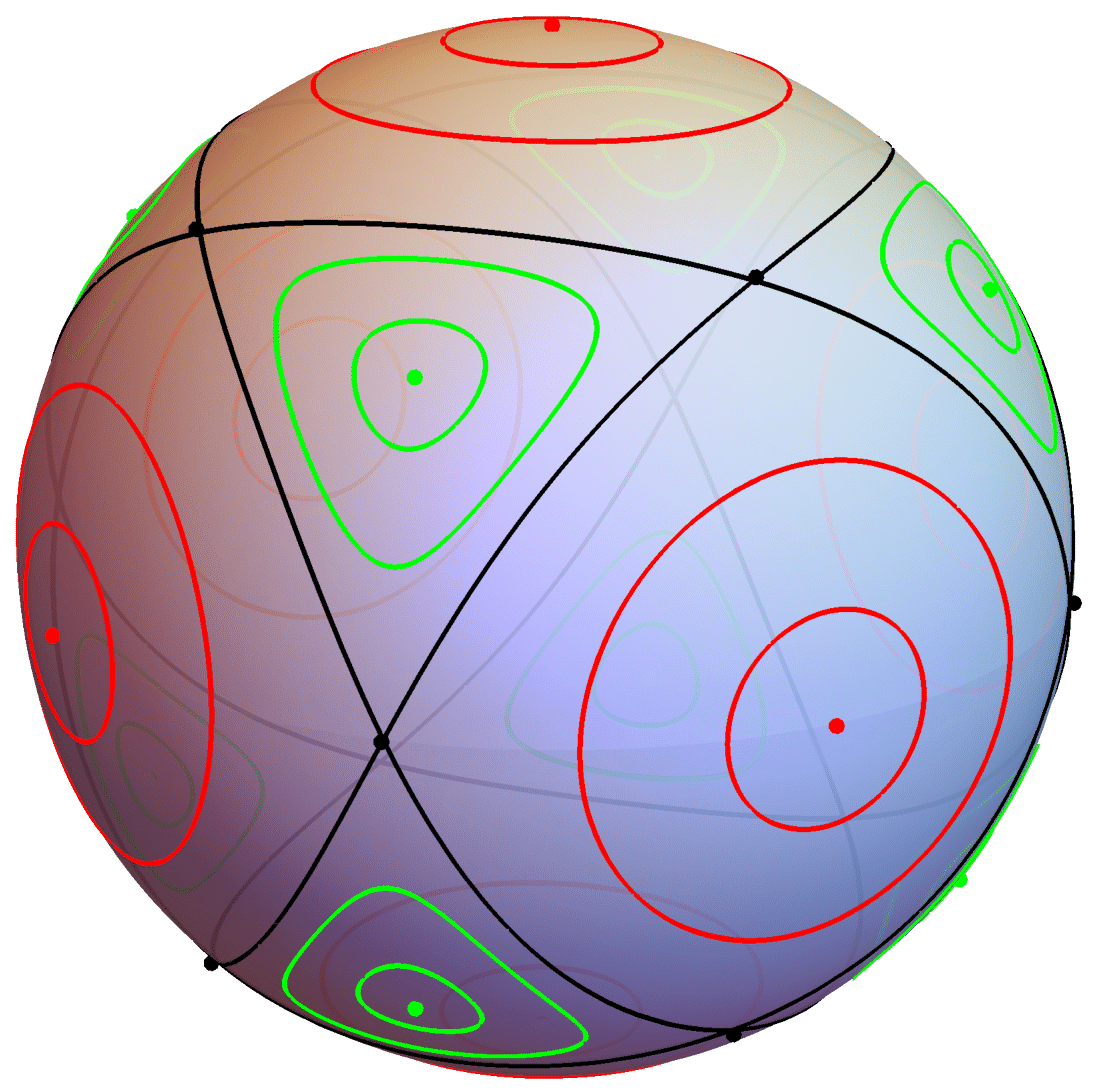}  
  \caption{$n=2$, $N=4$.}
  \label{fig:n4-D2}
\end{subfigure}
\;
\begin{subfigure}{.32\textwidth}
 \centering
  \includegraphics[width=.8\linewidth]{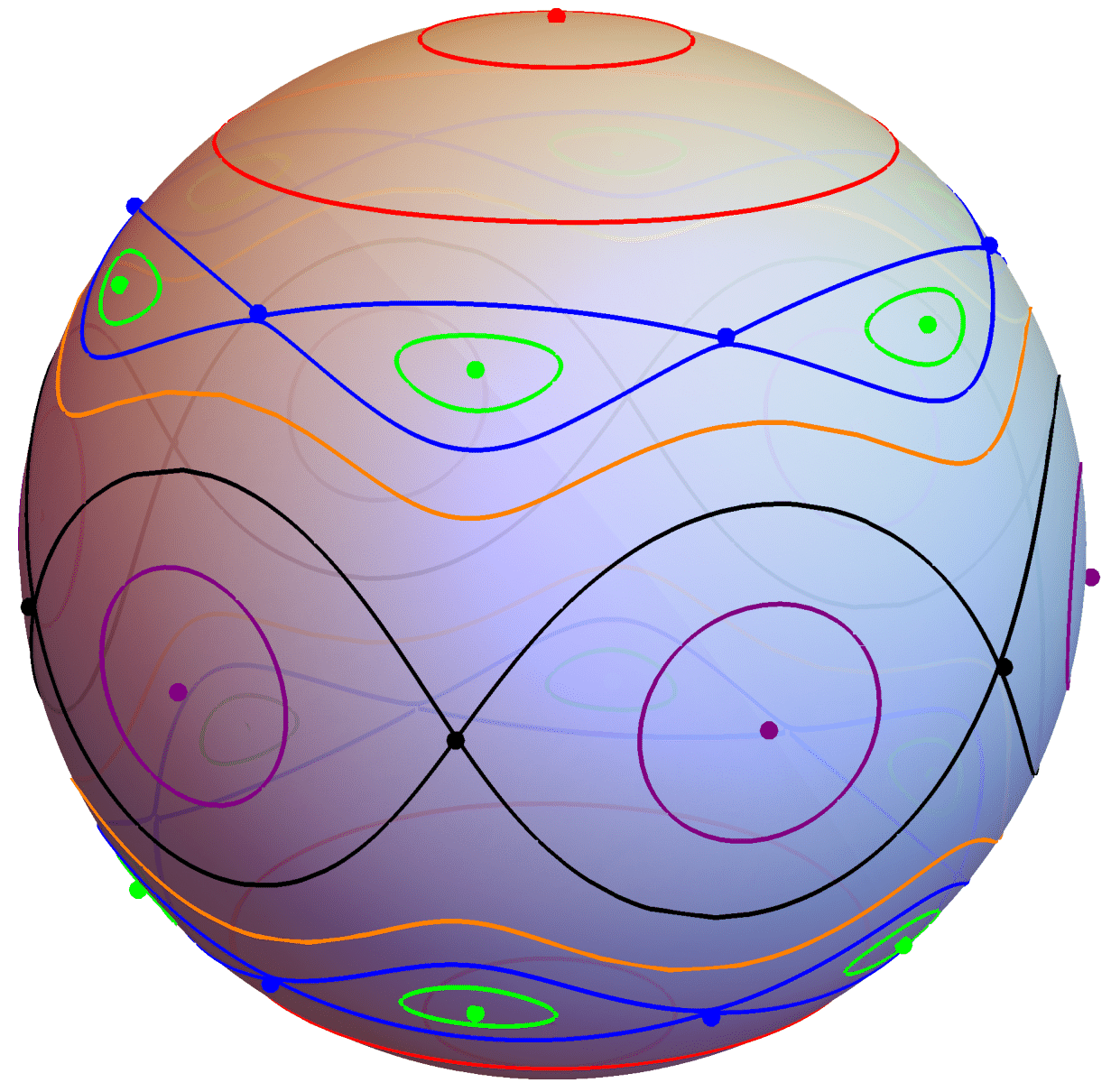}  
  \caption{$n=3$, $N=6$.}
  \label{fig:n6-D3}
\end{subfigure}
\;
\begin{subfigure}{.32\textwidth}
  \centering
  \includegraphics[width=.8\linewidth]{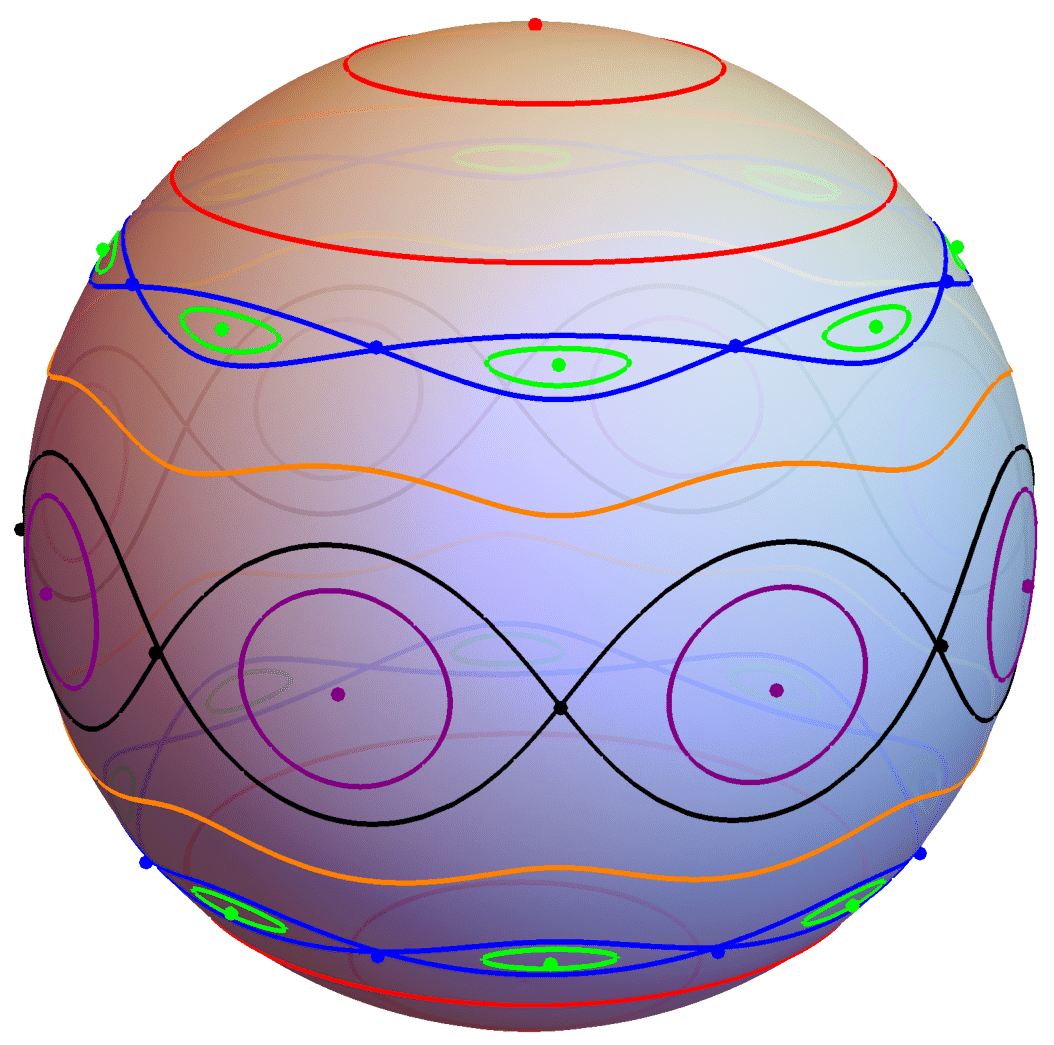}  
  \caption{$n=4$, $N=8$.}
  \label{fig:n8-D4}
\end{subfigure}
\caption{Phase space of the (regularised) reduced system \eqref{eq:reduced-system} for $K=\mathbb{D}_n$,  $F=\emptyset$ and $N=2n$ for  $n=2, 3, 4$. See text for explanations and description of the colour code. }
\label{fig:Dn-F-empty}
\end{figure}

\subsection{Proof of  Theorem \ref{th:dihedral}}
\label{sec:dihedral-proofs}

Our proof of Theorem \ref{th:dihedral} relies  on the following two lemmas that we state and prove first. The first  lemma gives us a 
working expression of  the reduced Hamiltonian $h_{(\mathbb{D}_n, \emptyset)}:S^2\to \R$ 
defined by \eqref{eq:red-Ham}, and the second one is a useful trigonometric identity.
To simplify notation, for the rest of this section we denote $h_{(\mathbb{D}_n, \emptyset)}$ simply by $h_n$.

\begin{lemma}
\label{lemma:simpHam-Dn}
In  cylindrical coordinates $(z,\theta)$ for $S^2$ defined by
\begin{equation}
\label{eq:cyl-coords}
x=\sqrt{1-z^2} \cos \theta, \quad y=\sqrt{1-z^2}\sin \theta, \quad z= z,
\end{equation}
we have, modulo  the addition of  constants:
\begin{equation}
\label{eq:hcyl}
h_n(z,\theta)=-\frac{n(n-1)}{2} \ln (1-z^2)  -\frac{n}{2} \, \sum_{j=1}^{2n} \,  \ln \left ( 1 - \sqrt{1-z^2} \cos \left ( \theta + \frac{j\zeta}{2} \right ) \right ), 
\end{equation}
and
\begin{equation}
\label{eq:hcyl-Cheb}
h_n(z,\theta)=-\frac{n(n-1)}{2} \ln (1-z^2) -\frac{n}{2} \,  \ln \left  ( \, q_{2n}(\sqrt{1-z^2}) -(1-z^2)^n \cos ( 2n\theta)  \, \right ), 
\end{equation}
 where $q_{2n}(\cdot )$ is the degree $2n$ polynomial defined by $q_{2n}(r)= r^{2n}T_{2n}(1/r)$. In particular we have
 \begin{equation}
 \label{eq:ineq-Cheb}
q_{2n}(\sqrt{1-z^2}) -(1-z^2)^n \cos (2n\theta)>0
\end{equation}
for all  $(z,\theta)$ corresponding to points on $S^2\setminus \mathcal{F}[\mathbb{D}_n]$.
%
%
%
%
%
\end{lemma}

\begin{proof}
We start by noticing that \eqref{eq:redHam-mainthm} yields
\begin{equation}
\label{eq:simp-red-ham-emptyset}
h_{(K,\emptyset)}(u)= -\frac{m}{4}  \sum_{j=2}^m \ln \left \vert  u-g_ju\right \vert^2 .
\end{equation}
Now we set $K=\mathbb{D}_n$ and  work with the following 
ordering of $\mathbb{D}_n$
\begin{equation*}
g_j=A^{j-1}, \quad j=1,\dots, n, \qquad g_j=B A^{j-n-1}, \quad j=n+1, \dots, 2n,
\end{equation*}
 where the matrices $A$ and $B$ are defined by \eqref{eq:AB-dihedral}. In view of \eqref{eq:simp-red-ham-emptyset} we have
 \begin{equation*}
h_n(u)=h_{(\mathbb{D}_n,\emptyset)}(u)= - \frac{n}{2}  \left ( \sum_{j=1}^{n-1}\ln  \left\vert u -A^ju \right \vert^2 + \sum_{j=1}^{n}\ln  \left\vert u -BA^ju \right \vert^2 \right ).
\end{equation*}
Writing $u$ in the  $(z,\theta)$-coordinates \eqref{eq:cyl-coords} we compute
\[
\left\vert u-A^{j}u \right\vert ^{2}=\left(  1-z^{2}\right)  \left\vert
1-e^{ij\zeta}\right\vert ^{2}=4\left(  1-z^{2}\right)  \sin^{2}\frac{j\zeta
}{2},
\]
and
\begin{align*}
\left\vert u -BA^{j}u \right\vert ^{2}  &  =\left(  1-z^{2}\right)
\left\vert 1-e^{i\left (  2\theta+j\zeta\right )  }\right\vert ^{2}+4z^{2}
 \\ & =4\left(  1-z^{2}\right)  \sin^{2}\left(  \theta+j\zeta /2 \right ) 
+4z^{2} =4-4\left(  1-z^{2}\right)  \cos^{2}\left(  \theta+j\zeta /2 \right ).
\end{align*}
Therefore, modulo the addition of  terms that are independent of $(z,\theta)$, we have 
\begin{equation*}
\ln \left \vert u-A^{j}u \right\vert ^{2} =\ln (1-z^2), \qquad \ln \left\vert u -BA^{j}u \right\vert ^{2} = \ln \left (1-\left(  1-z^{2}\right)  \cos^{2}\left(  \theta+j\zeta/2 \right )\right),
\end{equation*}
and hence,
\begin{equation*}
h_n(z,\theta)= - \frac{n(n-1)}{2} \ln (1-z^2)  - \frac{n}{2}  \sum_{j=1}^{n} \ln \left (1-\left(  1-z^{2}\right)  \cos^{2}\left(  \theta+j\zeta/2 \right )\right).
\end{equation*}
The proof that \eqref{eq:hcyl} holds follows by noting that 
\begin{align*}
&\sum_{j=1}^{n}\ln\left(  1-\left(  1-z^{2}\right)  \cos^{2}\left(
 \theta+j\zeta/2 \right)  \right)    \\
&  \qquad   =\sum_{j=1}^{n}\ln\left(  1-\left(
1-z^{2}\right)  ^{1/2}\cos\left(  \theta+j\zeta/2\right)  \right)
+\ln\left(  1+\left(  1-z^{2}\right)  ^{1/2}\cos\left(   \theta+j\zeta/2 \right)  \right) \\
& \qquad    =\sum_{j=1}^{n}\ln\left(  1-\left(
1-z^{2}\right)  ^{1/2}\cos\left(  \theta+j\zeta/2\right)  \right) +\ln\left(  1-\left(  1-z^{2}\right)
^{1/2}\cos\left( \theta+(n+j)\zeta/2 \right)  \right).
\end{align*}

In order to prove \eqref{eq:hcyl-Cheb} we begin with the identity
\begin{equation*}
\frac{1}{\lambda^k} \left ( \cosh (k\mu) -\cos (k\theta) \right ) = 2^{k-1} \prod_{j=1}^k\left ( 1 -\frac{1}{\lambda} \cos \left ( \theta + \frac{2j\pi}{k}\right ) \right ),
\end{equation*}
where $k\in \mathbb{N}$ and $\lambda =\cosh \mu \geq1$. This identity is a simple consequence of   \cite[Formula 1.395(2)]{Gradshteyn}.
Using the definition properties of the Chebyshev polynomials we may write $\cosh(k\mu)=T_k(\lambda)$, so, applying the above  identity with $k=2n$, we obtain
\begin{equation}
\label{eq:aux-ineq-proof}
\frac{1}{2^{2n-1}}\left ( q_{2n}(1/\lambda)- \frac{\cos 2n \theta}{\lambda^{2n} }\right ) =\prod_{j=1}^{2n} \left ( 1 - \frac{1}{\lambda} \cos \left ( \theta + j \zeta /2\right ) \right ).
\end{equation}
Setting $\lambda=(1-z^2)^{-1/2}$ and taking logarithms we obtain, modulo the addition of a constant,
\begin{equation*}
\sum_{j=1}^{2n}\ln\left(  1-\left(  1-z^{2}\right)^{1/2} \cos^{2}\left(
 \theta+j\zeta/2 \right)  \right) = \ln\left(   q_{2n}(\sqrt{1-z^2}) - (1-z^2)^n \cos 2n \theta \right ), 
\end{equation*}
which, in combination with  \eqref{eq:hcyl},  proves that \eqref{eq:hcyl-Cheb} indeed holds. 
Finally, note that, since $\lambda \geq 1$, the right hand side of \eqref{eq:aux-ineq-proof} 
is non-negative and can only vanish if $\lambda=1$ and $\theta= 2\pi -j\zeta/2$, $j=1,\dots, 2n$.
This observation shows that 
inequality \eqref{eq:ineq-Cheb} holds away from the points $C_j\in \mathcal{F}[\mathbb{D}_n]$.

\end{proof}

%
%

\begin{lemma}
\label{lemma:trig-identity} The following trigonometric identity holds
\[
\sum_{j=1}^{2n}\frac{\cos\left(  \zeta/4+j\zeta/2\right)  }{1-\cos\left(
\zeta/4+j\zeta/2\right)  }=2n(n-1).
\]
\end{lemma}

\begin{proof}
We begin by recalling the following identity from  \cite[Proposition 26]{GaIz11} 
\[
\frac{1}{2}\sum_{j=1}^{{l-1}}\frac{\sin^{2}(kj\pi/l%
)}{\sin^{2}(j\pi/l)}=\frac{1}{2}k\left(  l-k\right),
\]
that holds for $l\in \mathbb{N}$ and $0\leq k\leq l$.
In particular, for $l$ even and  $k=l/2$, we obtain
\begin{equation}
\label{eq:triglemma-aux}
l^{2}/8=\frac{1}{2}\sum_{j=1}^{l%
-1}\frac{\sin^{2}(j\pi/2)}{\sin^{2}(j\pi/l)}=\frac{1}{2}%
\sum_{\substack{j=1 \\ (j\text{ odd)}}}^{l-1}\frac{1}{\sin^{2}(j\pi/l%
)}\text{.}%
\end{equation}
On the other hand, we have 
\begin{align*}
\sum_{j=1}^{2n}\frac{1}{1-\cos\left(  j\pi/n + \pi/  2n \right)}  &
=\sum_{j=1}^{2n}\frac{1}{2\sin^{2}\left(  (2j+1)\pi/4n\right)  } =\frac{1}{2}\sum_{\substack{j=1 \\ (j\text{ odd)}}}^{4n-1}\frac{1}{\sin^{2}\left(  j\pi
/4n\right)  }=2n^{2},
\end{align*}
where we have used \eqref{eq:triglemma-aux} in the last identity with $l=4n$.
The desired result is an immediate consequence of the above identity since we may write 
\[
\sum_{j=1}^{2n}\frac{\cos\left(  \zeta/4+j\zeta/2\right)  }{1-\cos\left(
\zeta/4+j\zeta/2\right)  }=\sum_{j=1}^{2n}\left( \frac{1}{1-\cos\left(  j\pi/n + \pi/  2n \right)} -1  \right)  =2n^2-2n\text{.}%
\]

\end{proof}

We are now ready to present:

\begin{proof}[Proof of Theorem  \ref{th:dihedral}] For  item (ii), recall that the collision equilibrium configurations  occur at the points in
 $\mathcal{F}[\mathbb{D}_n]$ and are always stable. The set $\mathcal{F}[\mathbb{D}_n]$ is described by \eqref{eq:colset-Dn} 
 and consists of the north and south poles,  and the points $C_j$.
 Moreover, one can verify that the isotropy group of each of the the poles has order $n$, and the isotropy group of $C_j$ has order 2. Moreover,
  $\mathcal{F}[\mathbb{D}_n]$  contains three different  $\mathbb{D}_n$-orbits which are $\{(0,0,\pm1)\}$,  $\{C_j, \; j \; \mbox{odd}\}$
  and $\{C_j, \; j \; \mbox{even}\}$, and the latter ones determine a regular $n$-gon at the equator. These observations, together with item (ii) of 
  Proposition \ref{prop:collision-description}, show that the collision equilibria described above indeed correspond to the collision configurations of \eqref{eq:motion}
  described in the statement of the theorem.

In order to prove item (i) about the non-collision equilibria, we rely on item (i) of Proposition \ref{prop:reduced-periodic},
and determine  the  critical points of $h_n$. We will prove that these critical points  are $A_j^\pm$, $P_j^\pm$ and $Q_j$, and  that $A_j^\pm$ are local minima while
$P_j^\pm$ and $Q_j$ are saddle points.
We will work with the coordinates $(z,\theta)$ defined by \eqref{eq:cyl-coords}. These 
coordinates cover the whole sphere except for the north and south poles which are collision equilibria by item (ii)(a).

In view of item (iii) of  Theorem \ref{th-main-symmetry} and Table \eqref{eq:table-normalizers},
we know that $h_n$ is  $\mathbb{D}_{2n}$-invariant (for $n=2$ the group  $\mathbb{D}_4$ is a subgroup of the full symmetry group $\mathbb{O}$). 
This symmetry  implies that $h_n$ is $\zeta/2$-periodic in $\theta$, i.e. $h_n(z,\theta)=h_n(z,\theta+\zeta/2)$, and
also that $h_n(z,\theta)=h_n(-z,-\theta)$. Therefore, in our analysis of the critical points of $h_n$, we may restrict our attention to $(z,\theta) \in [0,1)\times  [0,\zeta/2)$. 
Note that, out of the points $A_j^\pm$, $P_j^\pm$ and $Q_j$ in the statement of the theorem, only     $A_1^+$, $P_1^+$ and $Q_1$ 
lie on this region, and the remaining ones may be obtained as $\mathbb{D}_{2n}$-orbits
of $A_1^+$, $P_1^+$ and $Q_1$ respectively. Thus, we only need to prove that  $A_1^+$, $P_1^+$ and $Q_1$ have the aforementioned properties and that
$h_n$ has no other (regular) critical points on  $(z,\theta)\in  [0,1)\times  [0,\zeta/2)$. For the rest of the proof we write these latter points in terms of their 
$(z,\theta)$ coordinates, namely
\begin{equation*}
A_1^+=(z_a, \zeta/4), \qquad P_1^+=(z_p,0), \qquad Q_1=(0,\zeta/4).
\end{equation*}

Using Eq. \eqref{eq:hcyl-Cheb} from Lemma \ref{lemma:simpHam-Dn} we have $\partial_\theta h_n(z,\theta) = - G(z,\theta)\sin ( 2n\theta)$ where 
\begin{equation*}
G(\theta,z):= \frac{n^2(1-z^2)^n}{q_{2n}(\sqrt{1-z^2}) -(1-z^2)^n \cos ( 2n\theta)}.
\end{equation*}
The inequality \eqref{eq:ineq-Cheb} shows that $G$ is a 
positive function away from the collision-equilibria.
In particular, we conclude that $\partial_\theta h_n(z,\theta)=0$  if $\theta=0$ or $\theta=\zeta/4$
and that $\partial_\theta h_n(z,\theta)\neq 0$ for other values of $\theta\in [0,\zeta/4)$. Hence, equilibria of $h_n$ in the region of interest can only occur 
if $\theta=0$ or $\theta=\zeta/4$.
Next we note from  Lemma \ref{lemma:simpHam-Dn} that  $h_n(z,\theta)$ is an even function of $z$ and thus $\partial_z h_n(0,\theta)=0$.
Therefore, we have  $\partial_z h_n(0,\zeta/4) =\partial_\theta h_n(0,\zeta/4)=0$ which shows that    $Q_1$ is indeed a
 critical point of $h_n$ (the other critical point  $(0,0)$ corresponds to the collision equilibrium $C_1$ at which $h_n$ is undefined).

Now we prove that there is exactly one zero $z_{a},z_{p}\in(0,1)$ of
$\partial_{z}h_{n}\left(  z,\zeta/4\right)  =0$ and $\partial_{z}h_{n}\left(
z,0\right)  =0$, respectively. In order to simplify the proof we make
the change of variables $r(z)=\sqrt{1-z^{2}}:(0,1)\rightarrow(0,1)$. Since
$r^{\prime}(z)\neq0$, the existence of a unique critical point of
$h_{n}(z,\theta)$ for $\theta=0,\zeta/4$ is equivalent to the existence of a
unique critical point of $h_{n}(r,\theta)$ for $\theta=0,\zeta/4$.
Using Eq. \eqref{eq:hcyl} from Lemma \ref{lemma:simpHam-Dn} we have 
\begin{equation*}
h_n(r,\theta) = -n(n-1) \ln (r) -\frac{n}{2} \, \sum_{j=1}^{2n} \,  \ln \left ( 1 -  r\cos \left ( \theta + j\zeta/2 \right ) \right ).
\end{equation*}
Since
$\lim_{r\rightarrow0}h_{n}(r,\theta)=\lim_{r\rightarrow1}%
h_{n}(r,0)=+\infty$, there exists a minimum $r_{p}\in (0,1)$ of the function $r\mapsto h_{n}(r,0)$.
On the other hand, differentiating the above expression and using  Lemma \ref{lemma:trig-identity} we find
that  for $\theta=\zeta/4$, we have 
\begin{equation}
\label{eq:aux-thm-dihedral-partialhr}
\partial_{r}h_{n}(1,\zeta/4)   =-n\left(  n-1\right)  +\frac{n}{2}\sum
_{j=1}^{2n}\frac{\cos\left(  \zeta/4+j\zeta/2\right)  }{1-\cos\left(
\zeta/4+j\zeta/2\right)  }  =n\left(  n-1\right)
^{2}>0.
\end{equation}
Therefore, using again that $\lim_{r\rightarrow0}h_{n}(r,\theta)=+\infty$, we conclude that there exists a minimum $r_{a}\in (0,1)$ of 
the function $r\mapsto h(r ,\zeta/4)$. However, since
\begin{equation}
\label{eq:ineq-sign-sec-derivative-hn-r} 
\partial_{r}^{2}h_{n}(r,\theta)=n\left(  n-1\right)  \frac{1}{r^{2}}+\frac
{n}{2}\sum_{j=1}^{2n}\frac{\cos^{2}\left(  \theta+j\zeta/2\right)  }{\left(
1+r\cos\left(  \theta+j\zeta/2\right)  \right)  ^{2}}>0,
\end{equation}
then $h_{n}(r,\theta)\ $has at most one critical point for $r\in(0,1)$. We
conclude that $z_{p}=\sqrt{1-r_{p}^2}$ and $z_{a}=\sqrt{1-r_{a}^2}$, are, respectively,  the unique critical points of
$h_{n}(z,0)$ and   $h_{n}(z,\zeta/4)$  on the interval
$z\in(0,1)$.

It remains to prove 
 that $z_a$ and $z_p$ may indeed be determined
in terms of the zeros of the polynomials $\mathcal{P}_a$ and $\mathcal{P}_p$ given in the  statement of the theorem. For this purpose note that
 Eq. \eqref{eq:hcyl}   and the condition $\partial_r h_n(r_a,\zeta/4)=0$ yield
\begin{equation*}
(2n-1)r_a^{2n}+(n-1)q_{2n}(r_a)+\frac{r_a}{2}q_{2n}'(r_a)=0.
\end{equation*}
Using the definition of $q_{2n}$, and since $r_a>0$, this is equivalent to
\begin{equation*}
2n-1+(2n-1)T_{2n}(1/r_a)-\frac{1}{2r_a}T_{2n}'(1/r_a)=0.
\end{equation*}
Therefore,  $\lambda_a:=1/r_a$ satisfies 
\begin{equation*}
2n-1+(3n-1)T_{2n}(\lambda_a)-nU_{2n}(\lambda_a)=0,
\end{equation*}
where we have made use of the 
 Chebyshev polynomial identities:
\begin{equation*}
T'_{2n}(s) =2n U_{2n-1}(s), \qquad sU_{2n-1}(s)=U_{2n}(s)-T_{2n}(s).
\end{equation*}
This shows that $z_a$ is indeed determined by a root  $\lambda_a>1$ of $\mathcal{P}_a$ as explained in the theorem.
 The unicity of $z_a$ as a critical point of $z\mapsto h_n(z,\zeta/4)$ shown above proves that such root  of $\mathcal{P}_a$
is necessarily unique. The analogous conclusion for $z_p$ is obtained {\em mutatis mutandis} starting from the condition $\partial_r h_n(r_p,0)=0$.

Thus, we have shown that indeed $A_1^+$, $P_1^+$ and $Q_1$ are the unique (non-collision) critical points of $h_n$ on the region $(z,\theta)\in [0,1)\times [0,\zeta/2)$.  We will now  prove that  $A_1^+$ is a local minimum, whereas $P_1^+$ and $Q_1$ are saddle points of $h_n$.
Starting from the condition $\partial_{\theta}h_{n}(z,\theta)=-\sin(2n\theta)G(z,\theta)$ with $G$ positive we have
\begin{equation*}
\partial_{\theta}^2h_{n}(z,\theta)=-2n\cos(2n\theta)G(z,\theta)-\sin
(2n\theta)\partial_{\theta}G(z,\theta),
 \quad  \partial_{z}\partial_{\theta}h_{n}(z,\theta)=-\sin(2n\theta)\partial_{z}G(z,\theta),
\end{equation*}
and therefore
\begin{equation}
\label{eq:ineqaux1Hessian}
\partial_{\theta}^2h_{n}(z,0)<0, \quad  \partial_{\theta}^2h_{n}(z,\zeta/4)>0, \qquad  \partial_{z}\partial_{\theta}h_{n}(z,0)=\partial_{z}\partial_{\theta}h_{n}(z,\zeta/4)=0.  
\end{equation}
On the other hand, we show below that
\begin{equation}
\label{eq:ineqaux2Hessian}
 \partial^2_{z}h_n(z_p,0 )>0, \qquad \partial^2_{z}h_n(z_a,\zeta/4)>0, \qquad  \partial^2_{z}h_n(0,\zeta /4)<0.
\end{equation}
The relations in
 \eqref{eq:ineqaux1Hessian} and  \eqref{eq:ineqaux2Hessian} prove that the Hessian matrix of $h_n$ is positive definite at $A_1^+$ and indefinite at $P_1^+$ and $Q_1$. 

To prove that   the  inequalities in \eqref{eq:ineqaux2Hessian} indeed hold we note that
\begin{equation}
\label{eq:chain-rule-aux-thm-dihedral}
\partial_{z}^{2}h_{n}=\partial_{z}\left( \left (  \partial_{r}h_{n} \right ) \left ( \partial
_{z}r\right) \right ) = \left ( \partial_{r}^{2}h_{n} \right ) \left(  \partial_{z}r\right)
^{2}+ \left ( \partial_{r}h_{n} \right ) \left ( \partial_{z}^{2}r \right ).
\end{equation}
Evaluating at $(z_p,0)$ and  $(z_a,\zeta/4)$ yields
\[
\partial_{z}^{2}h_{n}(z_{p},0)=\partial_{r}^{2}h_{n}(r_{p},0)\left(  \partial
_{z}r(z_{p})\right)  ^{2}>0, \qquad \partial_{z}^{2}h_{n}(z_{a},\zeta/4)=\partial_{r}^{2}h_{n}(r_{a},\zeta/4)\left(  \partial
_{z}r(z_{a})\right)  ^{2}>0,
\]
where we have used \eqref{eq:ineq-sign-sec-derivative-hn-r} and $\partial_{r}h_{n}(r_{p},0)=\partial_{r}h_{n}(r_{a},\zeta/4)=0$.
On the other hand, taking the 
the limit as $z\rightarrow0$ with $\theta=\zeta/4$ in \eqref{eq:chain-rule-aux-thm-dihedral}, and considering that in this limit  $r\rightarrow1$, $\partial
_{z}r\rightarrow0$ and  $\partial_{z}^{2}r\rightarrow-1$,  we obtain
\[
\partial_{z}^{2}h_{n}(0,\zeta/4)  =-\partial_{r}h_{n}(1,\zeta/4)=-n\left(  n-1\right)
^{2}\text{,}%
\]
where the last identity follows from \eqref{eq:aux-thm-dihedral-partialhr}. In particular, this  shows  that the third  inequality of \eqref{eq:ineqaux2Hessian} also holds.

Finally, we show that $A_j^\pm$,  $P_j^\pm$ and $Q_j$ respectively correspond to anti-prism, prism and polygonal equilibrium configurations of  \eqref{eq:motion} with the stated properties.
We begin by noting that  the set $\{A_j^\pm : j=1,\dots,2n\}$ consists of two  $\mathbb{D}_n$-orbits given by
\begin{equation*}
\{ A_{\mbox{\tiny $j$ odd}}^+,  A_{\mbox{\tiny $j$ even}}^-\} \quad \mbox{and} \quad \{ A_{\mbox{\tiny $j$ even}}^+,  A_{\mbox{\tiny $j$ odd}}^-\}. 
\end{equation*}
Each of these orbits has $2n$ points that lie on the vertices of an $n$-gon anti-prism as described in item (ii)(a) of the theorem. It follows that, for any ordering of $\mathbb{D}_n$,
 the mapping $\rho_{(\mathbb{D}_n,\emptyset)}$ defined by \eqref{eq:embedding} maps each of the points $A_j^{\pm}$ into  an anti-prism configuration with the given properties. Item (ii) of Theorem 
 \ref{th-main-symmetry} implies that these are equilibrium configurations of \eqref{eq:motion}. The analogous conclusion about the prism configurations is obtained by the same reasoning but  noting this time that   the set $\{P_j^\pm : j=1,\dots,2n\}$ consists of two  $\mathbb{D}_n$-orbits given by
$\{ P_{\mbox{\tiny $j$ even}}^\pm \}$ and $\{ P_{\mbox{\tiny $j$ odd}}^\pm\}$. 
The conclusion about $Q_j$ is also analogous but it is reached at once since $\{Q_j\}$ consists of a single  $\mathbb{D}_n$-orbit whose points lie on a regular $n$-gon at the equator.

%
%
%

\end{proof}

\section{$\mathbb{D}_n$-symmetric solutions of $N=2n +2$ vortices (two   antipodal vortices remain fixed)}
\label{sec:Dn-poles}

We continue to consider $K=\mathbb{D}_n$ but now we take $F=\{(0,0,\pm 1)\}$ so $N=2n+2$.
Note that the set $F$ satisfies both requirements in our setup since it is  $\mathbb{D}_n$-invariant and is contained in 
 $\mathcal{F}[\mathbb{D}_n]$ (see \eqref{eq:colset-Dn}). 
We analyse the  reduced system  \eqref{eq:reduced-system}  in detail. Since the set $F$ is also  $\mathbb{D}_{2n}$-invariant then, 
 in view of item (iii) of  Theorem \ref{th-main-symmetry} and Table \eqref{eq:table-normalizers},
 the system is  $\mathbb{D}_{2n}$-equivariant for all $n\geq 2$ (note that, in contrast with the previous section, the set
  $F=\{(0,0,\pm 1)\}$ is {\em not} $\mathbb{O}$-invariant so we cannot expect that
 the reduced system is  $\mathbb{O}$-equivariant for $n=2$).

%
%
%

\subsection{ Classification of $\mathbb{D}_n$-symmetric equilibrium configurations of $N=2n+2$ vortices}

The  analogous version of Theorem \ref{th:dihedral} on the classification and stability of the collision and non-collision equilibria 
of the reduced system  \eqref{eq:reduced-system} in this case is given next.

\begin{theorem}
\label{th:dihedral-poles}
Let  $K=\mathbb{D}_n$,  $F=\{(0,0,\pm 1)\}$, $n\geq 2$ and $N=2n+2$. 
The classification and stability of the  equilibrium points of the reduced system  \eqref{eq:reduced-system} is 
as follows.
\begin{enumerate}
\item The only non-collision equilibria of  \eqref{eq:reduced-system} are: 
\begin{enumerate}

\item For $n\geq3$, the   \defn{anti-prism with poles equilibrium configurations}  at the $4n$ points given by:
 \begin{equation*}
\hat A_j^\pm:=\left (  \sqrt{1-\hat z_a^2} \, \cos \left ( (2j-1) \zeta/4  \right ) \, , \,  \sqrt{1- \hat z_a^2} \, \sin \left ( (2j-1)\zeta/4  \right )\, ,  \, 
\pm \hat z_a \right ), \qquad j=1, \dots, 2n,
\end{equation*}
where $\hat z_a=\hat z_a(n)\in (0,1)$ is uniquely determined as  $\hat z_a^2=1-1/\hat \lambda_a^2$ where $\hat \lambda_a=\hat \lambda_a(n)$ 
is the unique root greater than $1$ of the polynomial 
\begin{equation*}
\hat{\mathcal{P}}_a(\lambda):=(3n+1)T_{2n}(\lambda)-nU_{2n}(\lambda)+2n+1.
\end{equation*}

These are {\em stable} equilibria of  \eqref{eq:reduced-system}
which  correspond to equilibrium configurations of \eqref{eq:motion} where $2n$ vortices occupy  the vertices of the  $S^2$-inscribed $n$-gon-anti-prism of height 
$2 \hat z_a$, and the 2 remaining vortices are antipodal and 
determine the diameter that is perpendicular to the antiprism   (see
Fig.\ref{F:anti-prism-poles}).

\item For all $n\geq 2$, the   \defn{prism  with poles equilibrium configurations}  at  the $4n$ points given by:
 \begin{equation*}
\hat P^\pm_j:=\left (  \sqrt{1- \hat z_p^2} \, \cos \left ( (j-1) \zeta/2  \right ) \, , \,  \sqrt{1-\hat z_p^2} \, \sin \left ( (j-1) \zeta/2  \right )\, ,  \, \pm \hat 
z_p \right ), \qquad j=1, \dots, 2n,
\end{equation*}
where $\hat z_p=\hat z_p(n)\in (0,1)$ is uniquely determined as  $\hat z_p^2=1 -1/\hat \lambda_p^2$ where $\hat \lambda_p=\hat \lambda_p(n)$ 
is the unique root greater than $1$ of the polynomial
\begin{equation*}
\hat{\mathcal{P}}_p(\lambda):= (3n+1)T_{2n}(\lambda)-nU_{2n}(\lambda)-2n-1.
\end{equation*}
These are {\em unstable} equilibria  (saddle points)  of   \eqref{eq:reduced-system} 
which  correspond to equilibrium configurations of \eqref{eq:motion} where $2n$ vortices occupy  the vertices of the  $S^2$-inscribed $n$-gon-prism of height 
$2 \hat z_p$, and the 2 remaining vortices are antipodal and 
determine the diameter that is perpendicular to the prism   (see
Fig.\ref{F:prism-poles}).

\item For all $n\geq 2$, the  \defn{polygon with poles equilibrium configurations} at  the $2n$ points given by:
\begin{equation*}
\hat  Q_j:=\left (  \cos \left ( (2j-1)/4  \right ), \sin \left ( (2j-1)\zeta/4 \right ), 0 \right ), \qquad j=1, \dots, 2n.
\end{equation*}
These points are stable equilibria of \eqref{eq:reduced-system}  if $n=2$ and unstable (saddle points) if $n\geq 3$.   
Moreover, they  correspond to equilibrium configurations of \eqref{eq:motion}  where $2n$ vortices occupy  the vertices of a regular $2n$-gon at the equator 
and the 2 remaining vortices are antipodal and determine the diameter that is perpendicular to the polygon  (see
Fig.\ref{F:polygon-poles}).
\end{enumerate}

\item The only  collision equilibria of (the regularisation of)  \eqref{eq:reduced-system} are:
\begin{enumerate}
\item The \defn{polar collisions}  at the north and south poles $(0,0,\pm 1)$. These correspond to collision configurations of \eqref{eq:motion} having
two simultaneous $(n+1)$-tuple collisions at antipodal points (see
Fig.\ref{F:polar-coll}).
\item The \defn{polygonal  with poles collisions} at  the $2n$ points given by:
\begin{equation*}
\hat  C_j:=\left (  \cos \left ( (j-1)\zeta/2  \right ), \sin \left ( (j-1) \zeta/2 \right ), 0 \right ), \qquad j=1, \dots, 2n.
\end{equation*}
 These correspond to collision configurations of \eqref{eq:motion} having
$n$  simultaneous binary collisions at a regular $n$-gon at the equator  and the 2 remaining vortices are antipodal and determine the diameter that is perpendicular to the polygon 
 (see
Fig.\ref{F:polygon-coll}).  
\end{enumerate}
All collision configurations are stable equilibria of  (the regularisation of)  \eqref{eq:reduced-system}.
\end{enumerate}
\end{theorem}

\begin{proof}[Proof of Theorem  \ref{th:dihedral-poles}] 
In broad terms, the proof of the theorem is analogous to that of Theorem \ref{th:dihedral} so we only indicate the key differences. The main one is that 
the expressions for the reduced Hamiltonian in  Lemma \ref{lemma:simpHam-Dn} have to be modified 
 to account for the presence of the vortices at the poles. In view of \eqref{eq:redHam-mainthm}, such correction is given by the addition of the term
\begin{equation*}
-\frac{m}{2} \sum_{j=m+1}^N\ln \left \vert u - f_j \right \vert^2=
-n\left  (\ln \left \vert u - (0,0,1)\right  \vert^2   +  \ln \left \vert u - (0,0,-1)\right  \vert^2 \right  ).
\end{equation*}
Writing $u$ in the cylindrical coordinates \eqref{eq:cyl-coords} and performing elementary operations shows that, up to the addition of a constant,
 the above expression equals $-n \ln (1-z^2)$. Therefore, if we simplify the notation and  denote
 the reduced Hamiltonian $h_{(\mathbb{D}_n, \{(0,0,\pm 1)\} )}:S^2\to \R$ simply by $\hat h_n$, we conclude that
 \begin{equation}
 \label{eq:red-ham-dihedral-poles}
\hat h_n(z,\theta)=h_n(z,\theta) - n \ln (1-z^2),
\end{equation}
where $h_n(z,\theta)$ is given by \eqref{eq:hcyl}, \eqref{eq:hcyl-Cheb}. 

Using the expression \eqref{eq:red-ham-dihedral-poles}, one may proceed in direct analogy with the  
proof of Theorem \ref{th:dihedral} to prove the result.  One difference that is worth  pointing out is the computation of 
\[
\partial_{r}\hat{h}_{n}(1,\zeta/4)=\partial_{r}h_{n}(1,\zeta/4)-2n=n\left(
n-1\right)  ^{2}-2n=n\left(  n^{2}-2n-1\right)  ,
\] 
Thus $\partial_{r}\hat{h}_{n}(1,\zeta/4)=-2$ for $n=2$ and
$\partial_{r}\hat{h}_{n}(1,\zeta/4)>0$ for $n\geq3$. This implies that for $n\geq 3$  there is a unique 
$r_{a}\in(0,1)$ such that $\partial
_{r}\hat{h}_{n}(r_{a},\zeta/4)=0$, while for  $n=2$ there are no solutions. Another difference is that
 the polygonal equilibrium points $Q_j$ 
are local minima of  $\hat h_n$ for $n=2$ and saddle points of  $\hat h_n$  for $n\geq 3$. 
We omit this and all other details. 
\end{proof}

\begin{figure}[ht]
\begin{subfigure}{.3\textwidth}
  \centering
  \includegraphics[width=.8\linewidth]{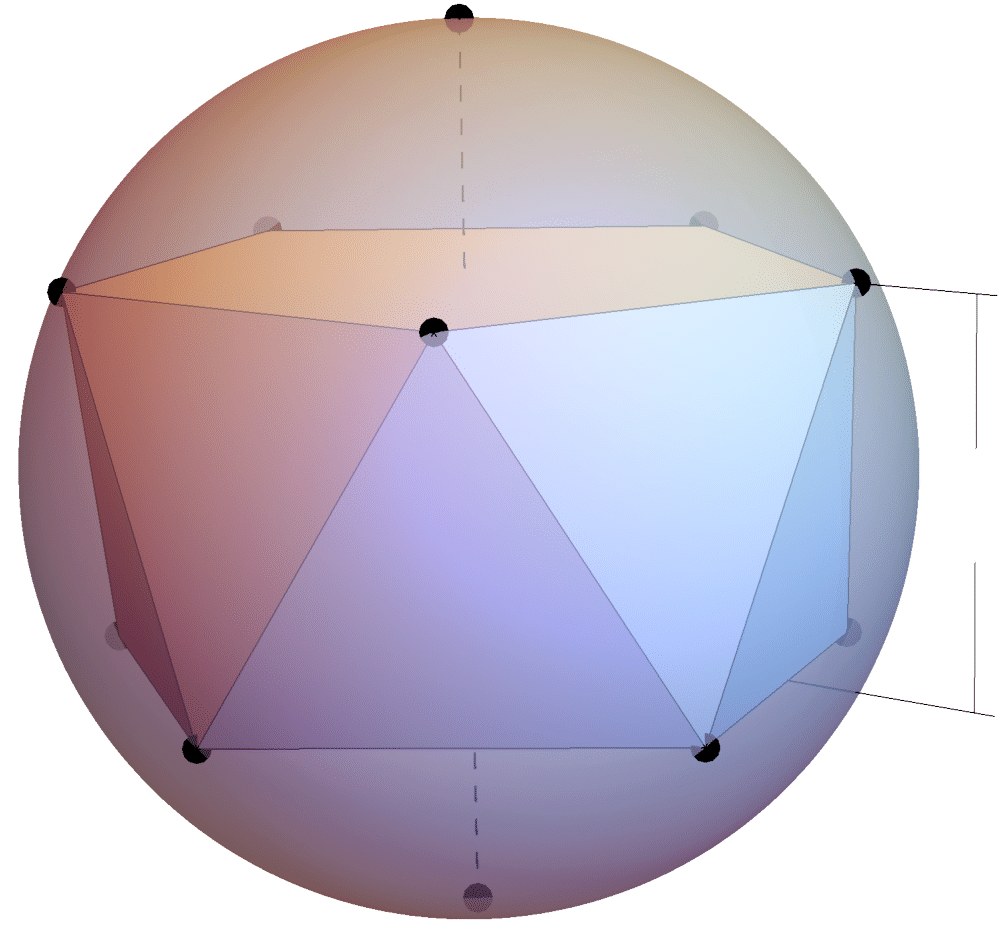}  
   \put (-8,42) {\small{$2\hat z_a$}}
  \caption{Anti-prism  with poles equilibrium. }
  \label{F:anti-prism-poles}
\end{subfigure}
\quad
\begin{subfigure}{.3\textwidth}
  \centering
  \includegraphics[width=.8\linewidth]{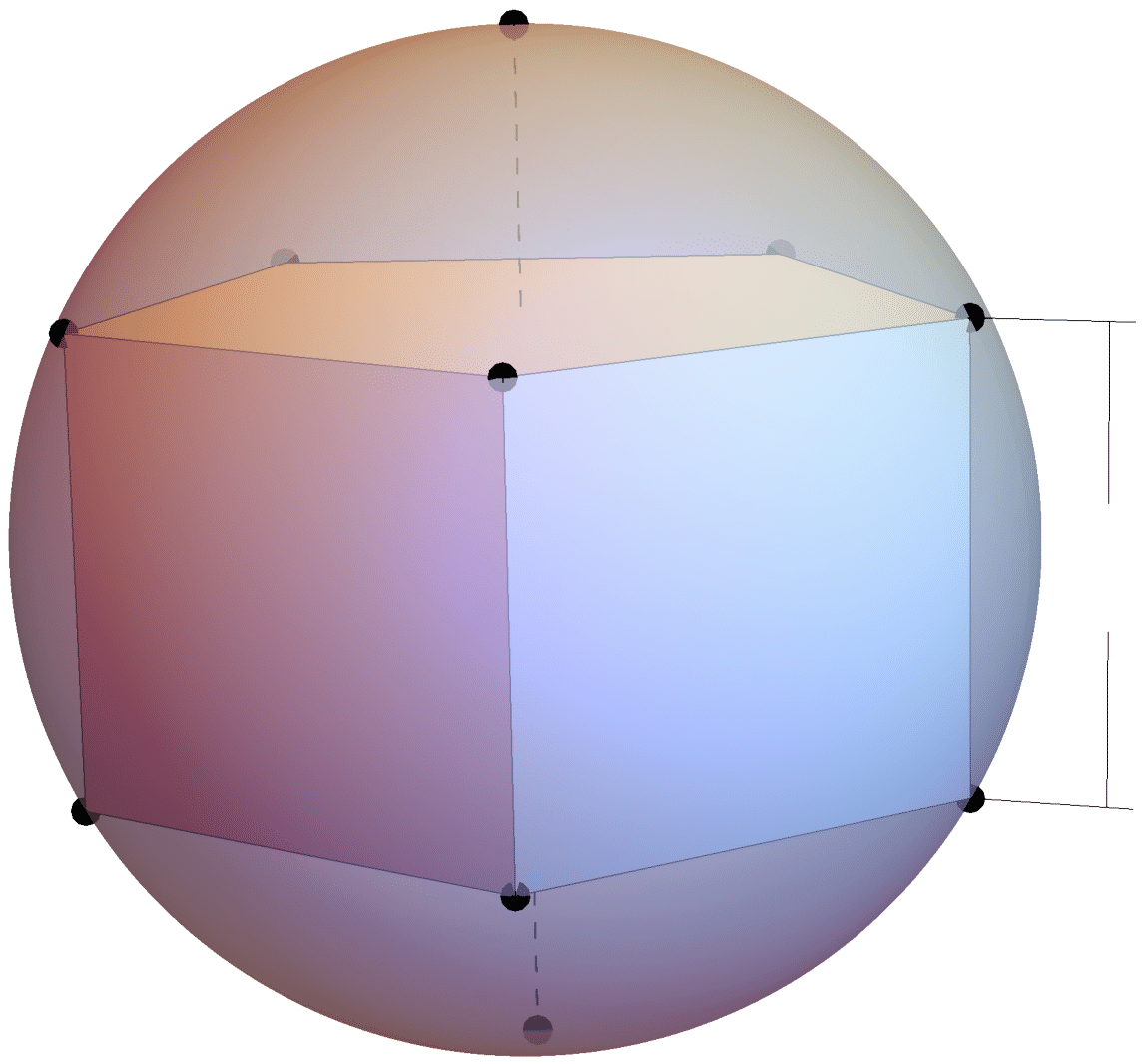}  
    \put (-8,43) {\small{$2\hat z_p$}}
  \caption{Prism with poles equilibrium. }
  \label{F:prism-poles}
\end{subfigure}
\quad
\begin{subfigure}{.3\textwidth}
  \centering
  \includegraphics[width=.7\linewidth]{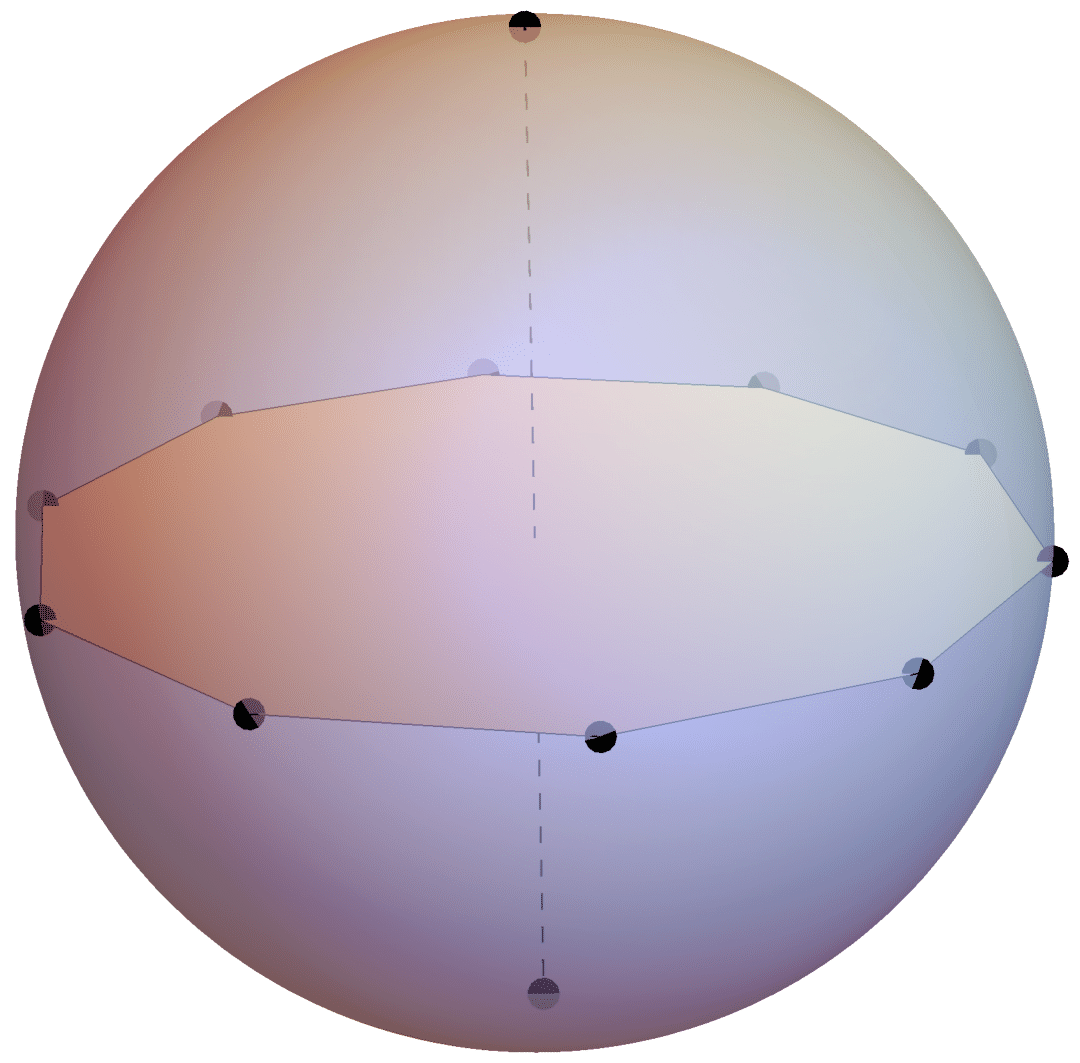}  
  \caption{Polygon with poles equilibrium.}
  \label{F:polygon-poles}
\end{subfigure}
 \\
\begin{subfigure}{.5\textwidth}
  \centering
  \includegraphics[width=.4\linewidth]{pole-col-bis-1.png}  
  \caption{Polar collision (($n+1$)-tuple collision  at \\ antipodal points).}
  \label{F:polar-coll}
\end{subfigure} 
\begin{subfigure}{.5\textwidth}
  \centering
  \includegraphics[width=.4\linewidth]{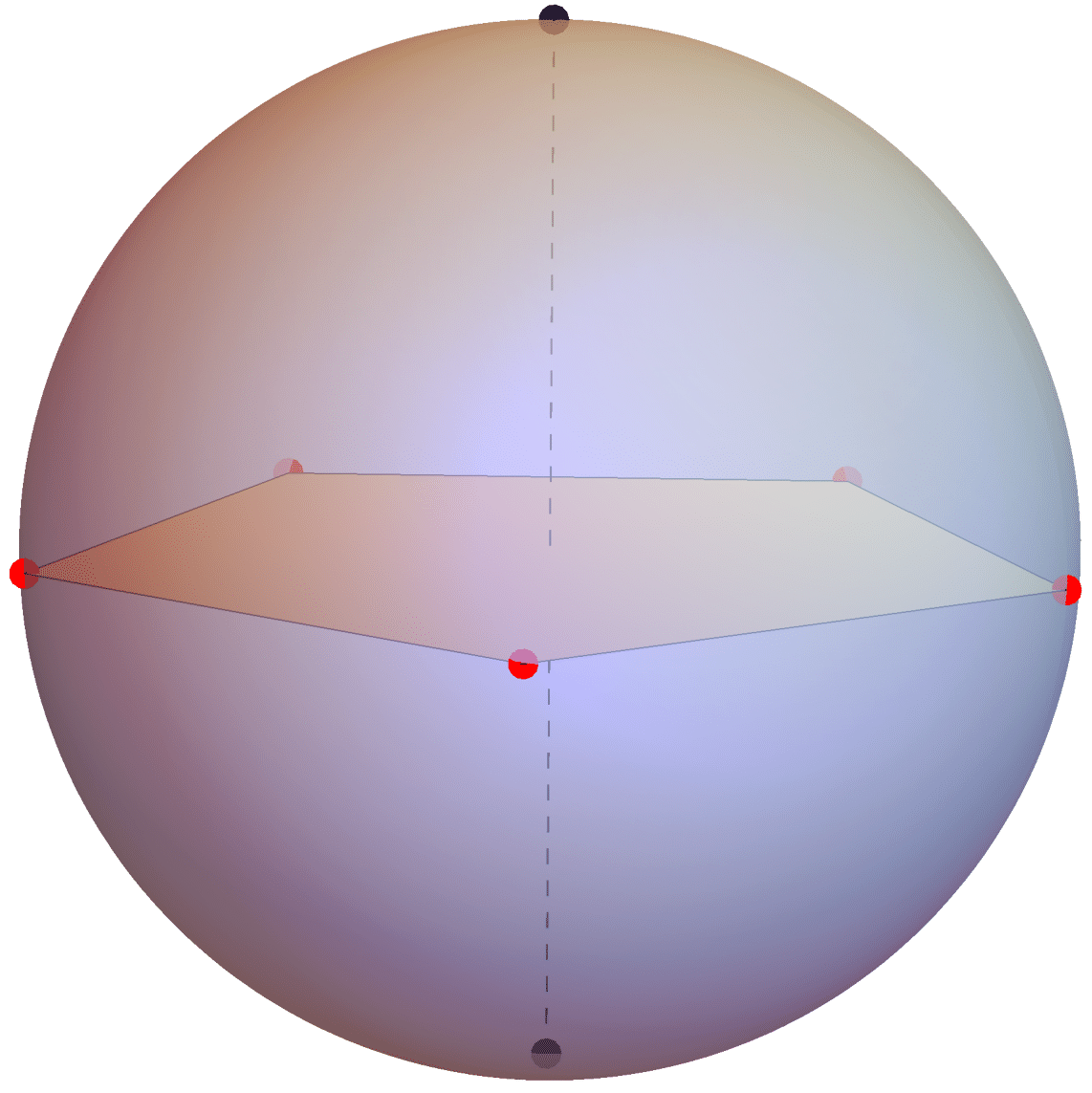}  
  \caption{Polygonal with poles collision  (binary collisions at the vertices of a regular $n$-gon).}
  \label{F:polygon-coll}
\end{subfigure} 
\caption{Non-collision and collision equilibrium configurations described in  Theorem \ref{th:dihedral-poles} for $n=5$ and $N=12$.}
\label{fig:dihedral-equilibria-poles}
\end{figure}

The  Tables below provide explicit expressions for the polynomials $\hat{\mathcal{P}}_a(\lambda)$,  $\hat{\mathcal{P}}_p(\lambda)$ 
and the numbers $\hat \lambda_a$,  $\hat z_a$, $\hat \lambda_p$ and  $\hat  z_p$,  in the statement of the theorem for $n=2, \dots, 5$.

\begin{equation}
\label{eq:table-D_n-za-roots-w-poles}
\small
\begin{array}
[c]{|c|c|c|c|}%
\hline
n  & \hat{\mathcal{P}}_a(\lambda)  &\hat \lambda_a &\hat z_a \\ \hline
2  &24 \lambda ^4-32 \lambda ^2+10  &-  &-   \\ \hline
3  & 128 \lambda ^6-240 \lambda ^4+108 \lambda ^2 &\frac{3}{2 \sqrt{2}} &\frac{1}{3} \\ \hline
4 &640 \lambda ^8-1536 \lambda ^6+1120 \lambda ^4-256
   \lambda ^2+18 &\frac{1}{2} \sqrt{\frac{1}{5}
   \left(14+\sqrt{106}\right)} &\frac{1}{3} \sqrt{2 \sqrt{106}-19} \\ \hline
5&   3072 \lambda ^{10}-8960 \lambda ^8+8960 \lambda
   ^6-3600 \lambda ^4+500 \lambda ^2 &\frac{\sqrt{5}}{2} &\frac{1}{\sqrt{5}}  \\ \hline
\end{array}
\end{equation}

\begin{equation}
\label{eq:table-D_n-zp-roots-w-poles}
\small
\begin{array}
[c]{|c|c|c|c|}%
\hline
n  &\hat{\mathcal{P}}_p(\lambda)  &\hat \lambda_p &\hat z_p \\ \hline
2  &24 \lambda ^4-32 \lambda ^2 &\frac{2}{\sqrt{3}}&\frac{1}{2}  \\ \hline
3  &128 \lambda ^6-240 \lambda ^4+108 \lambda ^2-14 &\frac{1}{4} \sqrt{13+\sqrt{57}} &\sqrt{\frac{1}{7} \left(\sqrt{57}-6\right)} \\ \hline
4 &640 \lambda ^8-1536 \lambda ^6+1120 \lambda ^4-256
   \lambda ^2 &\frac{1}{2} \sqrt{\frac{1}{5}
   \left(19+\sqrt{41}\right)}&\frac{1}{4} \sqrt{\sqrt{41}-3}\\ \hline
5&  3072 \lambda ^{10}-8960 \lambda ^8+8960 \lambda
   ^6-3600 \lambda ^4+500 \lambda ^2-22 &\approx 1.12677...  &\approx 0.460816...  \\ \hline
\end{array}
\end{equation}

\subsection{ Dynamics of $\mathbb{D}_n$-symmetric configurations of $N=2n$ vortices}

In analogy with Corollary \ref{cor:dihedral-osc}, we may  combine Theorem \ref{th:dihedral-poles}  
 with  Corollary \ref{cor:periodic-orbits-general} to establish the existence of three families 
of periodic orbits of the equations of motion \eqref{eq:motion}.
\begin{corollary} 
\label{cor:dihedral-osc-poles}
Let  $N=2n+2$.
\begin{enumerate}
\item For $n\geq 3$, there exists a 1-parameter family of periodic solutions $v_h(t)$ of  the equations of motion \eqref{eq:motion},  
emanating from the anti-prism with poles equilibrium configurations   described in  Theorem \ref{th:dihedral-poles}. Along these solutions, two vortices remain fixed at the north and south poles  and each remaining vortex travels around a small closed loop around a vertex 
of the $n$-gon anti-prism of height $2\hat z_a(n)$ 
(see Fig. \ref{F:antiprism-osc-poles}).

  \item For $n\geq 2$, there exists a 1-parameter family of periodic solutions $v_h(t)$ of  the equations of motion \eqref{eq:motion}
emanating from the polar collision  described in  Theorem \ref{th:dihedral-poles}.   Along these solutions, two vortices remain fixed at the north and south poles,  $n$ vortices travel along a closed
loop around the north pole and the remaining  $n$ vortices   travel along a closed
loop around the south pole in the opposite direction (see Fig. \ref{F:polar-oscillations-poles}). 

\item  For $n\geq 2$, there exists a 1-parameter family of periodic solutions $v_h(t)$ of  the equations of motion \eqref{eq:motion}
converging to the polygonal collisions with poles described in  Theorem \ref{th:dihedral-poles}.  Along these solutions,  two vortices remain fixed at the north and south poles and there is a pair of vortices that travels along a small closed
loop  around each of the vertices of the regular $n$-gon at the equator  (see Fig. \ref{F:polygonal-osc-poles}).
\end{enumerate}
Each of these families may be  parametrised by the energy $h$. In cases (ii) and (iii) we have 
 $h\to \infty$ as the solutions  approach collision, and  the period approaches zero in this limit.

For each solution described above, the distinct  closed loops traversed by the vortices, and the position the vortices within the loop at each instant, may be obtained from a single one by the action of $\mathbb{D}_n$.
\end{corollary}

\begin{figure}[ht]
\begin{subfigure}{.3\textwidth}
  \centering
  \includegraphics[width=.8\linewidth]{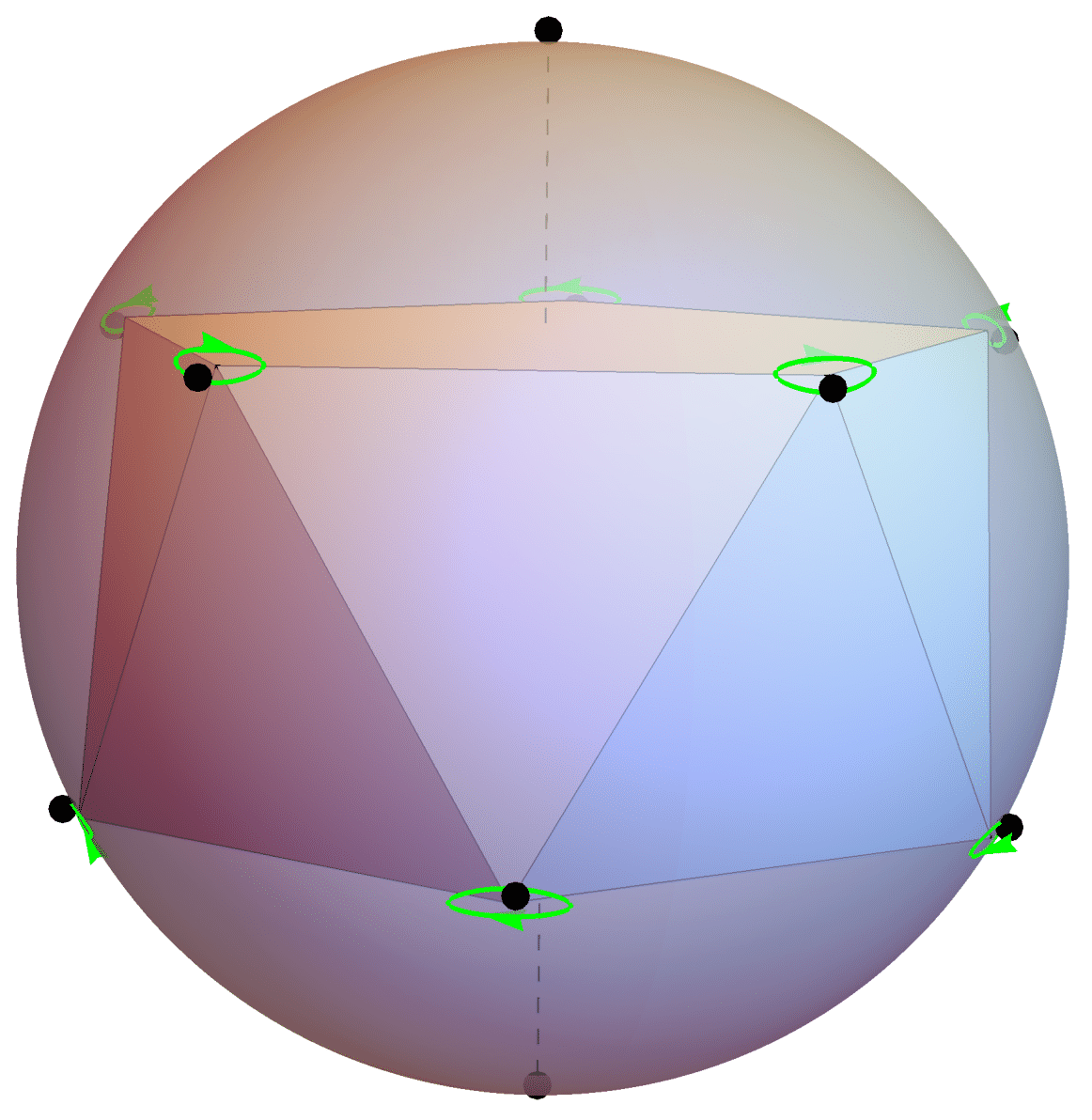}  
  \caption{Periodic solution near the anti-prism equilibrium with poles.}
  \label{F:antiprism-osc-poles}
\end{subfigure}
\quad
\begin{subfigure}{.3\textwidth}
 \centering
  \includegraphics[width=.7\linewidth]{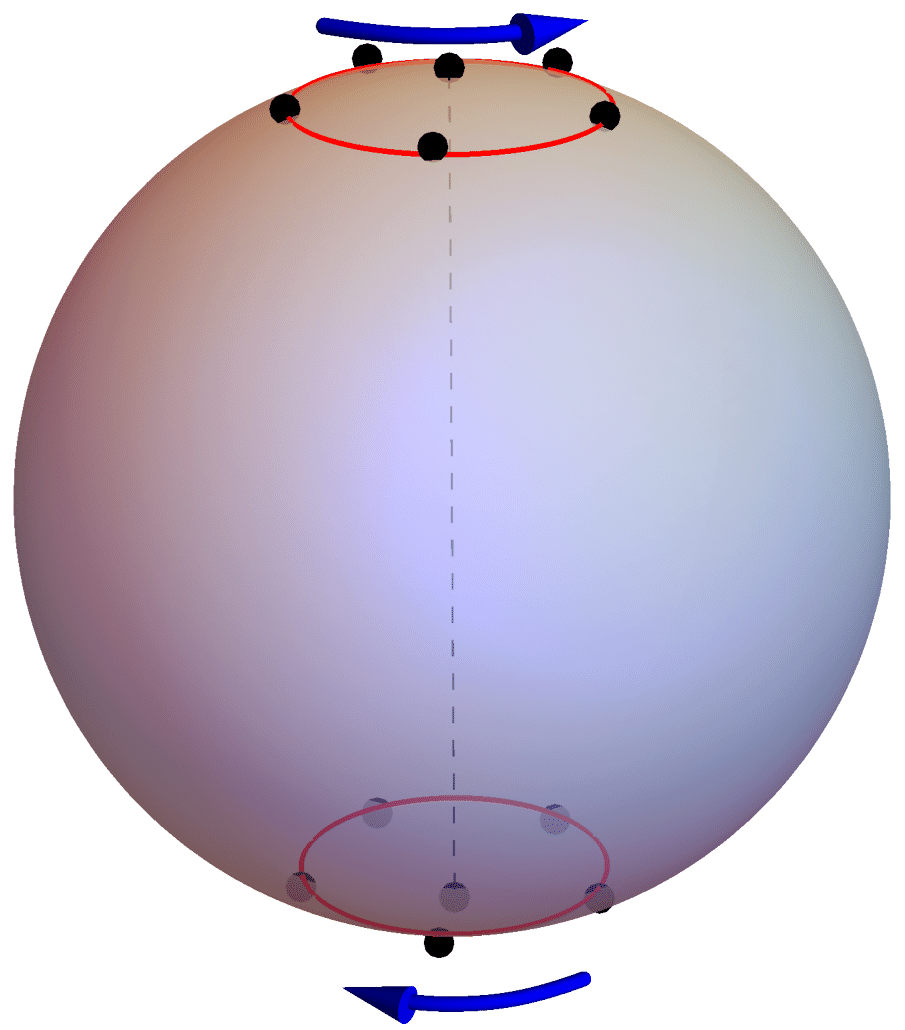}  
  \caption{Periodic solution near the
  polar-collision with poles.}
  \label{F:polar-oscillations-poles}
\end{subfigure}
\quad
\begin{subfigure}{.3\textwidth}
  \centering
  \includegraphics[width=.8\linewidth]{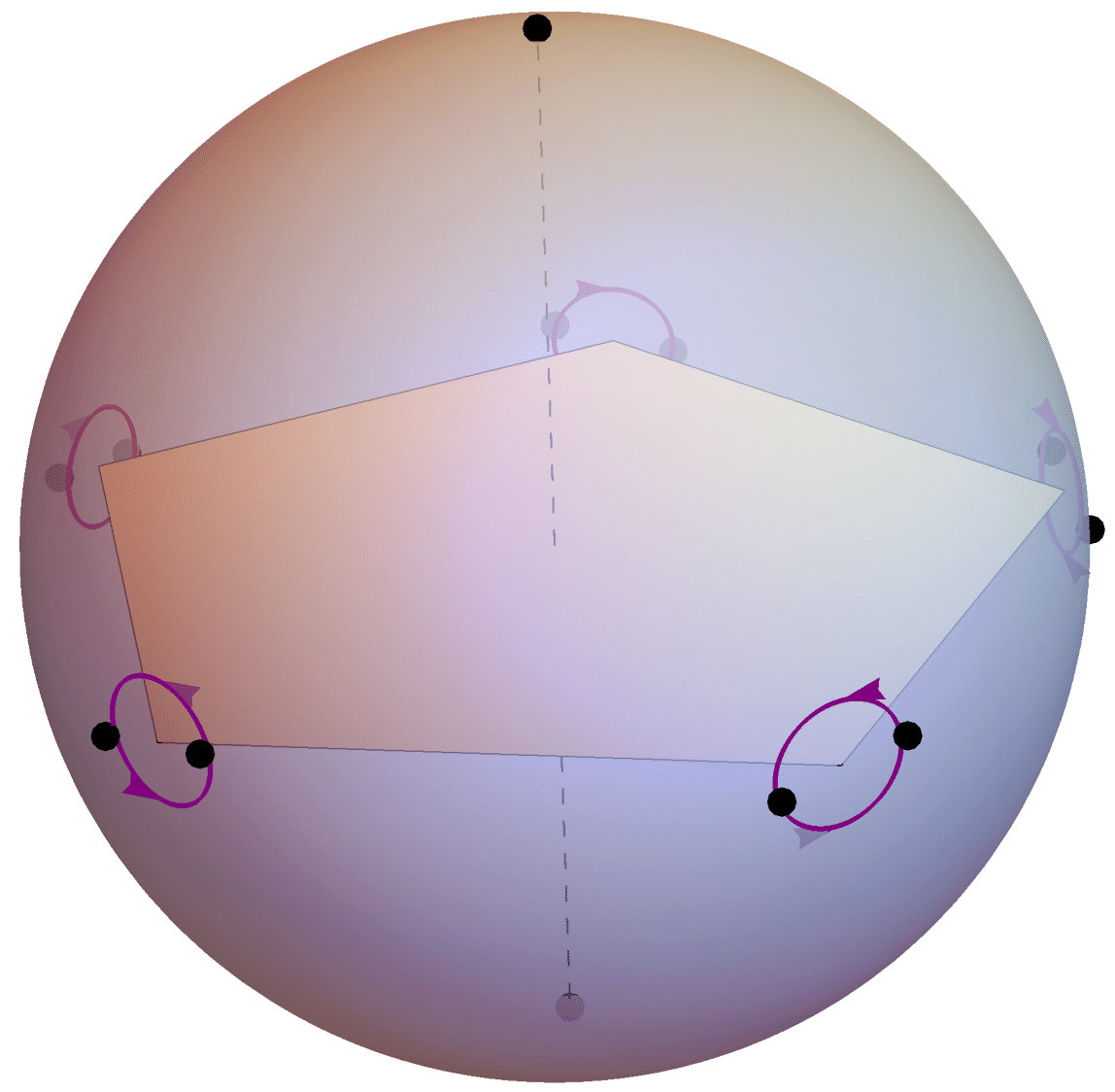}  
  \caption{Periodic solution near the
  polygonal-collision with poles.}
  \label{F:polygonal-osc-poles}
\end{subfigure}
\caption{Periodic solutions  described in Corollary \ref{cor:dihedral-osc-poles} for $n=5$, $N=12$.}
\label{fig:dihedral-cube-osc}
\end{figure}
 
We now specialise our discussion to the cases $n=2, 3, 5$ which lead to appearance of  platonic solids as either polygon with poles equilibria or 
as anti-prism with poles equilibria.

\subsubsection*{Case $n=2$, $N=6$. Nonlinear small oscillations around the octahedron.} 

In this case the polygonal with poles equilibrium configurations in Theorem \ref{th:dihedral-poles} are stable and correspond to octahedral 
equilibrium configurations of  \eqref{eq:motion}. A 1-parameter family of periodic solutions emanating from this configuration is 
established from Corollary \ref{cor:periodic-orbits-general}. Hence, we have a family of non-linear normal modes of oscillation with $\mathbb{D}_2$-symmetry around the octahedral
equilibrium. We emphasise  that this family is different from the one determined in the previous section
by looking at the octahedron as  the anti-prism equilibria with symmetry group  $\mathbb{D}_3$ and $F=\emptyset$.

On the other hand, according to Table \eqref{eq:table-D_n-zp-roots-w-poles} the prism with poles equilibrium configurations have height $1$. It is a simple
exercise to verify that the prisms degenerate and, together with the poles, form  hexagons which are contained on an equatorial plane.

The phase space of the (regularised) reduced dynamics obtained numerically is illustrated in Figure~\ref{fig:n6-D2} below. 
The polygonal with poles equilibrium points, $\hat Q_j$, corresponding to the octahedron configuration are indicated in green. The prism with poles 
equilibrium points, $\hat{P}_j^\pm$, corresponding to the hexagon configuration are illustrated in black. Finally, the polar collisions are red while the 
polygon with poles configurations are purple.  We have used
the same colour code to indicate either periodic orbits near the stable equilibria or heteroclinic orbits emanating from the unstable equilibria.

Finally, we note that the subset $F=\{(0,0,\pm 1)\}$ is not 
invariant under the action of $\mathbb{O}= N(\mathbb{D}_2)$. However, $F$ is invariant under $ \mathbb{D}_4$ and moreover, 
$\mathbb{D}_2<\mathbb{D}_4<\mathbb{O}=N(\mathbb{D}_2)$. So,
as  predicted by item (iii) of Theorem~\ref{th-main-symmetry}, 
 we observe a $\mathbb{D}_4$  symmetry in the reduced dynamics.

\subsubsection*{Case $n=3$, $N=8$. Nonlinear small oscillations around the cube.} 

For $n=3$, Table \eqref{eq:table-D_n-za-roots-w-poles}  indicates that  the height of the anti-prism  is  $2/3$. One may  verify 
that the resulting anti-prism with poles is in fact a cube whose edges have length $2/\sqrt{3}$. Despite the instability of these configurations
as equilibria of \eqref{eq:motion}, we conclude from item (i) of Corollary \ref{cor:dihedral-osc-poles} that there is a family of nonlinear small oscillations
emanating from these configurations.

The phase space of the (regularised) reduced dynamics obtained numerically is illustrated in Figure~\ref{fig:n8-D3} 
below and the colour code is similar to the
one used in Figures \ref{fig:n6-D3} and \ref{fig:n8-D4}.
The anti-prism equilibrium points $\hat{A}_j^{\pm}$ are indicated in green,
the prism equilibrium points $\hat{P}_j^\pm$ in blue, polygonal  equilibrium points  $\hat{Q}_j$ in black, polar  collisions in red and polygonal collisions  $\hat{C}_j$ in purple. 
The same colour is used to indicate either periodic orbits near the stable equilibria or heteroclinic/homoclinic orbits emanating from the unstable equilibria.
There is also a family  of periodic orbits that do not approach an equilibria or a collision that we have indicated in orange. 
Considering that the set $F=\{(0,0,\pm 1)\}$ is 
invariant under the action of $\mathbb{D}_6$, then, as  predicted by item (iii) of Theorem~\ref{th-main-symmetry}, 
 we observe a $\mathbb{D}_6$-symmetry in the reduced dynamics.

\subsubsection*{Case $n=5$, $N=12$.  Nonlinear small oscillations around the icosahedron.} 

For $n=5$, we read from Table \eqref{eq:table-D_n-zp-roots}  that  the height of the anti-prism with poles configuration is  $2/\sqrt{5}$ and one may show that this corresponds to an
inscribed icosahedron whose edges have  length $\sqrt{2-\frac{2}{\sqrt{5}}}$.  Figure \ref{F:anti-prism-poles} illustrates this. By connecting the north and south pole with each one of the vertices
on the top and bottom faces of the anti-prism we get an icosahedron. These configurations are known \cite{K} to be stable equilibria of  \eqref{eq:motion} and item (i) of  Corollary  \ref{cor:dihedral-osc-poles} 
shows the existence of a family of small oscillations
emanating from them.
The phase space of the (regularised) reduced dynamics obtained numerically is illustrated in Figure~\ref{fig:n12-D5} below. The colour 
code is identical to the one followed in the case $n=3$ described above. This time  we observe $\mathbb{D}_{10}$-symmetry in the reduced dynamics.

\begin{figure}[ht]
\begin{subfigure}{.32\textwidth}
  \centering
  \includegraphics[width=.8\linewidth]{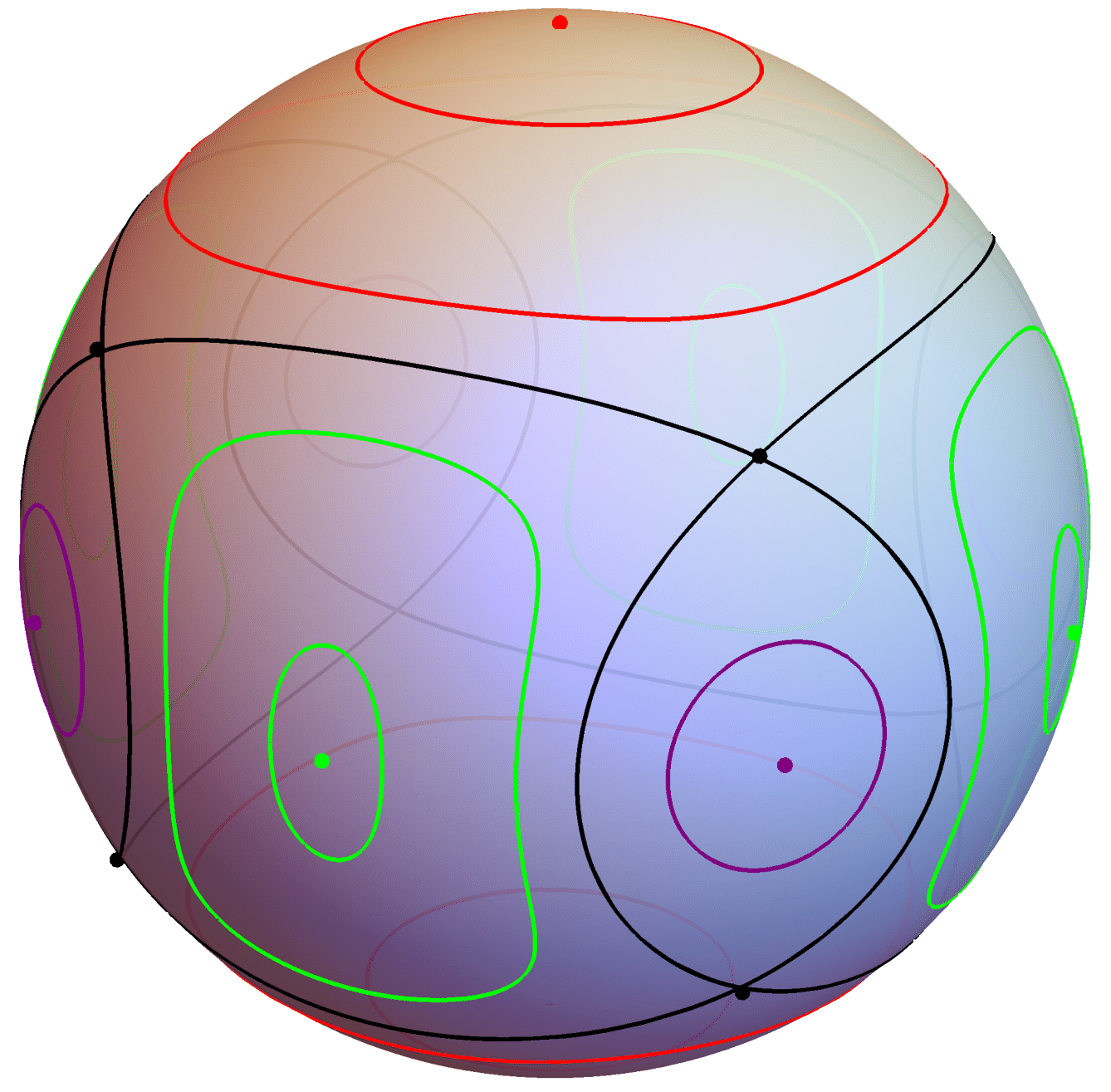}  
  \caption{$n=2$, $N=6$.}
  \label{fig:n6-D2}
\end{subfigure}
\;
\begin{subfigure}{.32\textwidth}
 \centering
  \includegraphics[width=.8\linewidth]{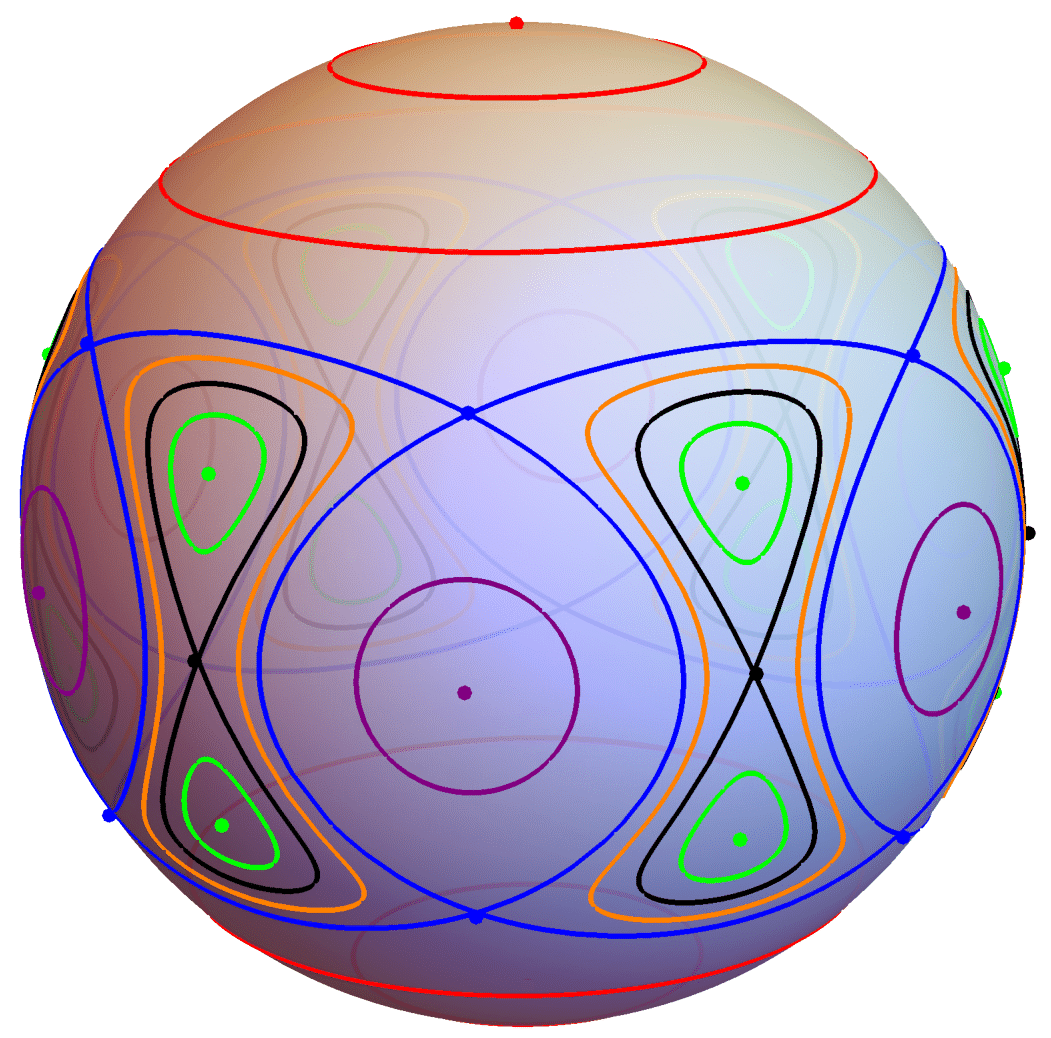}  
  \caption{$n=3$, $N=8$.}
  \label{fig:n8-D3}
\end{subfigure}
\;
\begin{subfigure}{.32\textwidth}
  \centering
  \includegraphics[width=.8\linewidth]{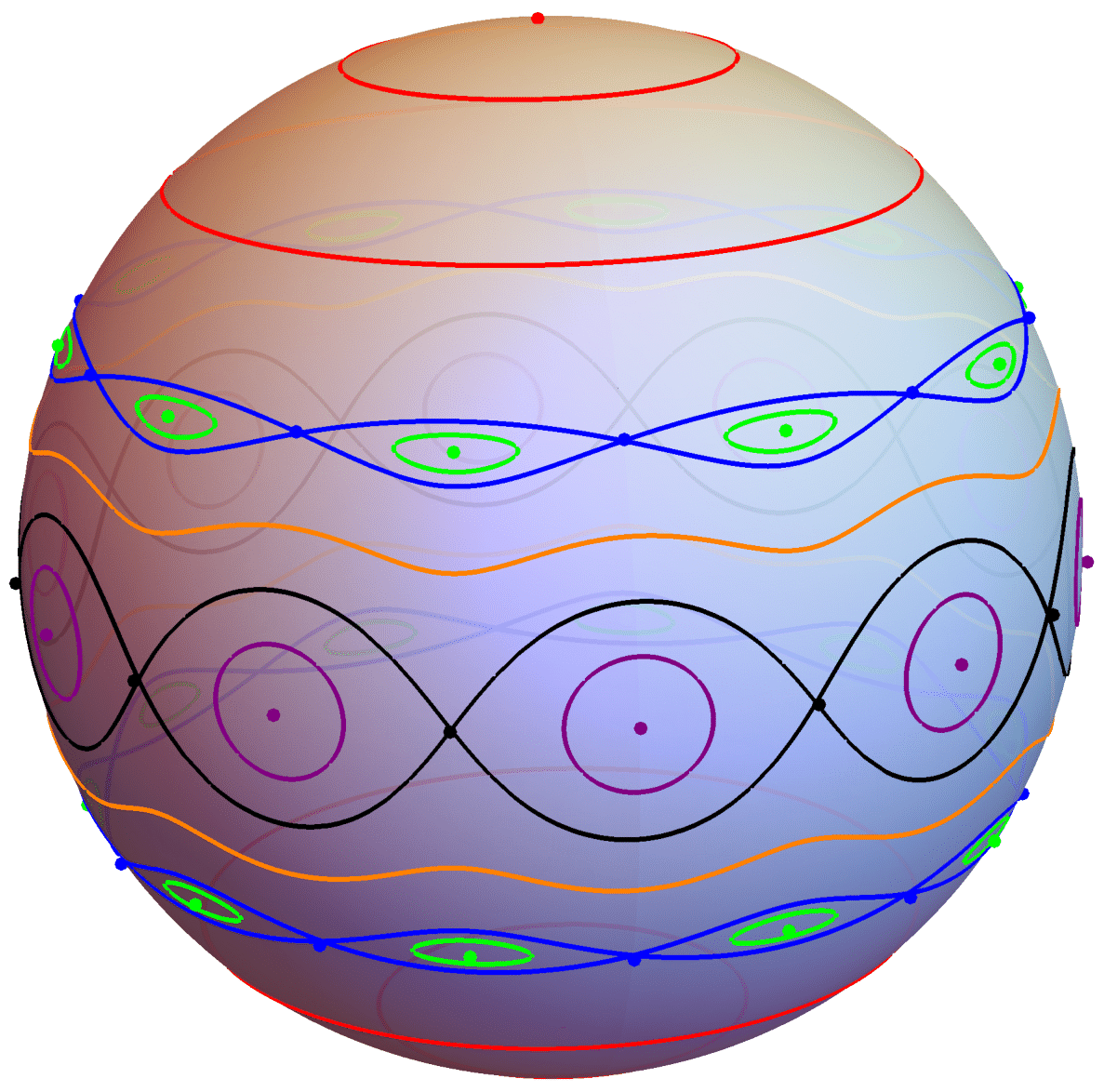}  
  \caption{$n=5$, $N=12$.}
  \label{fig:n12-D5}
\end{subfigure}
\caption{Phase space of the (regularised) reduced system \eqref{eq:reduced-system} for $K=\mathbb{D}_n$, $F=\{(0,0,\pm 1)\}$ and $N=2n+2$ for  $N=2, 3, 5$. See text for explanations and description of  the colour code. }
\label{fig:Dn-F-poles}
\end{figure}

\section{$\mathbb{T}$-symmetric solutions for $N=12$ vortices (with no fixed vortices)}
\label{sec:Tet}

We consider the tetrahedral subgroup $\mathbb{T}<\SO(3)$ generated by the matrices 
\begin{equation*}
\begin{pmatrix} 0 & 1 & 0 \\ 0 & 0  & 1 \\ 1 & 0 & 0  \end{pmatrix} \qquad \mbox{and} \qquad \begin{pmatrix}  1 & 0 & 0 \\ 0 & -1  & 0 \\ 0 & 0 & -1  \end{pmatrix}.
\end{equation*}
 Then $\mathbb{T}$ has order 12 and is isomorphic to the subgroup $A_4$ of even permutations of $4$ elements. An explicit group isomorphism may be defined in terms of the above generators as 
\begin{equation*}
\begin{pmatrix} 0 & 1 & 0 \\ 0 & 0  & 1 \\ 1 & 0 & 0  \end{pmatrix} \mapsto (1,2,3)(4), \qquad \begin{pmatrix} 0 & 1 & 0 \\ 0 & 0  & 1 \\ 1 & 0 & 0  \end{pmatrix} \mapsto (1,2)(3,4),
\end{equation*}
where we have used the standard cyclic notation for permutations. The group $\mathbb{T}$ consists of the orientation preserving symmetries of the tetrahedra $\mathcal{T}_1, \mathcal{T}_2$ with vertices at 
\begin{equation}
\label{eq:tetrahedra}
\begin{split}
\mathcal{T}_1=\{ c(1,1,1),\,  c(-1,-1,1), \, c(-1,1,-1),\, c(1,-1,-1) \},  \\
\mathcal{T}_2=\{ c(-1,-1,-1),\,  c(1,1,-1), \, c(1,-1,1),\, c(-1,1,1) \},
\end{split}
\end{equation}
where $c^{-1}=\sqrt{3}$. One may check that
\begin{equation}
\label{eq:F[T]}
\mathcal{F}[\mathbb{T}]=\mathcal{T}_1\cup \mathcal{T}_2 \cup \{(\pm 1,0,0), ( 0,\pm1,0),  ( 0,0,\pm1)\}.
\end{equation}

\subsection{Classification of $\mathbb{T}$-symmetric equilibrium configurations of $N=12$ vortices  }

We now analyse the  reduced system  \eqref{eq:reduced-system}  in detail in the case  $K=\mathbb{T}$ and $F=\emptyset$ so $N=12$.
Item (iii) of  Theorem \ref{th-main-symmetry} and Table \eqref{eq:table-normalizers} indicate that such
  system is $\mathbb{O}$-equivariant.
The theorem below gives the full classification of the collision and non-collision equilibria.

\begin{theorem}
\label{th:tetrahedral}
Let  $K=\mathbb{T}$,  $F=\emptyset$  and $N=12$. 
The classification and stability of the  equilibrium points of the reduced system  \eqref{eq:reduced-system} is 
as follows.
\begin{enumerate}
\item The only non-collision equilibria of  \eqref{eq:reduced-system} are: 
\begin{enumerate}

\item The   \defn{icosahedron  equilibrium configurations} occurring at all $24$ points obtained by permuting the entries and considering all  sign flips of 
\begin{equation*}
 \frac{1}{\sqrt{1+\phi^2}} \left (\pm \phi, \pm 1, 0 \right ),
\end{equation*} 
where $\phi=\frac{1+\sqrt{5}}{2}$ is the golden mean.
These are {\em stable} equilibria of  \eqref{eq:reduced-system}
which  correspond to equilibrium configurations of \eqref{eq:motion} where the vortices occupy  the 
vertices of an $S^2$-inscribed regular icosahedron    (see Fig.\ref{F:icosahedron}).

\item The   \defn{truncated tetrahedron configurations} occurring at all  $24$ points obtained by permuting the entries and considering all  sign flips of 
\begin{equation*}
\left (\pm \alpha , \pm \alpha, \pm \sqrt{1-2\alpha^2}\right ),
\end{equation*} 
where $0<\alpha\approx 0.269484\dots $ is characterised by the condition that $\alpha^2$ is the unique
zero  of the polynomial $p(\lambda)= 1-13\lambda -13\lambda^2+33\lambda^3$ between $0$ and $1/2$.
These are {\em unstable} equilibria of  \eqref{eq:reduced-system}
which  correspond to equilibrium configurations of \eqref{eq:motion} where the vortices occupy  the 
vertices of an irregular $S^2$-inscribed truncated tetrahedron    (see Fig.\ref{F:trunc-tetrahedron} and Remark \ref{rmk:dimensions-truc-tet}).

\item The   \defn{cub-octahedron configurations} occurring at all $12$ points obtained by permuting the entries and considering all  sign flips of 
\begin{equation*}
\left (\pm \frac{1}{\sqrt{2}}, \pm \frac{1}{\sqrt{2}}, 0 \right ).
\end{equation*}
These are {\em unstable} equilibria of  \eqref{eq:reduced-system}
which  correspond to equilibrium configurations of \eqref{eq:motion} where the vortices occupy  the 
vertices of a regular $S^2$-inscribed cuboctahedron    (see Fig.\ref{F:cuboctahedron}). 
\end{enumerate}

\item The only  collision equilibria of (the regularisation of)  \eqref{eq:reduced-system} are:
\begin{enumerate}
\item The \defn{tetrahedral collisions}  at the 8 points of  $\mathcal{T}_1\cup \mathcal{T}_2$. These correspond to collision configurations of \eqref{eq:motion} having
four simultaneous triple collisions at the vertices of a tetrahedron (see
Fig.\ref{F:tetrahedron-coll}).
\item The \defn{octahedral collisions} at  the $6$ points $(\pm 1,0,0), ( 0,\pm1,0),  ( 0,0,\pm1)$.
 These correspond to collision configurations of \eqref{eq:motion} having
$6$  simultaneous binary collisions at the vertices of an octahedron   (see
Fig.\ref{F:octahedron-coll}).  
\end{enumerate}
All collision configurations are stable equilibria of  (the regularisation of)  \eqref{eq:reduced-system}.
\end{enumerate}
\end{theorem}
\begin{remark}
\label{rmk:dimensions-truc-tet}
The polyhedron determined by the 
 truncated tetrahedron equilibria  consists of  4 irregular hexagonal faces and 4 four equilateral triangular faces. 
The distance between the vertices 
   forming an edge between adjacent hexagonal faces is $2\sqrt{2}\alpha\approx0.762215$ while the distance between vertices 
   forming an edge of an equilateral triangular  face is $\sqrt{2}(-\alpha +\sqrt{1-2\alpha^2})\approx 0.926377$.
\end{remark}

\begin{figure}[ht]
\begin{subfigure}{.30\textwidth}
  \centering
  \includegraphics[width=.7\linewidth]{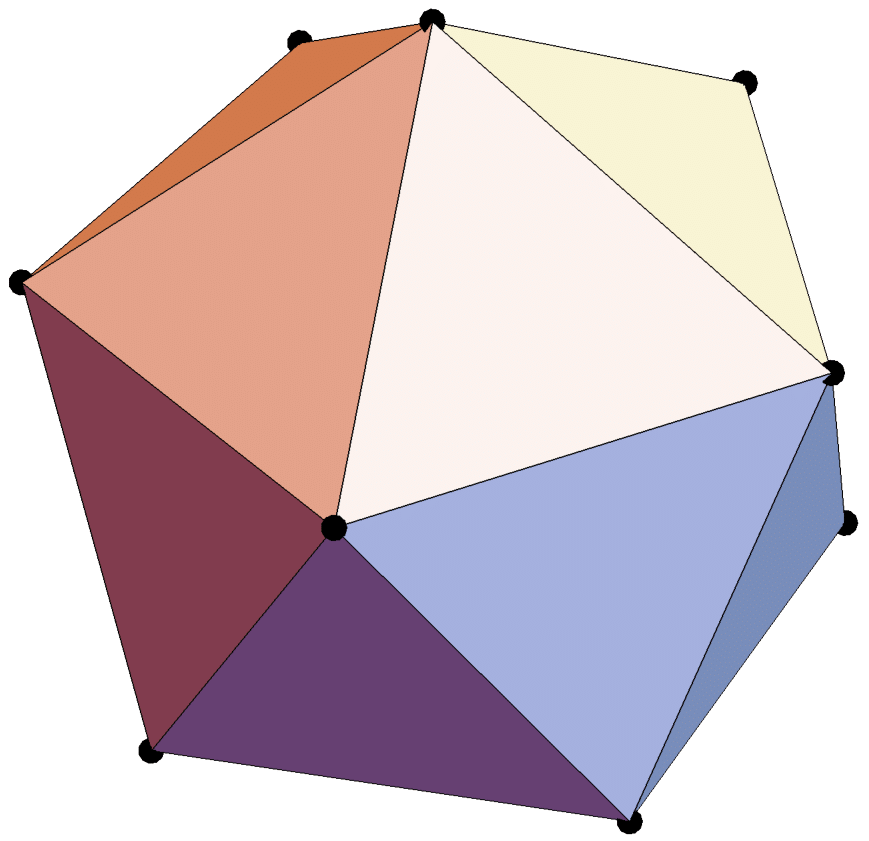}  
  \caption{Icosahedron  equilibrium. The distance between adjacent vertices is $2/\sqrt{3}$.}
  \label{F:icosahedron}
\end{subfigure}
\quad
\begin{subfigure}{.30\textwidth}
  \centering
   \includegraphics[width=.65\linewidth]{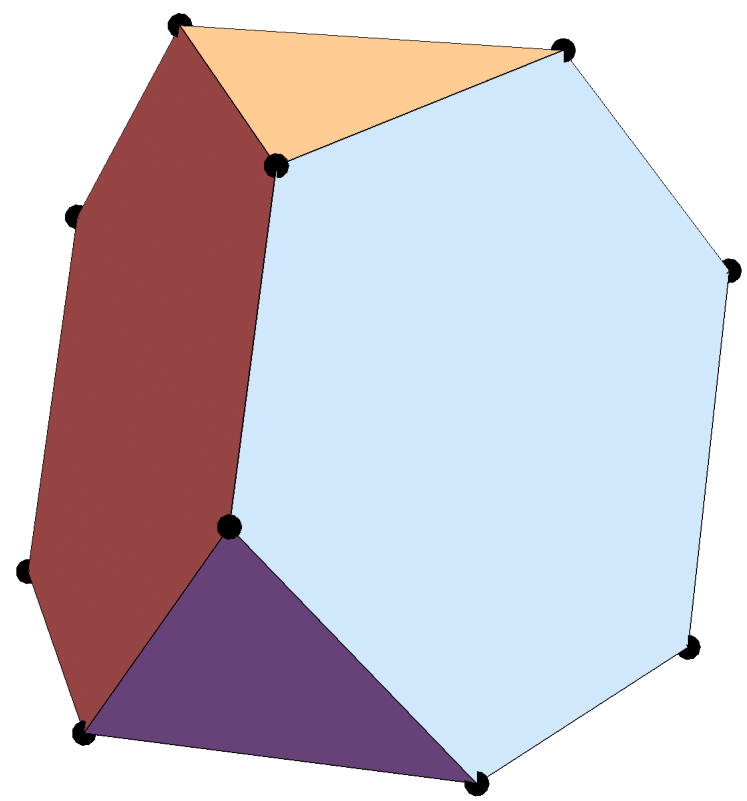}  
  \caption{Truncated tetrahedron equilibrium. See Remark \ref{rmk:dimensions-truc-tet} for details on the dimensions.}
  \label{F:trunc-tetrahedron}
\end{subfigure}
\quad 
\begin{subfigure}{.30\textwidth}
  \centering
 \includegraphics[width=.7\linewidth]{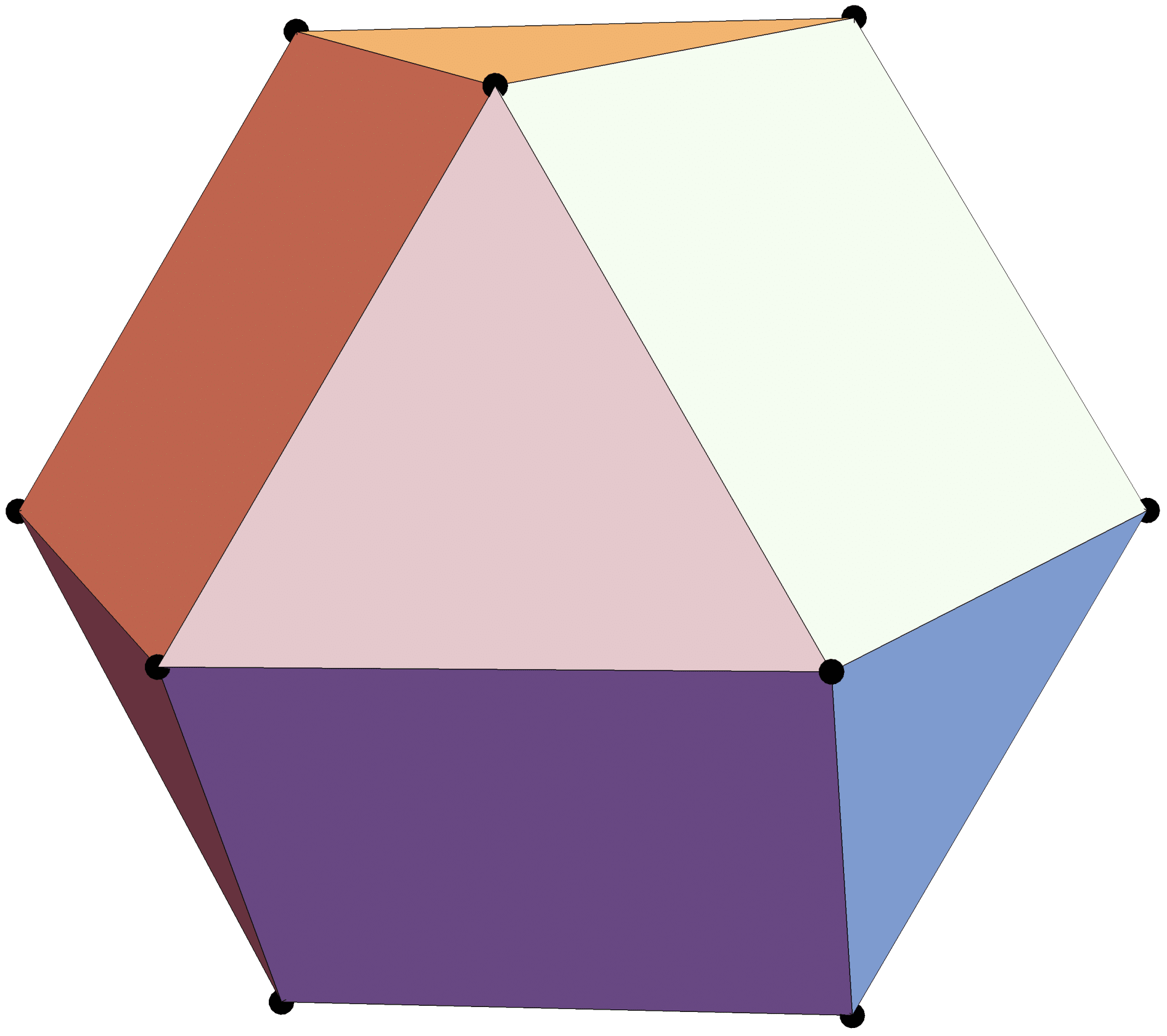}  
  \caption{Cuboctahedron equilibrium.  The distance between adjacent vertices is 1.}
  \label{F:cuboctahedron}
\end{subfigure}
 \\
\begin{subfigure}{.45\textwidth}
  \centering
  \includegraphics[width=.4\linewidth]{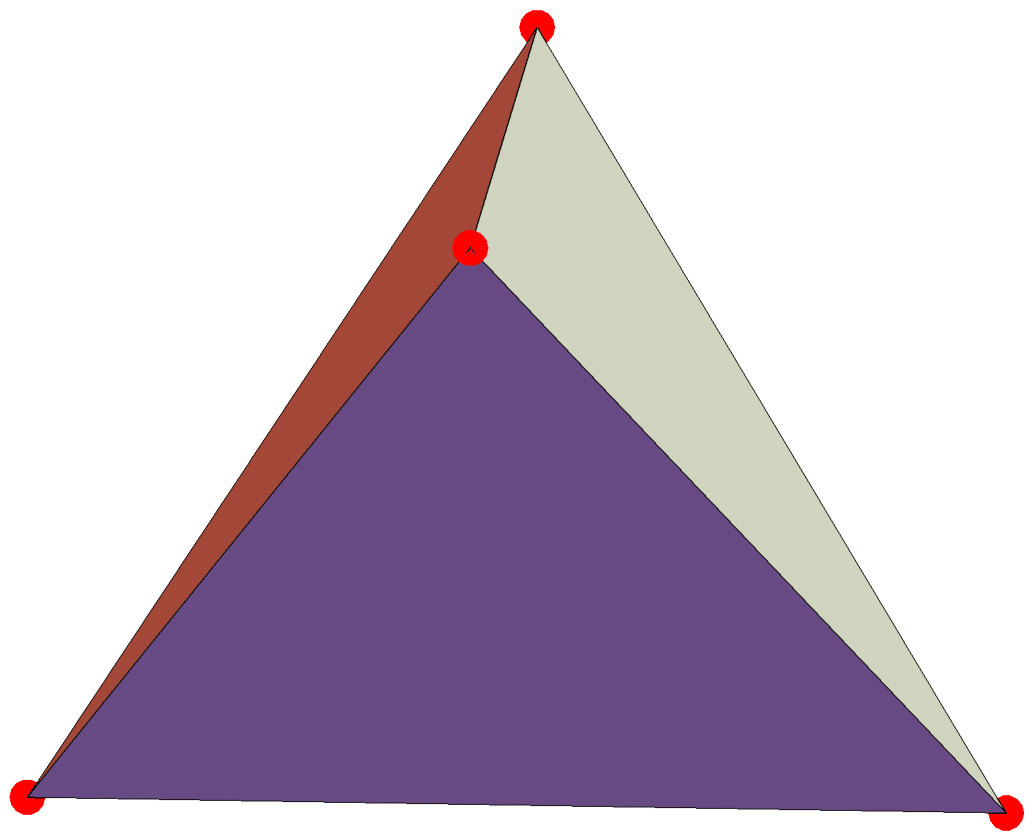}  
  \caption{Tetrahedron collision (simultaneous triple collision  at each vertex).}
  \label{F:tetrahedron-coll}
\end{subfigure} 
\quad
\begin{subfigure}{.45\textwidth}
  \centering
  \includegraphics[width=.4\linewidth]{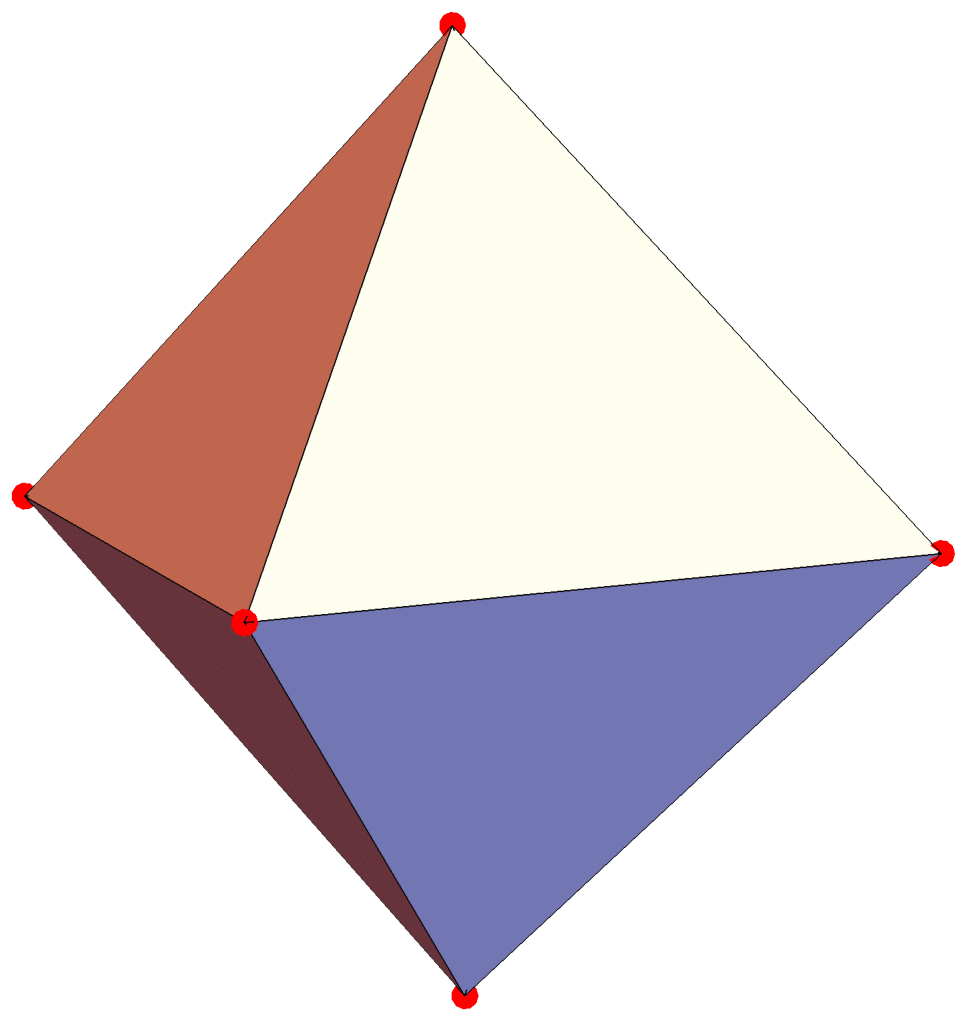}  
  \caption{Octahedron collision (simultaneous binary collision  at each vertex).}
  \label{F:octahedron-coll}
\end{subfigure} 
\caption{Non-collision and collision equilibrium configurations described in  Theorem \ref{th:tetrahedral}.}
\label{fig:n12-eq}
\end{figure}

Before giving the proof of the theorem we present:
\begin{lemma}
\label{lemma:Tet}
Let
\begin{equation*}
\begin{split}
   p_1^\pm(R,\Theta)&=2+2R^2 +R^2 \sin 2 \Theta  \pm  2 R( \cos \Theta - \sin \Theta  ) ,\\
 p_2^\pm  (R,\Theta) &= 2+2R^2 -R^2 \sin 2 \Theta \pm  2 R(  \cos \Theta + \sin \Theta  ), \\
  p_3(R,\Theta)&=   5 R^8 \cos 8 \Theta +76 \left(R^4+8 R^2+8\right) R^4 \cos 4 \Theta +47 R^8-864 R^6-4320 R^4-6912
   R^2-3456.
  \end{split}
\end{equation*}
For $(R,\Theta)\in [0,\sqrt{2}]\times (0,\pi/4)$ we have  $p_1^\pm(R,\Theta)>0$,  $p_2^\pm(R,\Theta)>0$ and $p_3(R,\Theta)<0$.
\end{lemma}
\begin{proof}
For $\Theta \in (0,\pi/4)$ we have $\sin 2\Theta \in (0,1)$, $ \cos \Theta - \sin \Theta\in (0,1)$ and  $ \cos \Theta + \sin \Theta\in (1,\sqrt{2})$. Therefore,
\begin{equation*}
\begin{split}
 p_1^\pm(R,\Theta)&> 2+2R^2- 2R( \cos \Theta - \sin \Theta  )> 2(R^2-R+1)\geq 3/2,\\
  p_2^\pm(R,\Theta)&> 2+R^2- 2R( \cos \Theta + \sin \Theta  )> R^2-2\sqrt{2}R+2\geq0.
 \end{split}
\end{equation*}
On the other hand, we have
\begin{equation*}
\begin{split}
  p_3(R,\Theta) &\leq   5 R^8  +76 \left(R^4+8 R^2+8\right) R^4 +47 R^8-864 R^6-4320 R^4-6912
   R^2-3456 \\
   &=128R^8-256R^6-3712R^4-6912
   R^2-3456\\
    &\leq 128R^8-3456.
   \end{split}
\end{equation*}
Therefore, for  $R\in [0,\sqrt{2}]$ we may estimate $  p_3(R,\Theta) \leq 128(2^4)-3456=-1408$.
\end{proof}

\begin{proof}[Proof of Theorem  \ref{th:tetrahedral}] 
Recall that the collision equilibria are always stable and   occur at the points in
 $\mathcal{F}[\mathbb{T}]$. This set   is given by   \eqref{eq:F[T]} and  consists of three $\mathbb{T}$-orbits: the tetrahedra $\mathcal{T}_1$ and  $\mathcal{T}_2$, and 
  $\{ (\pm 1,0,0), ( 0,\pm1,0),$  $( 0,0,\pm1)\}$. The points on the latter orbit lie on the vertices of an octahedron. The proof of item (ii) in the theorem
  follows from these observations and item (ii) of 
  Proposition \ref{prop:collision-description}.
 
 In order to prove item (i), we will classify the critical points of the regularised reduced Hamiltonian $\tilde{h}_{(\mathbb{T},\emptyset)}$. 
 Using \eqref{eq:regul-red-Ham} and writing $u=(x,y,z)$ we find $\tilde{h}_{(\mathbb{T},\emptyset)}(x,y,z)=2^{84}a(x,y,z)^6$, where $a:S^2\to \R$ is given by
 \begin{equation*}
 \begin{split}
a(x,y,z)&= (x^2+y^2) (y^2+z^2)   (x^2+z^2) (1-xy-xz-yz)^2 (1+xy+xz-yz)^2  \\ &  \qquad \cdot (1-xy+xz+yz)^{2}(1+xy-xz+yz)^{2}.
\end{split}
\end{equation*}
The critical points of $\tilde{h}_{(\mathbb{T},\emptyset)}$ and $a$ coincide and are of the same type, so, in what follows, we 
instead classify the critical points of $a$. 

We  begin by noting that the value of $a$ does not change if $x$, $y$ and $z$ are permuted; and also if 
any of $x$,  $y$ or $z$ are changed into $-x$,  $-y$ or $-z$. This shows that $a$ is invariant under the action of the group $\mathbb{O}_h$   
consisting of all rotational and reflectional symmetries of a regular octahedron
(the   $\mathbb{O}$-symmetry of $a$ was expected from item (iii) of Theorem \ref{th-main-symmetry} and the reflectional part is
 inherited from the invariance of the
Hamiltonian $H:M\to \R$ under the diagonal 
action of $\mathrm{O}(3)$).

 The group $\mathbb{O}_h$ has order  48 and a fundamental region $\mathcal{R}\subset S^2$ is determined by $0\leq y\leq x \leq z \leq 1$.
 Without loss of generality, we will restrict the analysis of the critical points of $a$ to this region.
Our strategy is to introduce local coordinates on $S^2$ tracking carefully the parametrisation of $\mathcal{R}$. First consider the gnomonic (or stereographic) 
projection from the origin to the tangent plane 
to the north pole. This defines the coordinates $(X,Y)\in \R^2$
on the northern hemisphere by 
\begin{equation*}
x=\frac{X}{\sqrt{X^2+Y^2+1}}, \quad y=\frac{Y}{\sqrt{X^2+Y^2+1}}, \quad z=\frac{1}{\sqrt{X^2+Y^2+1}},
\end{equation*}
and the  fundamental region  $\mathcal{R}$ corresponds to  the triangle $0\leq Y\leq X\leq 1$. Now pass to polar coordinates
\begin{equation*}
X=R\cos \Theta, \qquad Y=R\sin \Theta.
\end{equation*}
The northern hemisphere is parametrised by $R\geq 0$ and $\Theta\in [0,2\pi)$, and the fundamental region $\mathcal{R}$ corresponds to
\begin{equation}
\label{eq:fund-region-polar}
0\leq R\leq \frac{1}{\cos\Theta}, \qquad 0\leq \Theta \leq \pi/4.
\end{equation}
In particular, it will be convenient to notice that $\mathcal{R}$ is contained in the region parametrised  by $(R,\Theta)\in [0,\sqrt{2}]\times [0,\pi/4]$. In 
these coordinates we have:

\begin{equation*}
\begin{split}
a(R,\Theta)=\frac{R^2 \left(R^2+2- R^2 \cos 2 \Theta \right) 
\left(R^2+2+R^2 \cos 2 \Theta \right) 
\left ( p_1^+(R,\Theta) p_1^-(R,\Theta) p_2^+(R,\Theta)
p_2^-(R,\Theta) \right )^2
}{1024 \left(R^2+1\right)^{11}},     
 \end{split}
\end{equation*}
where $p_1^\pm$, and  $p_2^\pm$ are defined in the statement of Lemma \ref{lemma:Tet}. 
With the help of a symbolic algebra software, one finds that the partial derivative $\partial_\Theta a(R,\Theta)$ may be written as
\begin{equation*}
\partial_\Theta a(R,\Theta)=\frac{R^6\sin 4\Theta}{4096 \left(R^2+1\right)^{11}}  p_1^+(R,\Theta)p_1^-(R,\Theta)p_2^+(R,\Theta)
p_2^-(R,\Theta)p_3(R,\Theta),
\end{equation*}
with  $p_3$ given in the statement of Lemma \ref{lemma:Tet}. Because of this lemma we conclude that, when
restricted to the fundamental region $\mathcal{R}$, the partial derivative $\partial_\Theta a(R,\Theta)$  can only vanish if $\Theta=0$ or $\Theta=\pi/4$.

Now, on the one hand one computes
\begin{equation*}
\begin{split}
\partial_R a(R,0)&=- 2 R \left(R^2+1\right)^{-11} \left(R^4+R^2+1\right)^3 \left(  R^{2}-1\right)  \left(  R^{4}-3R^{2}+1\right) , \end{split}
\end{equation*}
whose real roots are $R=0$, $R=\pm 1$ and $R=(\pm 1\pm\sqrt{5})/2$. Hence, in view of \eqref{eq:fund-region-polar}, for $\Theta=0$, the only critical points of $a$ on the fundamental region
$\mathcal{R}$ occur when $R=0$, $R=(\sqrt{5}-1)/2$ and $R=1$. These respectively correspond to the following points  on $S^2$:
\begin{equation*}
(0,0,1), \qquad  \frac{1}{\sqrt{1+\phi^2}} \left (1,0, \phi \right ), \qquad   \frac{1}{\sqrt{2}} \left (1,0, 1 \right ),
\end{equation*}
that are, respectively, representatives of the octahedral collisions of (ii)(b), of the icosahedral equilibria of (i)(a) and the
 cuboctahedron equilibria of (i)(c). 

On the other hand, one finds
\begin{equation}
\label{eq:aux-th-tet-symm}
\begin{split}
\partial_R a(R,\pi/4)=\frac{R}{512}\left (R^2+1\right)^{-12} \left(3 R^2+2\right)^3 \left(R^2-2 \right)^2 \left(R^4-4\right) q_3(R^2),
    \end{split}
\end{equation}
where $q_3$ is the cubic polynomial $q_3(\lambda)=37
   \lambda^{3}+106 \lambda^2+28 \lambda-8$. This  polynomial  has a unique positive root $\lambda=2\alpha^2/(1-2\alpha^2)$ with 
   $\alpha$ as defined in the statement of item (i)(b).  Therefore, the only real roots of the right hand side of \eqref{eq:aux-th-tet-symm}
are $R=0$,  $R=\pm \sqrt{2}$ and $R=\pm \sqrt{2}\alpha/(1-2\alpha^2)^{1/2}$, and, in view of  \eqref{eq:fund-region-polar}, we conclude that for 
$\Theta=\pi/4$, the only critical points of $a$ on the fundamental region $\mathcal{R}$ occur when $R=0$, $R= \sqrt{2}\alpha/(1-2\alpha^2)^{1/2}$ and $R=\sqrt{2}$. 
 These respectively correspond to the following points  on $S^2$:
\begin{equation*}
(0,0,1), \qquad  \left ( \alpha ,  \alpha,  \sqrt{1-2\alpha^2}\right ),  \qquad \frac{1}{\sqrt{3}} \left (1,1,1\right ),
\end{equation*}
that are, respectively, representatives of the octahedral collisions of (ii)(b), of the irregular truncated tetrahedron equilibria of (i)(b) and the
 tetrahedron collisions of (ii)(a). 

The analysis above proves that, indeed, the only equilibrium points of the (regularised) system are those described in the statement of the theorem. Now
recall that the collision equilibria are always stable. To investigate
the stability of the non-collision equilibria we compute the Hessian matrix of $a$ at the representatives of these points. Using a symbolic algebra program one obtains: 
\begin{equation*}
\begin{split}
&\mbox{Hess}(a)\left (\frac{\sqrt{5}-1}{2},0 \right ) =\frac{128}{3125} \begin{pmatrix} -3-\sqrt{5} & 0 \\ 0 & \frac{17}{5}(\sqrt{5}-5) \end{pmatrix}, \qquad
\mbox{Hess}(a)\left (1,0 \right ) =\frac{27}{512} \begin{pmatrix} 1 & 0 \\ 0 & - \frac{37}{2} \end{pmatrix}, \\
&\mbox{Hess}(a)\left ( \frac{\sqrt{2}\alpha}{\sqrt{1-2\alpha^2}},\pi /4 \right ) \approx \begin{pmatrix} -1.42703 & 0 \\ 0 & 0.0859734 \end{pmatrix}.
\end{split}
\end{equation*}
The first of these matrices is negative definite and  the other two are indefinite.  We conclude that icosahedral equilibria are local maxima of $a$,
whereas cub-octahedral and truncated tetrahedral equilibria are saddle points of $a$. The same is true for the 
regularised reduced  Hamiltonian $\tilde h_{(\mathbb{T},\emptyset)}$. On the other hand, this implies that the (non-regularised) reduced Hamiltonian
$h_{(\mathbb{T},\emptyset)}$ has local minima at the icosahedral equilibria and saddle points at the cub-octahedral and truncated tetrahedral equilibria.
In view of item (i) of Proposition \ref{prop:reduced-periodic}, these observations imply that the stability properties described in the theorem hold.

It remains to show that  the non-collision equilibrium points in items (i)(a)-(c) indeed correspond
to the polyhedron equilibria of \eqref{eq:motion} described in the statement of the theorem. 
Let  $\gamma=\phi/(1+\phi^2)^{1/2}$. The 24 points obtained by permuting the entries of $(\pm \gamma, \pm \sqrt{1-\gamma^2},0)$ lie at the vertices of a compound of two icosahedra. One of them
corresponds to the even and the other to the odd permutations.  Moreover, the vertices of each of these icosahedra lie on a $\mathbb{T}$-orbit. In particular, the embedding $\rho_{(\mathbb{T},\emptyset)}$
maps any of the 24 points into the vertices of a regular icosahedron.
A similar scenario occurs for the set   obtained by  permuting the entries of $(\pm \alpha, \pm \alpha, \pm \sqrt{1-2\alpha^2})$. 
Such set has 24 elements and consists of two  $\mathbb{T}$-orbits according to whether the product of the entries is  positive or negative. Each of these orbits
determine the vertices of a (irregular) truncated tetrahedron. 
The situation is simpler for the points obtained by permuting $1/\sqrt{2}(\pm 1, \pm 1, 0)$ since there are only 12 of them, they lie on a $\mathbb{T}$-orbit and lie on the vertices
of a cuboctahedron.
\end{proof}

\subsection{Dynamics of $\mathbb{T}$-symmetric configurations of $N=12$ vortices}

In view of Theorem \ref{th:tetrahedral} and  Corollary \ref{cor:periodic-orbits-general} 
we deduce the existence of  three families 
of periodic orbits of the equations of motion \eqref{eq:motion} for $N=12$ that we describe in the following corollary. We note that
the existence of the solutions described in items (ii) and (iii) had been already indicated by Souli\`ere \& Tokieda \cite[Section 5]{SouliereTokieda}.
\begin{corollary} 
\label{cor:tet-ico-osc}
Let $N=12$.
\begin{enumerate}
\item There exists a 1-parameter family of periodic solutions $v_h(t)$ of  the equations of motion \eqref{eq:motion},  
emanating from the icosahedral  equilibrium configurations. 
Along these solutions, each vortex travels around a small closed loop around a vertex of the icosahedron (see Fig. \ref{F:ico-osc}).

  \item There exists a 1-parameter family of periodic solutions $v_h(t)$ of  the equations of motion \eqref{eq:motion}
converging to the tetrahedral collision  described in  Theorem \ref{th:tetrahedral}.  
Along these solutions, three vortices travel along a closed
loop around each of the 4 vertices of a tetrahedron (see Fig. \ref{F:tet-coll-osc}).

\item There exists a 1-parameter family of periodic solutions $v_h(t)$ of  the equations of motion \eqref{eq:motion}
converging to the octahedral collisions  described in  Theorem \ref{th:tetrahedral}.  
Along these solutions,  a pair of vortices  travels along a small closed
loop  around each of the 6 vertices of an  octahedron (see Fig. \ref{F:oct-coll-osc}).
\end{enumerate}
Each of these families may be  parametrised by the energy $h$. In cases (ii) and (iii) we have 
 $h\to \infty$ as the solutions  approach collision, and  the period approaches zero in this limit.

For each solution described above, the distinct  closed loops traversed by the vortices, and the position the vortices within the loop at each instant, may be obtained from a single one by the action of $\mathbb{T}$.
\end{corollary}

\begin{figure}[ht]
\begin{subfigure}{.3\textwidth}
  \centering
  \includegraphics[width=.75\linewidth]{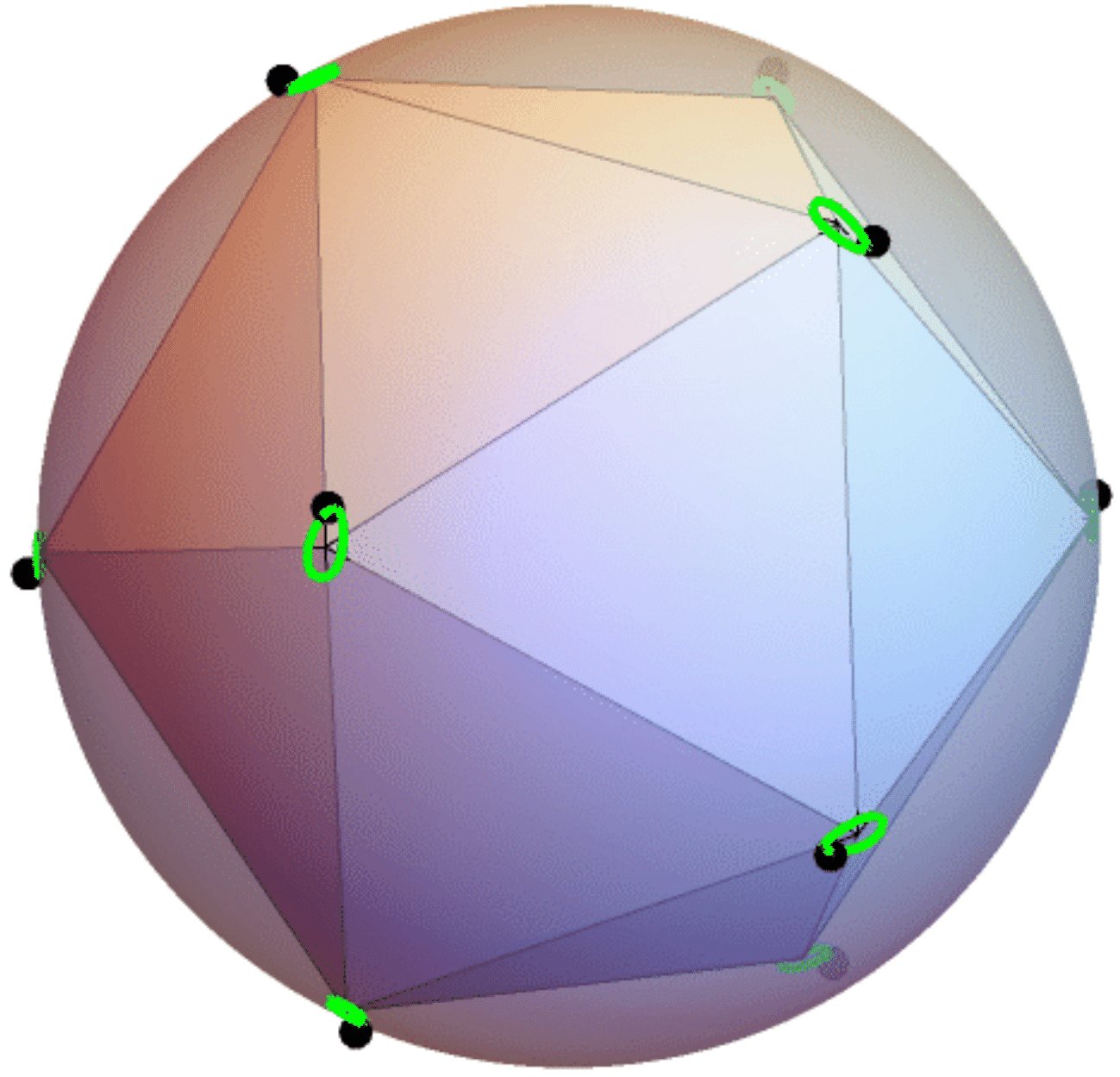}  
  \caption{Periodic solution near the icosahedron equilibrium.}
  \label{F:ico-osc}
\end{subfigure}
\quad
\begin{subfigure}{.3\textwidth}
 \centering
  \includegraphics[width=.75\linewidth]{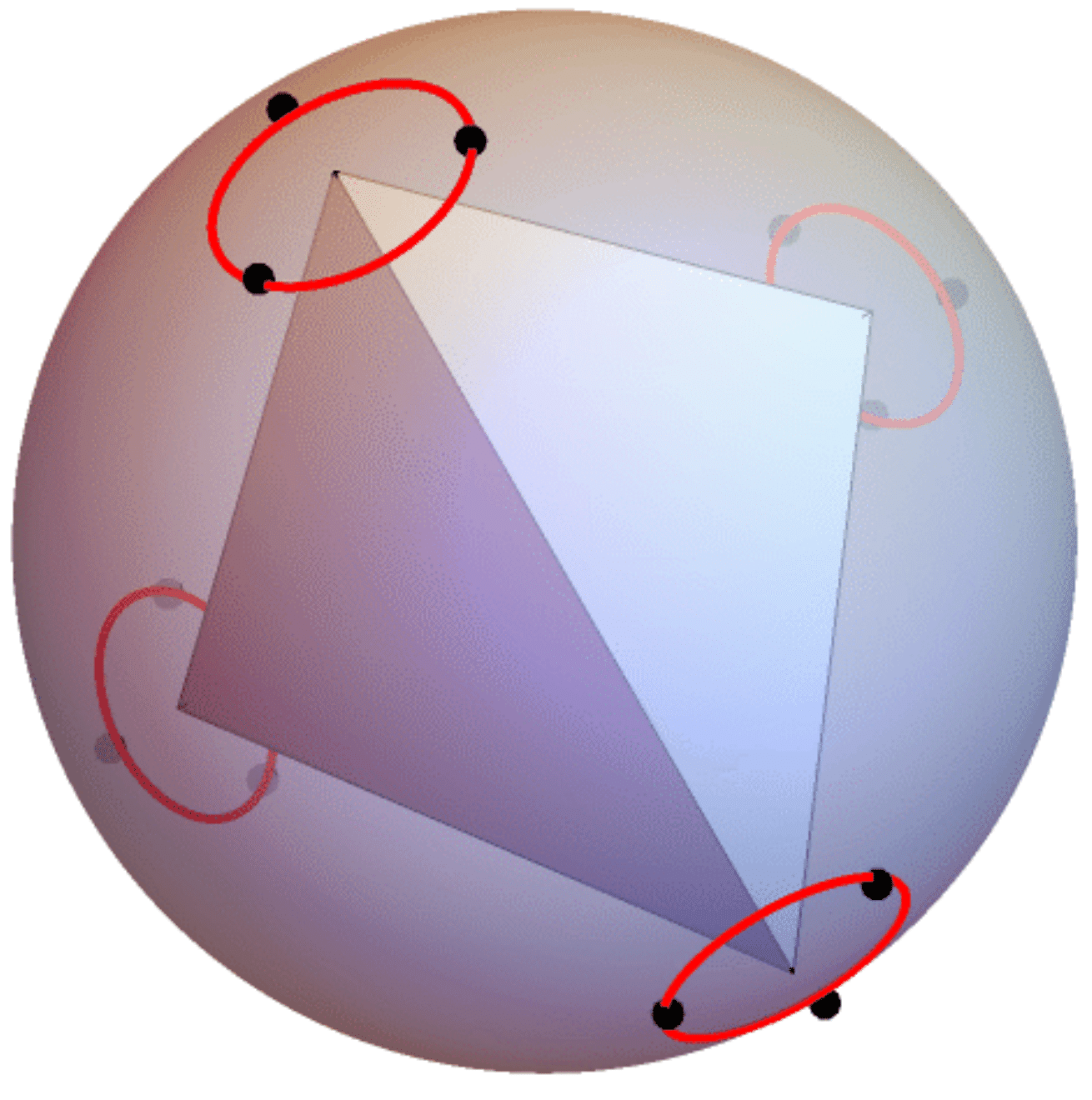}  
  \caption{Periodic solution near the
  tetrahedral collision.}
  \label{F:tet-coll-osc}
\end{subfigure}
\quad
\begin{subfigure}{.3\textwidth}
  \centering
  \includegraphics[width=.75\linewidth]{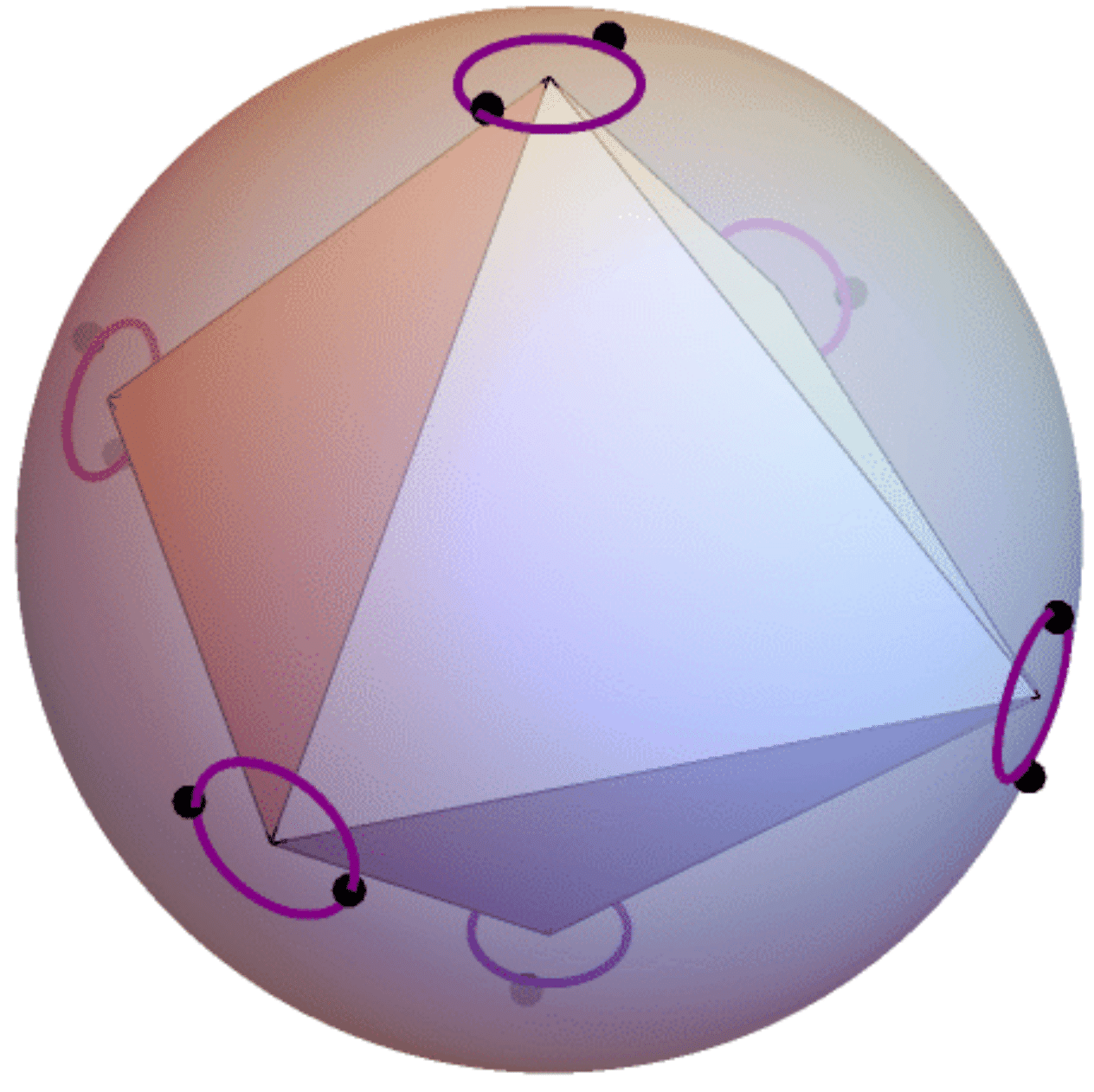}  
  \caption{Periodic solution near the
  octahedral collision.}
  \label{F:oct-coll-osc}
\end{subfigure}
\caption{Periodic solutions  described in Corollary \ref{cor:tet-ico-osc}. (The view angle is different from the one in Fig \ref{fig:n12-eq}.)}
\label{fig:ico-osc}
\end{figure}

 \begin{figure}[htp]
\centering
\includegraphics[width=0.4\textwidth]{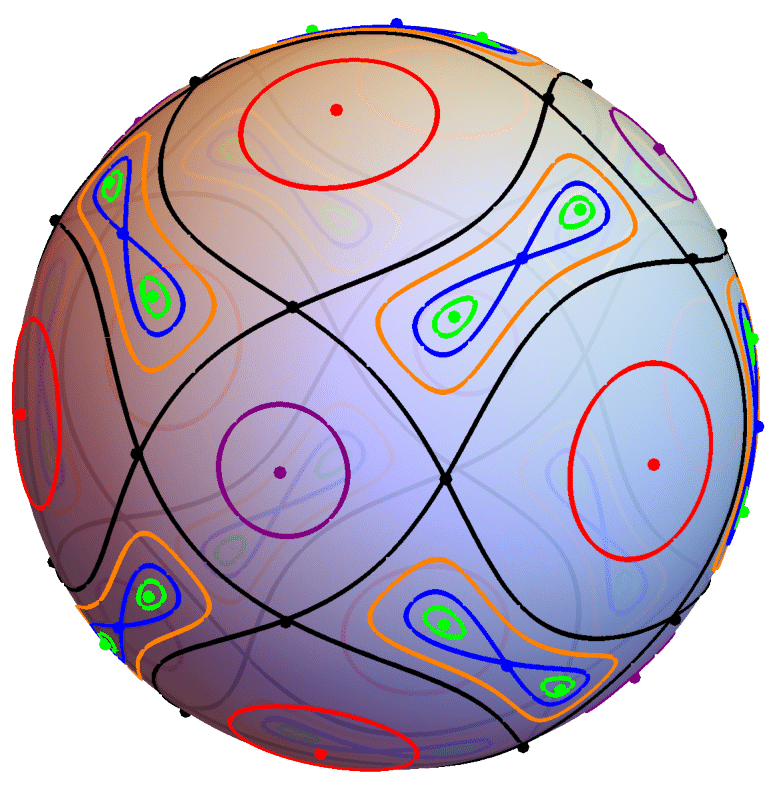}  
  \caption{The reduced phase space for $\mathbb{T}$-symmetric solutions of the 12-vortex problem.}
  \label{fig:Spheren12T}
\end{figure}

We emphasise that the family of periodic orbits emanating from the icosahedron configurations described in item (i)  above is different than the one obtained from item (i) of 
Corollary \ref{cor:dihedral-osc-poles} with $n=5$. 
Figure~\ref{fig:Spheren12T}  shows the phase space of the (regularised) reduced dynamics obtained numerically. 
The icosahedron equilibrium points  are indicated in green, 
the truncated tetrahedron equilibrium points in blue,  the cuboctahedron   equilibrium points   in black, tetrahedron collisions in red and octahedron  collisions   in purple. 
The same colour is used to indicate either periodic orbits near the stable equilibria or heteroclinic/homoclinic orbits emanating from the unstable equilibria.
There is also a family  of periodic orbits that do not approach an equilibria or a collision that we have indicated in orange.

\section{$\mathbb{T}$-symmetric solutions of  $N=20$ vortices  (8 vortices are fixed at  the vertices of a cube)   }
\label{sec:Tet-cube}

We again consider $K=\mathbb{T}$ but now we take $F=\mathcal{T}_1\cup \mathcal{T}_2$ (see Eq. \eqref{eq:tetrahedra}) so $N=20$.
The points in the set $F$ lie on the vertices of a cube so all solutions of \eqref{eq:motion} treated in this section will have  a fixed vortex at each vertex of this cube.
 We note that the set $F$ is  $\mathbb{T}$-invariant and, in view of \eqref{eq:F[T]},  is contained in 
 $\mathcal{F}[\mathbb{T}]$ so it satisfies both requirements in our setup.
We analyse the  reduced system  \eqref{eq:reduced-system}  in detail. Since the set $F$ is also  $\mathbb{O}$-invariant then, 
 in view of item (iii) of  Theorem \ref{th-main-symmetry} and Table \eqref{eq:table-normalizers},
 the reduced system is  $\mathbb{O}$-equivariant.

\subsection{ Classification of $\mathbb{T}$-symmetric equilibrium configurations of $N=20$ vortices}

The  analogous version of Theorem \ref{th:tetrahedral} on the classification and stability of the collision and non-collision equilibria 
of the reduced system  \eqref{eq:reduced-system} in this case is given next.

\begin{theorem}
\label{th:tetrahedral-cube}
Let  $K=\mathbb{T}$,  $F=\mathcal{T}_1\cup \mathcal{T}_2$ (see Eq. \eqref{eq:tetrahedra})  and $N=20$. 
The classification and stability of the  equilibrium points of the reduced system  \eqref{eq:reduced-system} is 
as follows.
\begin{enumerate}
\item The only non-collision equilibria of  \eqref{eq:reduced-system} are: 
\begin{enumerate}

\item The   \defn{dodecahedron  equilibrium configurations} occurring at all $24$ points obtained by permuting the entries and considering all  sign flips of 
\begin{equation*}
 \frac{1}{\sqrt{3}} \left  (\pm \phi , \pm  \phi^{-1}, 0 \right ),
\end{equation*} 
where $\phi=\frac{1+\sqrt{5}}{2}$ is the golden mean.
These are {\em stable} equilibria of  \eqref{eq:reduced-system}
which  correspond to equilibrium configurations of \eqref{eq:motion} where the vortices occupy  the 
vertices of an $S^2$-inscribed regular dodecahedron    (see Fig.\ref{F:icosahedron-w-cube}).

\item The   \defn{truncated tetrahedron - cube configurations} occurring at all  $24$ points obtained by permuting the entries and considering all  sign flips of 
\begin{equation*}
\left (\pm \hat \alpha , \pm \hat \alpha , \pm \sqrt{1-2\hat \alpha^2}\right ),
\end{equation*} 
where $\hat \alpha\approx 0.21228\dots $ is  characterised by the condition that $\hat \alpha^2$ is the unique zero of the polynomial
 $\hat p(\lambda)=57\lambda^3-29\lambda^2-21\lambda+1$ between $0$ and $1/2$.
These are {\em unstable} equilibria of  \eqref{eq:reduced-system}
which  correspond to equilibrium configurations of \eqref{eq:motion} where the vortices occupy  the 
vertices of the compound of an irregular $S^2$-inscribed truncated tetrahedron and a cube   (see Fig.\ref{F:trunc-tetrahedron-w-cube} and Remark \ref{rmk:dimensions-truc-tet-cube}).

\item The   \defn{cuboctahedron - cube configurations} occurring at all $12$ points obtained by permuting the entries and considering all  sign flips of 
\begin{equation*}
\left (\pm \frac{1}{\sqrt{2}}, \pm \frac{1}{\sqrt{2}}, 0 \right ).
\end{equation*}
These are {\em unstable} equilibria of  \eqref{eq:reduced-system}
which  correspond to equilibrium configurations of \eqref{eq:motion} where the vortices occupy  the 
vertices of the compound of a regular $S^2$-inscribed cuboctahedron and a cube   (see Fig.\ref{F:cuboctahedron-w-cube}). 
\end{enumerate}

\item The only  collision equilibria of (the regularisation of)  \eqref{eq:reduced-system} are:
\begin{enumerate}
\item The \defn{tetrahedral - cube collisions}  at the 8 points of  $\mathcal{T}_1\cup \mathcal{T}_2$. 
These correspond to collision configurations of \eqref{eq:motion} where the 20 vortices lie on the vertices of a cube. 
There are  four simultaneous quadruple collisions at the vertices of a tetrahedron  and the other four vortices lie at  each antipodal point (see
Fig.\ref{F:tetrahedron-coll-w-cube}).
\item The \defn{octahedral - cube collisions} at  the $6$ points $(\pm 1,0,0), ( 0,\pm1,0),  ( 0,0,\pm1)$. These correspond to collision configurations of \eqref{eq:motion} where the 20 vortices lie on the vertices 
of the compound of an octahedron and a cube and there are $6$  simultaneous binary collisions at the vertices of the octahedron   (see
Fig.\ref{F:octahedron-coll-w-cube}).  
 \end{enumerate}
All collision configurations are stable equilibria of  (the regularisation of)  \eqref{eq:reduced-system}.
\end{enumerate}
\end{theorem}

\begin{remark}
\label{rmk:dimensions-truc-tet-cube}
The 
 truncated tetrahedron on the compound of item (i)(b)  consists of  4 irregular hexagonal faces and 4 four equilateral triangular faces. 
The distance between vortices 
   forming an edge between adjacent hexagonal faces is $2\sqrt{2}\hat \alpha\approx0.600421$. The distance between vortices 
   forming an edge of an equilateral triangular  face is $\sqrt{2}(-\hat \alpha +\sqrt{1-2\hat \alpha^2})\approx 1.04877$.
\end{remark}

\begin{figure}[ht]
\begin{subfigure}{.30\textwidth}
  \centering
  \includegraphics[width=.7\linewidth]{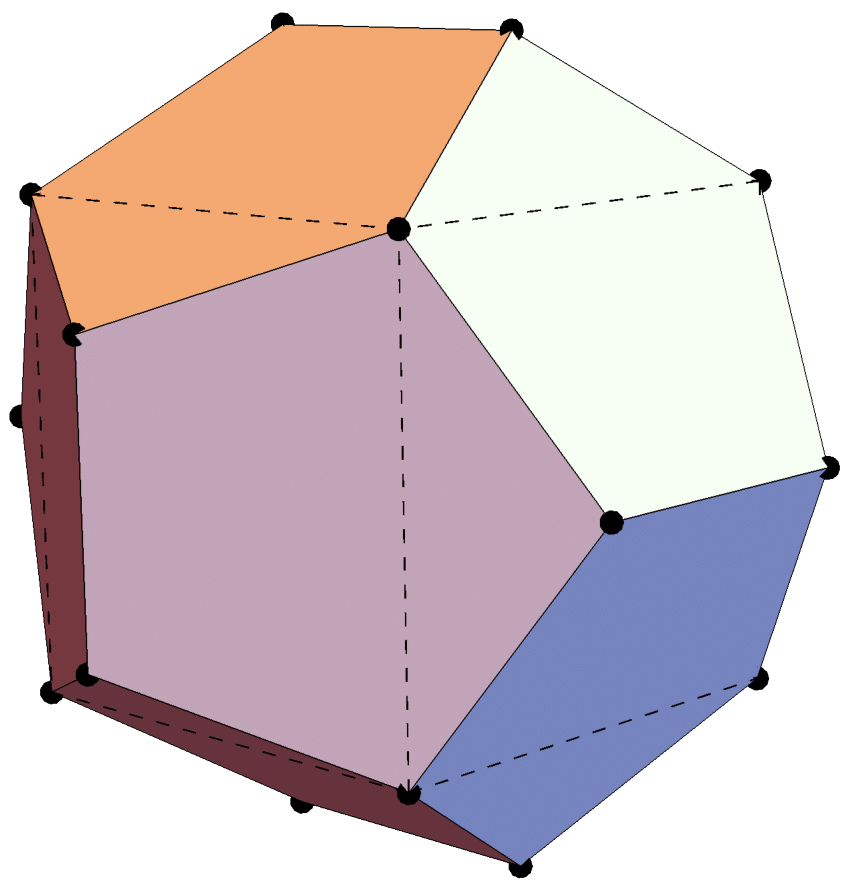}  
  \caption{Dodecahedron equilibrium. The distance between adjacent vertices  is $\sqrt{2-\frac{2 \sqrt{5}}{3}}$. }
  \label{F:icosahedron-w-cube}
\end{subfigure}
\quad
\begin{subfigure}{.30\textwidth}
  \centering
  \includegraphics[width=.7\linewidth]{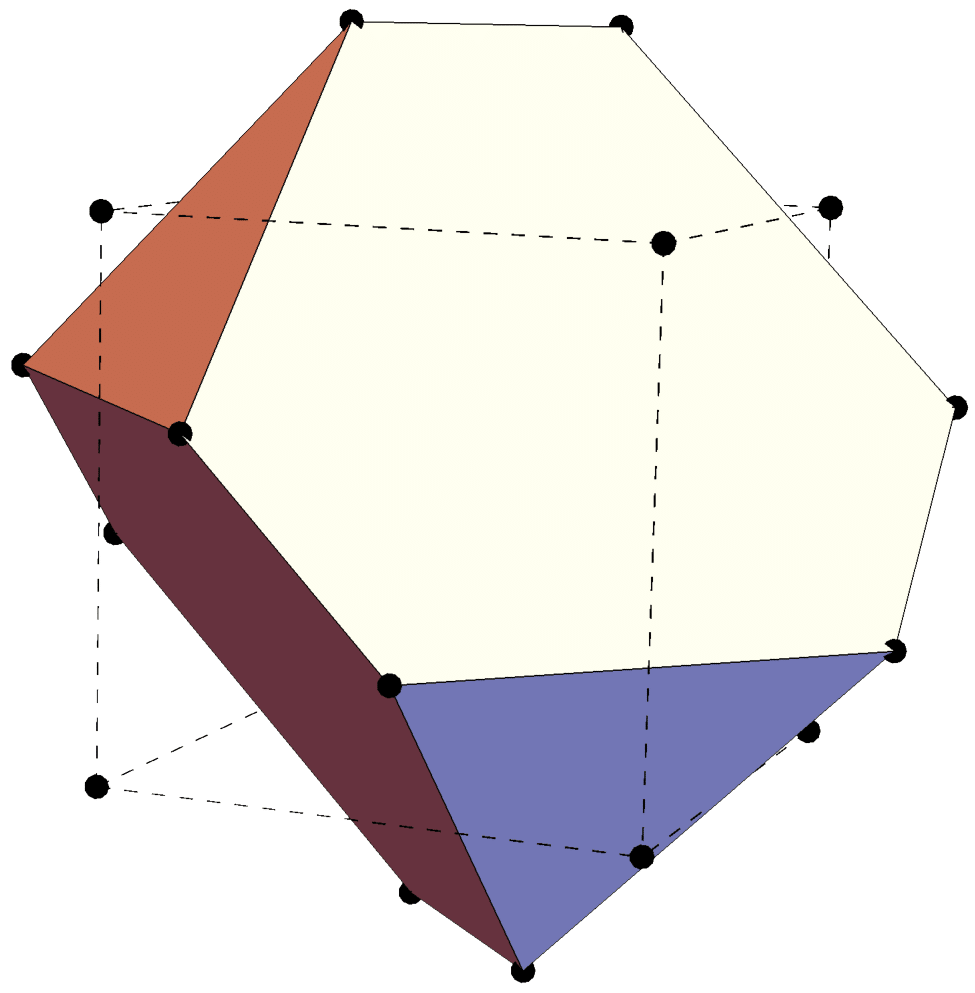}   
  \caption{Compound of a truncated tetrahedron and cube equilibrium configuration. See Remark \ref{rmk:dimensions-truc-tet-cube} for details on the dimensions of the
  truncated tetrahedron.}
  \label{F:trunc-tetrahedron-w-cube}
\end{subfigure}
\quad
\begin{subfigure}{.30\textwidth}
  \centering
  \includegraphics[width=.7\linewidth]{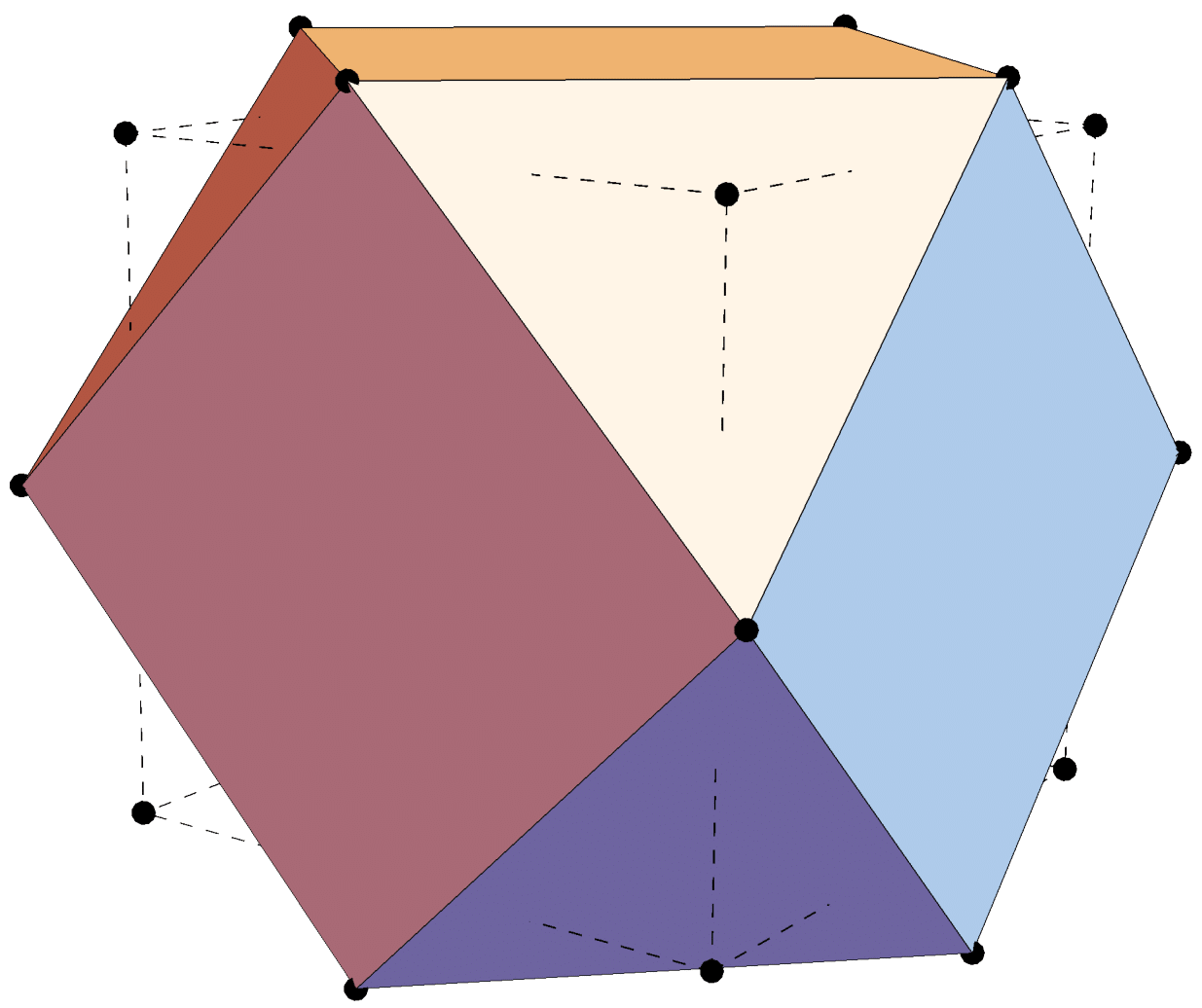} 
  \caption{Compound of a regular cuboctahedron and a  cube equilibrium configuration. The distance between adjacent vertices in the cuboctahedron is 1. }
  \label{F:cuboctahedron-w-cube}
\end{subfigure}
 \\
 \begin{subfigure}{.45\textwidth}
  \centering
  \includegraphics[width=.4\linewidth]{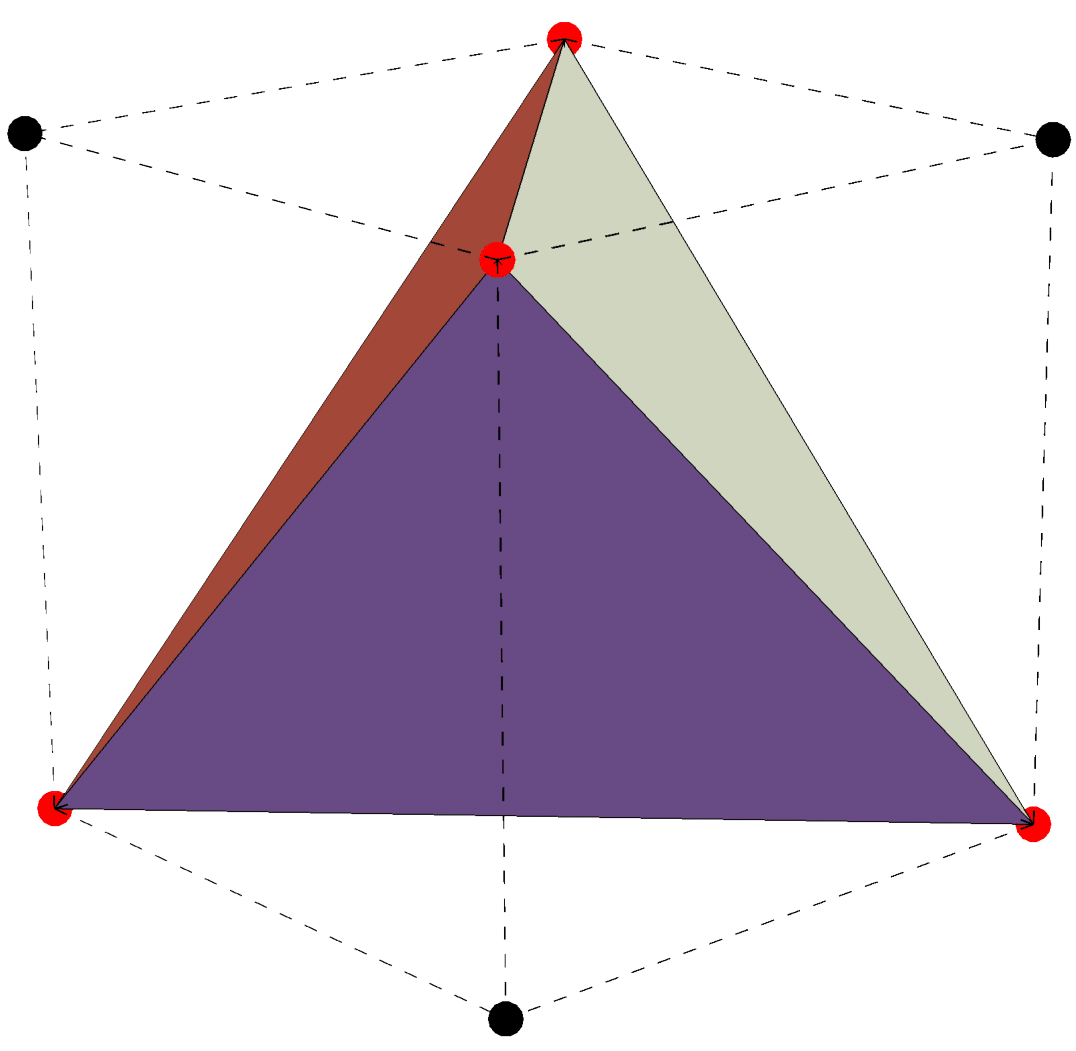}  
  \caption{Tetrahedron - cube collision. There is a quadruple collision  at each vertex of the tetrahedron and  no collisions at the other 4 vertices of the cube (marked in black).}
  \label{F:tetrahedron-coll-w-cube}
\end{subfigure} 
\quad 
\begin{subfigure}{.45\textwidth}
  \centering
  \includegraphics[width=.4\linewidth]{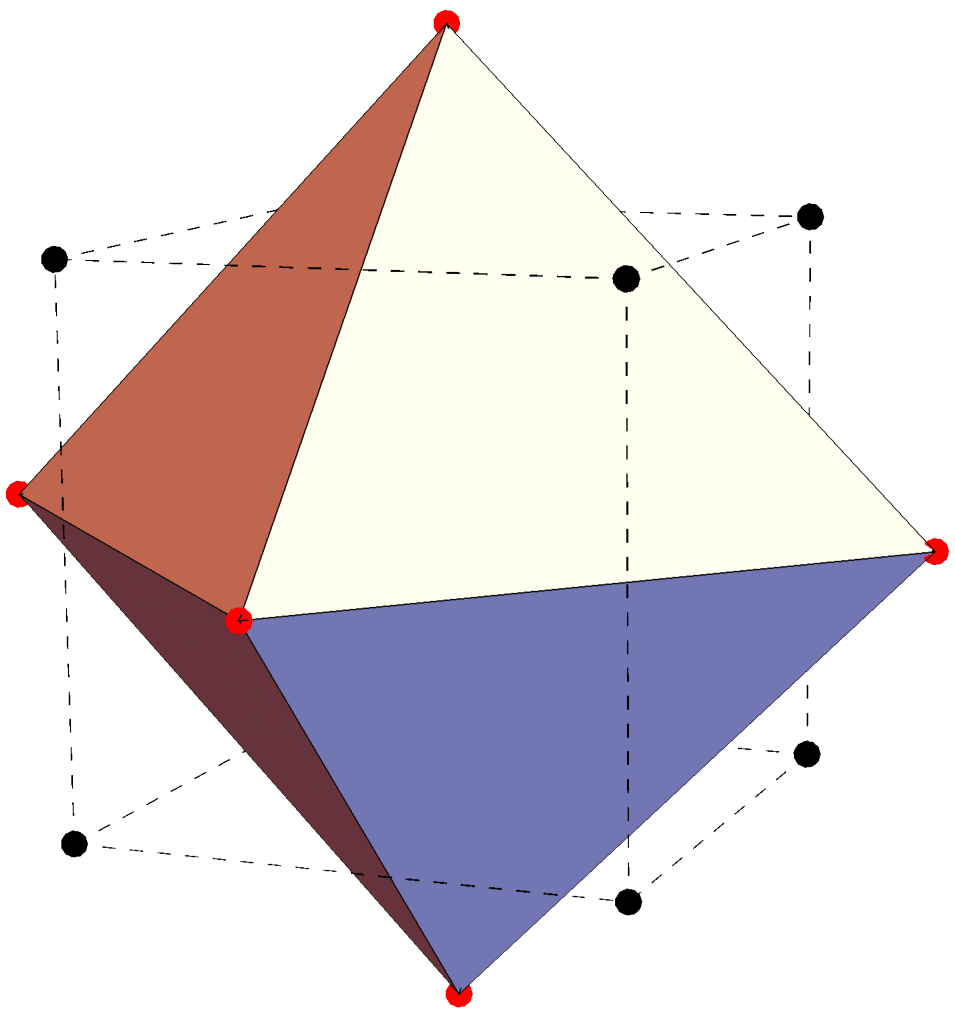}  
  \caption{Octahedron -cube collision. There is a  binary collision  at each vertex of the octahedron and  no collision at the 8 vertices of  the cube.}
  \label{F:octahedron-coll-w-cube}
\end{subfigure} 
\caption{Non-collision and collision equilibrium configurations described in  Theorem \ref{th:tetrahedral-cube}. The dashed lines connect the 
elements in $F=\mathcal{T}_1\cup \mathcal{T}_2$ that lie on the vertices of a cube.}
\label{fig:n20-eq}
\end{figure}

\begin{proof}
The conclusions of item (ii) about the collision equilibria follow from the description of $\mathcal{F}[\mathbb{T}]$ in   \eqref{eq:F[T]}  and  Proposition \ref{prop:collision-description} in analogy  with the proof of 
Theorem  \ref{th:tetrahedral}.

In order to analyse the non-collision equilibria, let $f_{13}, \dots, f_{20}$ denote the
points in $F=\mathcal{T}_1\cup \mathcal{T}_2$.  A direct calculation shows that  for $u=(x,y,z)\in S^2$ we have
\begin{equation*}
 \prod_{j=m+1}^N \left \vert u-f_j \right \vert^{2} = \prod_{j=13}^{20} \left \vert u-f_j \right \vert^{2}=\left ( \frac{8}{3} \right )^4 (1-xy-xz-yz) (1+xy+xz-yz)  (1-xy+xz+yz)(1+xy-xz+yz).
\end{equation*}
Therefore, in view of  \eqref{eq:regul-red-Ham}  and the proof of Theorem   \ref{th:tetrahedral}  we find $\tilde{h}_{(\mathbb{T},F)}(x,y,z)=2^{84}\left  ( \frac{8}{3} \right )^{48} \hat a(x,y,z)^6$, where $\hat a:S^2\to \R$ is given by
 \begin{equation*}
 \begin{split}
\hat a(x,y,z)&= (x^2+y^2) (y^2+z^2)   (x^2+z^2) (1-xy-xz-yz)^4 (1+xy+xz-yz)^4  \\ &  \qquad \cdot (1-xy+xz+yz)^{4}(1+xy-xz+yz)^{4}.
\end{split}
\end{equation*}
The proof proceeds by finding the critical points of $\hat a$ on the fundamental region $\mathcal{R}$ described in the proof of Theorem  \ref{th:tetrahedral} and is analogous to it. We omit the details.
\end{proof}

\subsection{Dynamics of $\mathbb{T}$-symmetric configurations of $N=20$ vortices}

By combining Theorem \ref{th:tetrahedral-cube} with  Corollary \ref{cor:periodic-orbits-general} 
we deduce the existence of  three families 
of periodic orbits of the equations of motion \eqref{eq:motion} for $N=20$ that we describe in the following:
\begin{corollary} 
\label{cor:tet-ico-osc-cube}
Let $N=20$.
\begin{enumerate}
\item There exists a 1-parameter family of periodic solutions $v_h(t)$ of  the equations of motion \eqref{eq:motion},  
emanating from the dodecahedral  equilibrium configurations. 
Along these solutions, 8  vortices are fixed at the vertices of a cube inscribed in the dodecahedron and  the remaining 12 vortices 
  travel around a small closed loop around the remaining  vertices  of the dodecahedron (see Fig. \ref{F:ico-osc}).

  \item There exists a 1-parameter family of periodic solutions $v_h(t)$ of  the equations of motion \eqref{eq:motion}
converging to the tetrahedral-cube collision  described in  Theorem \ref{th:tetrahedral-cube}.  
Along these solutions, 8  vortices are fixed on the vertices of a cube and there are 4 triples of vortices that 
 travel along a closed
loop around each of the 4 vertices of a tetrahedron inscribed in the cube (see Fig. \ref{F:tet-coll-osc}).

\item There exists a 1-parameter family of periodic solutions $v_h(t)$ of  the equations of motion \eqref{eq:motion}
converging to the octahedral collisions  described in  Theorem \ref{th:tetrahedral}.  
Along these solutions,  a pair of vortices  travels along a small closed
loop  around each of the 6 vertices of an  octahedron and the remaining 8 vortices are fixed at the vertices of the dual cube (see Fig. \ref{F:oct-coll-osc}) .
\end{enumerate}
Each of these families may be  parametrised by the energy $h$. In cases (ii) and (iii) we have 
 $h\to \infty$ as the solutions  approach collision, and  the period approaches zero in this limit.

For each solution described above, the distinct  closed loops traversed by the vortices, and the position the vortices within the loop at each instant, may be obtained from a single one by the action of $\mathbb{T}$.
\end{corollary}

\begin{figure}[ht]
\begin{subfigure}{.3\textwidth}
  \centering
  \includegraphics[width=.75\linewidth]{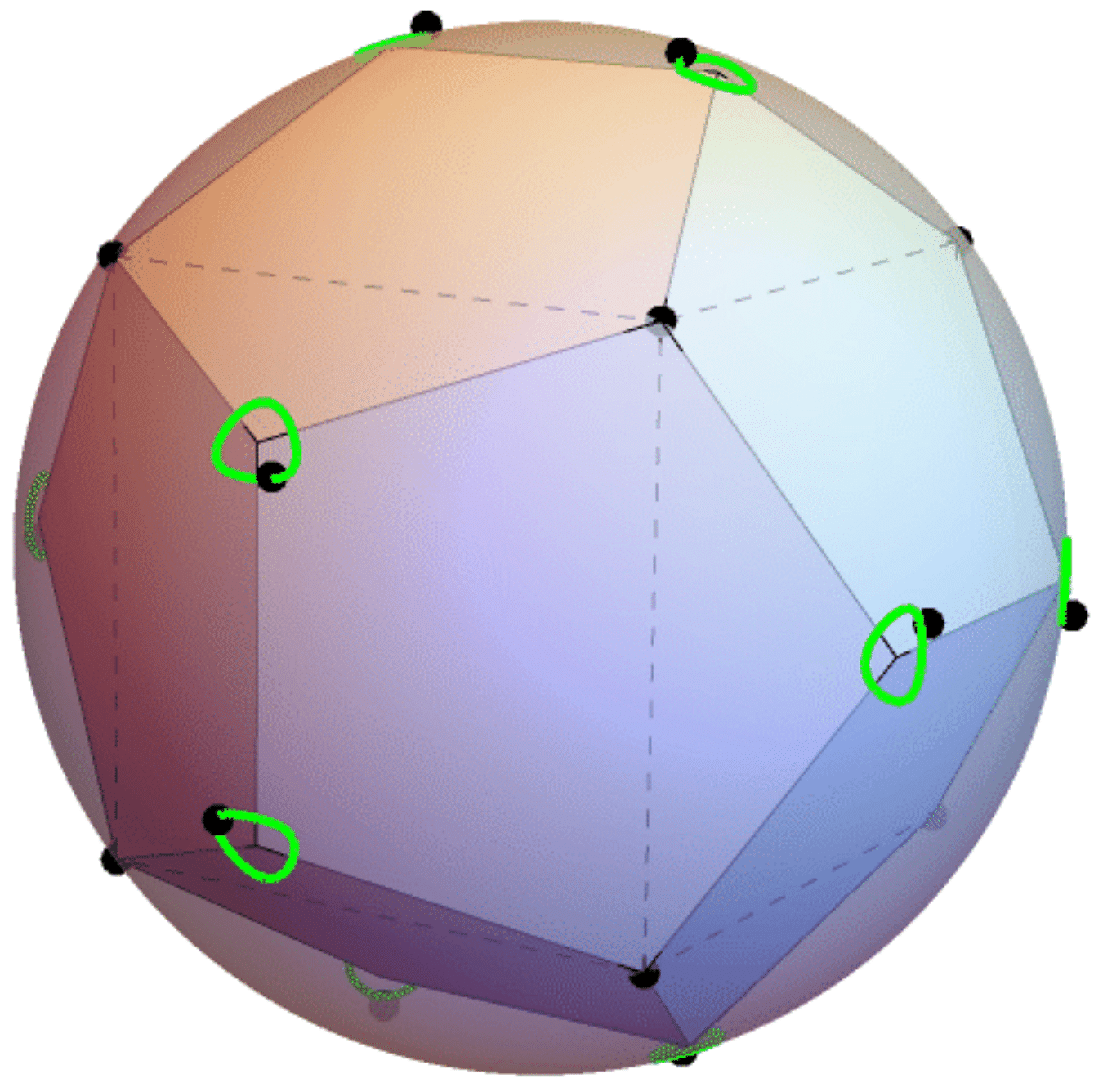}  
  \caption{Periodic solution near the dodecahedron equilibrium.}
  \label{F:ico-osc}
\end{subfigure}
\quad
\begin{subfigure}{.3\textwidth}
 \centering
  \includegraphics[width=.75\linewidth]{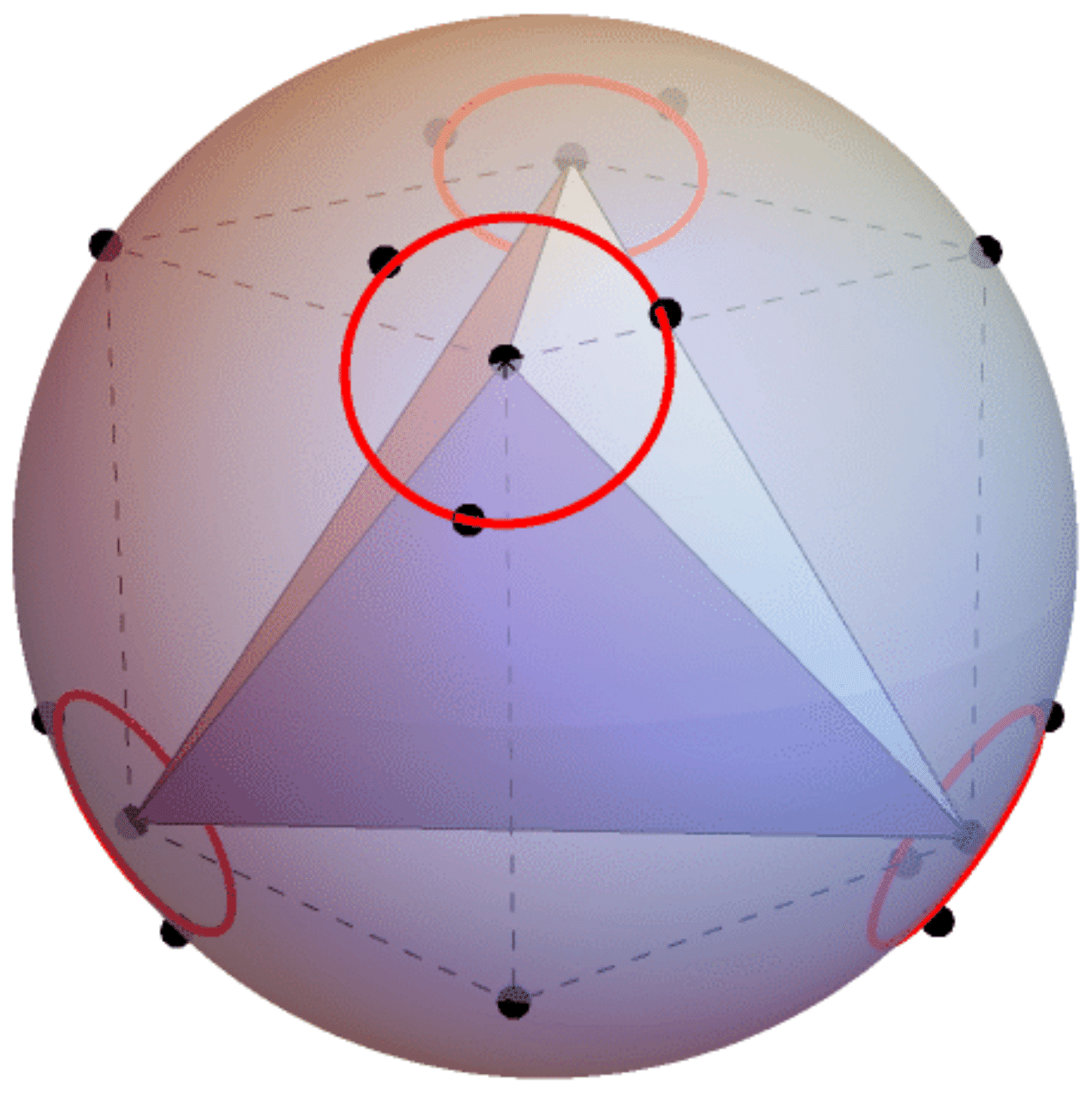}  
  \caption{Periodic solution near the
  tetrahedral-cube collision.}
  \label{F:tet-coll-osc}
\end{subfigure}
\quad
\begin{subfigure}{.3\textwidth}
  \centering
  \includegraphics[width=.75\linewidth]{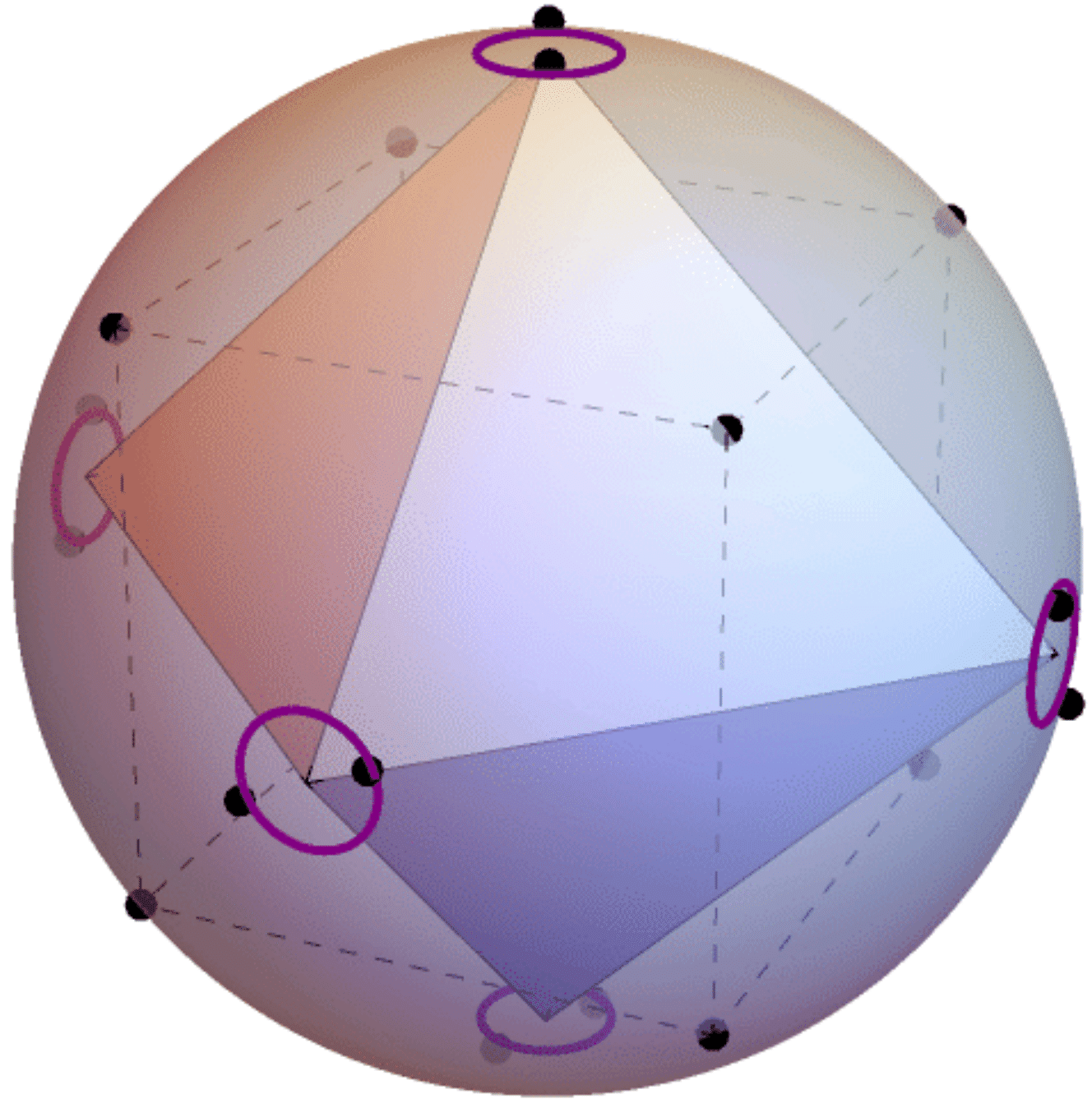}  
  \caption{Periodic solution near the
  octahedral-cube collision.}
  \label{F:oct-coll-osc}
\end{subfigure}
\caption{Periodic solutions  described in Corollary \ref{cor:tet-ico-osc-cube}. }
\label{fig:ico-osc-cube}
\end{figure}

 \begin{figure}[htp]
\centering
\includegraphics[width=0.4\textwidth]{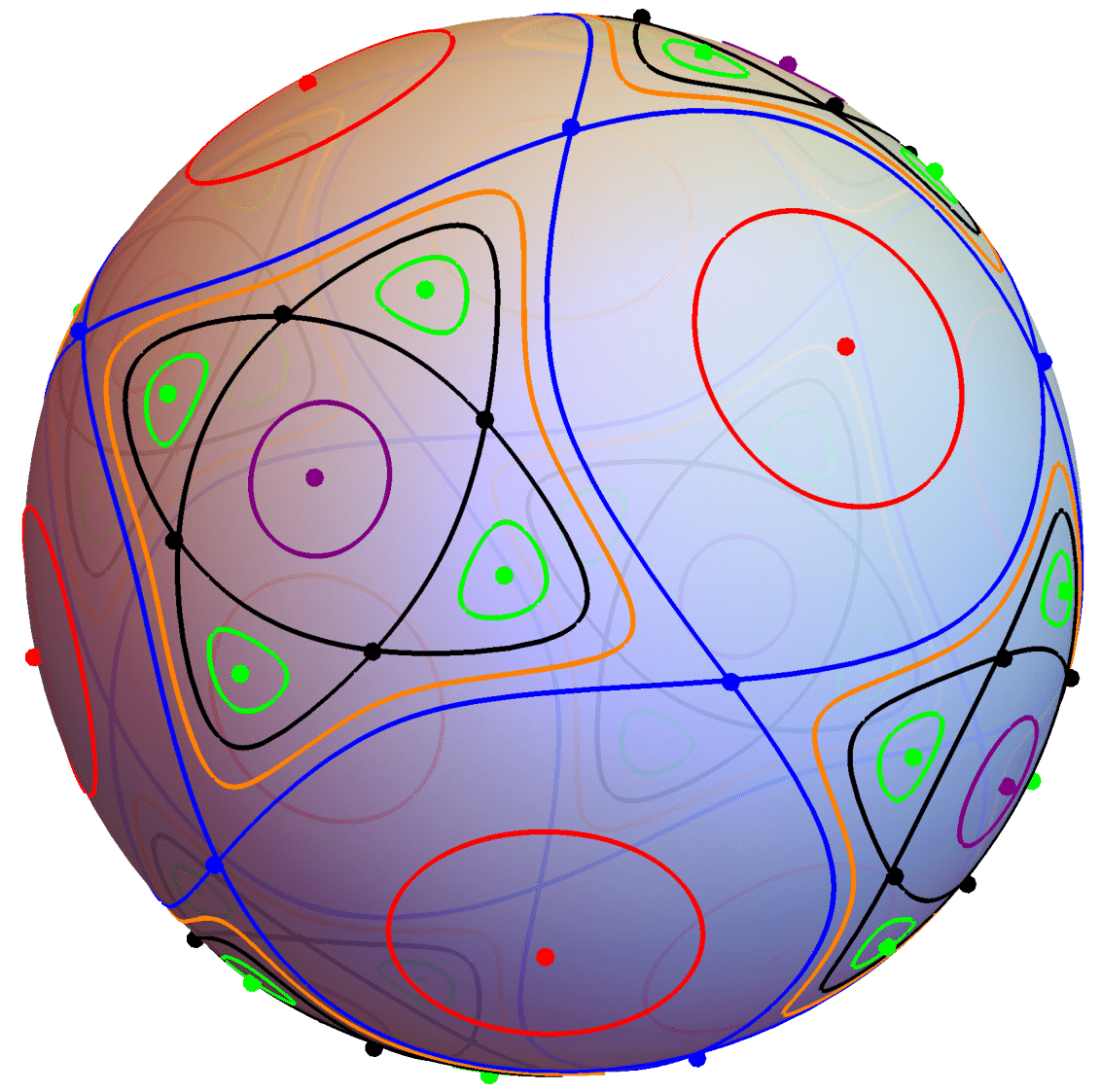}  
  \caption{The reduced phase space for $\mathbb{T}$-symmetric solutions of the 20-vortex problem.}
  \label{fig:Sphere-n20-T}
\end{figure}

Figure~\ref{fig:Sphere-n20-T}  shows the phase space of the (regularised) reduced dynamics obtained numerically. 
The colour code is similar to the one used in the previous section. The  dodecahedron equilibrium points  are indicated in green,
the truncated tetrahedron--cube equilibrium points in blue,  the cuboctahedron--cube   equilibrium points   in black, tetrahedral--cube collisions in red and octahedral--cube  collisions   in purple. 
As usual, we use the  same colour to indicate either periodic orbits near the stable equilibria or heteroclinic orbits emanating from the unstable equilibria and we indicate 
 a family  of periodic orbits that do not approach an equilibria  in orange.

\section*{Acknowledgements}
CGA and LGN respectively acknowledge support for their research from the Programs UNAM-PAPIIT-IN115019 and UNAM-PAPIIT-IN115820.

\vskip 1cm

\noindent CGA and LGN: Departamento de Matem\'aticas y Mec\'anica, 
IIMAS-UNAM.
Apdo. Postal 20-126, Col. San \'Angel,
Mexico City, 01000,  Mexico. cga@mym.iimas.unam.mx. luis@mym.iimas.unam.mx.
 
\end{document}